\documentclass[reqno]{amsart}
\usepackage{amsmath,amssymb,amsxtra,anysize,xcolor,tikz,hyperref,comment,tikz-cd}
\usepackage{graphicx}
\usepackage{enumitem,subcaption}
\usepackage{comment, rotating}

\newtheorem{theorem}{Theorem}[section]

\newtheorem{corollary}[theorem]{Corollary} 
\newtheorem{lemma}[theorem]{Lemma} 
 
\newtheorem{example}[theorem]{Example}

\newtheorem{remark}[theorem]{Remark}  
\newtheorem{definition}[theorem]{Definition}

\newcommand{\C}{\mathbb{C}}
\newcommand{\N}{\mathbb{N}}
 
\newcommand{\Id}{\mathrm{Id}}
\newcommand{\Mat}{\mathrm{Mat}}
\newcommand{\Gr}{\mathrm{Gr}}

\newcommand{\fan}{\mathrm{U_{fan}}}

\newcommand{\Br}{\mathrm{Br}}

\usetikzlibrary{decorations.markings}

\title{On the Cohomology of Two Stranded Braid Varieties}

\author{Tonie Scroggin}
\thanks{I wish to highly thank Eugene Gorsky,
Roger Casals, Jose Simental-Rodriguez,
Lauren Williams, Misha Mazin, Pavel Galashin and
Catharina Stroppel for
helpful conversations related
to this paper. I also would like to express my sincerest gratitude to the anonymous referee of my paper, their comments and suggested edits have been highly welcome.
In no particular order,
I wish to thank Ian Sullivan,
Trevor Oliveira-Smith, Shanon
Rubin, Matthew Corbelli,
Sari Ogami, Alex Black,
Milo Bechthoff 
Weising, Alexander Simons,
Greg DePaul, Brian Harvey, Laura Starkston, Becca Thomases, Tina Denena and most importantly my family. I would love to thank everyone that has greatly helped me, however I believe I would be required to write an entirely new paper simply for this purpose. This work was partially supported by NSF grant DMS-2302305.}

\begin{document}

\maketitle
\begin{abstract}
    We compute the cohomologies of two strand braid varieties using the two-form present in cluster structures. We confirm these results with proof using Alexander and Poincar\'e duality. Further, we consider products of braid varieties and their interactions with the cohomologies.  
\end{abstract}

\tableofcontents

\section{Introduction}

In this paper, we will study the relationship between braid varieties and their associated cluster structure in order to compute their cohomologies. Braid varieties are a class of affine algebraic varieties associated to positive braids \cite{CGGS1,CGGS2,CGGLSS}. Braid varieties are closely related to augmentation varieties of Legendrian links \cite{CW} and also include interesting geometric spaces such as positroid varieties, open Richardson varieties and double Bruhat cells. 

    To define the braid variety, we use the  braid group on $n$ strands, 
      $$
      \Br_n=\langle \sigma_1,\ldots,\sigma_{n-1}\;\textbf{:}\;\sigma_i\sigma_{i+1}\sigma_i=\sigma_{i+1}\sigma_i\sigma_{i+1},\,\sigma_i\sigma_j=\sigma_j\sigma_i\text{ if }|i-j|>1\rangle
      $$
      and restrict to positive crossings $\sigma_i$ between the $i$ and $i+1$ strand. At each crossing $\sigma_i$ of the positive braid we assign a complex variable $z$, see Figure \ref{fig:braid variables}, and a matrix 
$$B_i(z):=\begin{pmatrix}1&\cdots&&&\ldots&0\\ \vdots&\ddots&&&&\vdots\\ 0&\cdots&z&-1&\cdots&0\\ 0&\cdots&1&0&\cdots&0\\ \vdots&&&&\ddots&\vdots\\ 0&\cdots&&&\cdots&1\end{pmatrix}
$$
where the $2\times2$ embedded matrix is at the $i$ and $i+1$ row and column. Let $\beta=\sigma_{i_1}\dots\sigma_{i_k}\in\Br_n^+$ be a positive braid word, then the braid variety $X(\beta)$ is defined by 
$$X(\beta):=\left\{(z_1,\ldots,z_k) :\:
        \begin{pmatrix}
        0 & \dots & 1\\
        \vdots &\reflectbox{$\ddots$} &\vdots\\
        1 &\dots & 0\\
        \end{pmatrix}
        B_{i_1}(z_1)\cdots B_{i_k}(z_k)\ \text{is upper-triangular}.\right\}
        $$
There is an alternative geometric definition of braid varieties using certain configurations of flags, however, it will not be relevant to this paper, see \cite{CGGLSS,GLSBS} for further information. As indicated by the title, this paper will solely focus on two-strand braid varieties and we denote such a braid with $k$ crossings as $\sigma^k$.
    
\begin{theorem}[Hughes\cite{H}, Chantraine-Ng-Sivek\cite{CNS}] The braid variety $X(\sigma^k)$ is defined in $\C^k$ by the equation $F_k(z_1,\ldots,z_k)=0$ where $F_k$ is given by the recursion 
    $$B_\beta(z_1,\ldots,z_k)=\begin{pmatrix}
        F_k(z_1,\dots,z_k)&-F_{k-1}(z_1,\dots,z_{k-1})\\F_{k-1}(z_2,\dots,z_k)&-F_{k-2}(z_2,\dots,z_{k-1})
    \end{pmatrix}$$
    where 
    \begin{equation}
    \label{eq: def F}
    F_k(z_i,\dots,z_{i+k})=z_kF_{k-1}(z_i,\dots,z_{i+k-1})-F_{k-2}(z_i,\dots,z_{i+k-2})
    \end{equation}
    with initial values $F_1(z_i)=z_i$, $F_0\equiv1$ and $F_{-1}\equiv0$.
Moreover, if $F_k(z_1,\dots,z_k)=0$, then $F_{k-1}(z_1,\dots,z_{k-1})\neq0$ and $$z_k=\dfrac{F_{k-2}(z_1,\dots,z_{k-2})}{F_{k-1}(z_1,\dots,z_{k-1})}.$$
\begin{remark}
    As a corollary, we have $X(\sigma^k)\cong \{(z_1,\dots,z_{k-1})\in\C^{k-1}:F_{k-1}(z_1,\dots,z_{k-1})\neq0\}$ as algebraic varieties. In particular, $X(\beta)$ is smooth of complex dimension $k-1$.
\end{remark}
\end{theorem}

We construct an explicit isomorphism between the two-strand braid variety and positroid varieties in the Grassmannian $\Gr(2,k+1)$. 
By Scott \cite{S}, these admit a cluster structure of type $A$. 

We define the open positroid variety as the set of elements in the Grassmannian such that there is a representative $k\times n$-matrix such that all cyclically consecutive $k\times k$ minors don't vanish, i.e., $$\Delta_{i,\dots,i+k-1}=\det(v_i,\dots,v_{i+k-1})\ne0.$$ This condition does not depend on the representative, so the positroid is well-defined. 

\begin{theorem}
\label{thm: positroid}
Let $\Pi_{2,k+1}^{\circ}$ be the open positroid variety defined by the condition that all consecutive $2\times 2$ minors do not vanish, i.e., $\Delta_{i,i+1}=\det(v_i,v_{i+1})\ne 0$, and $\Pi_{2,k+1}^{\circ,1}$ be the subset of the open positroid variety such that for all $\Delta_{i,i+1}=1$ for all $1\le i\le k$. Then
\begin{enumerate}[label=\alph*),font=\itshape]
    \item $\Pi_{2,k+1}^{\circ,1}$ \textit{is isomorphic to }$X(\sigma^k)$.
    \item $\Pi_{2,k+1}^{\circ}$ \textit{is isomorphic to }$X(\sigma^k)\times (\C^*)^k$.
\end{enumerate}
\end{theorem}

One of the main motivations for studying the homologies of braid varieties is their relation to the Khovanov-Rozansky homology of the corresponding link. 
    
    \begin{theorem}[Trinh\cite{T}]
    \label{thm: trinh intro}
    For all $r$-strand braids $\beta\in Br_W^+$ we have $$\mathrm{HHH}^{r,r+j,k}(\beta\Delta)^{\vee}\simeq \mathrm{gr}^w_{j+2(r-N)}\mathrm{H}^{!,G}_{-(j+k+2(r-N))}(X(\beta)).
    $$
Equivalently, by Gorsky-Hogancamp-Mellit-Nakagane \cite{GHMN}, $\mathrm{H}^*(X(\beta))\simeq \mathrm{HHH}^{0,*,*}(\beta\Delta^{-1})^{\vee}$ where $\Delta$ is the half-twist (aka longest word). 
Here $\mathrm{gr}^w$ denotes the associated graded with respect to the weight filtration in cohomology.
\end{theorem}

    On two strands this equivalence simplifies to $$H^*(X(\sigma^k))\simeq \mathrm{HHH}^{0,*,*}(\sigma^{k-1})$$
    where the braid $\sigma^{k-1}$ closes up to the torus link $T(2,k-1)$.    
    
    The cohomology of $X(\sigma^k)$ was computed by Lam and Speyer in \cite{LS} using cluster algebra machinery. Here we give a simpler and more direct proof.

    First, we describe the cohomology of $X(\sigma^k)$  as a vector space.

\begin{theorem}Let $\beta=\sigma^n$, then the cohomology of the two-strand braid variety is given by: $$H^i(X(\beta);\C)=\begin{cases}
        \C &\text{for }0\leq i\leq n-1\\ 0 &\text{otherwise}.
    \end{cases}$$
\end{theorem}

Next, we identify the ring structure in cohomology using algebraic forms with (algebraic) de Rham cohomology. For this, we introduce in Section \ref{sec: forms} a regular one-form $\alpha$ and a regular two-form $\omega$ on $X(\beta)$. We write explicit formulas for $\alpha$ and $\omega$ in terms of both $z_i$ and the independent Pl\"ucker coordinates in Theorems \ref{lem: alpha form} and \ref{lem: Omega formula}.

\begin{theorem}The 1-form $\alpha$ and 2-form $\omega$ generate $H^*(X(\sigma^k))$ as a $\C$-algebra, modulo the following relations:\\
1) If $k$ is even, the only relation is $\omega^{\frac{k}{2}}=0$. The basis in cohomology is given by:
\begin{equation*}
1,\alpha,\omega,\alpha\omega,\ldots,\omega^{\frac{k}{2}-1},\alpha\omega^{\frac{k}{2}-1}.
\end{equation*}
2) If $k$ is odd, the relations are $\alpha\omega^{\frac{k-1}{2}}=\omega^{\frac{k+1}{2}}=0$. The basis in cohomology is given by: 
\begin{equation*}
1,\alpha,\omega,\alpha\omega,\ldots,\alpha\omega^{\frac{k-3}{2}},\omega^{\frac{k-1}{2}}.
\end{equation*}
\end{theorem}

Next, we study the relation between different two-strand braid varieties. We show that the product of two braid varieties $X(\sigma^a)\times X(\sigma^b)$ can be embedded as an open subset into a larger braid variety $X(\sigma^{a+b-1})$. All such embeddings are parametrized by the diagonals in the $(a+b)$-gon (we refer to them as to diagonal cuts), and we write them explicitly in coordinates.  

\begin{theorem}Performing one diagonal cut on $P$ along $D_{ij}$ defines an injective map \[      \Phi_{ij}:X(\sigma^{a})\times X(\sigma^{b})\longrightarrow X(\sigma^{a+b-1}).\] By Theorem \ref{thm: positroid} we identify $X(\sigma^{a+b-1})$ with $\Pi_{2,a+b}^{o,1}$ and the image of the map is the open subset $\{\Delta_{ij}\neq 0\}$ in $\Pi_{2,a+b}^{o,1}$.
\end{theorem}

We can study the corresponding maps in cohomology of braid varieties.

\begin{theorem}
We have
\begin{equation}
\label{eq: pullback intro}
\Phi_{ij}^*\alpha=\alpha_2+(-1)^{k-j} \alpha_1,\quad 
\Phi_{ij}^*\omega=\omega_1+\omega_2+(-1)^{k-j}\alpha_1\wedge \alpha_2.
\end{equation}
The pullback map in cohomology
$$
\Phi_{ij}^*:H^*(X(\sigma^{k}))\to H^*(X(\sigma^{j-i}))\otimes H^*(X(\sigma^{k-j+i+1}))
$$
is injective.
and can be described by \eqref{eq: pullback intro}.
\end{theorem}

\begin{theorem}
    The map $\Phi_{ij}$ defines a  quasi-equivalence of cluster varietes $\{\Delta_{ij}\neq 0\}\subset X(\sigma^{a+b-1})$ and $X(\sigma^{a})\times X(\sigma^{b})$. The latter has a cluster structure obtained by freezing $\Delta_{ij}$ in the cluster structure from $X(\sigma^{a+b-1})$.
\end{theorem}

\begin{remark}
By Gorsky and Hogancamp \textcolor{red}{\cite{GH}}, on the level of knot homology, the maps 
$X(\sigma^a)\times X(\sigma^b)\to X(\sigma^{a+b-1})$ correspond via Theorem \ref{thm: trinh intro} to the maps
$$
\mathrm{HHH}(\sigma^{a-1})\otimes \mathrm{HHH}(\sigma^{b-1})\to \mathrm{HHH}(\sigma^{a+b-2})
$$
induced by the cobordism between the closures of the corresponding braids $T(2,a-1)\sqcup T(2,b-1)$ and $T(2,a+b-2)$. 
\end{remark}

\begin{figure}
    \centering
    \tikzset{every picture/.style={line width=0.75pt}} 

\begin{tikzpicture}[x=0.75pt,y=0.75pt,yscale=-1,xscale=1]

\draw   (143.57,81.23) -- (131.06,118.74) -- (98.3,141.92) -- (57.81,141.92) -- (25.05,118.74) -- (12.54,81.23) -- (25.05,43.73) -- (57.81,20.55) -- (98.3,20.55) -- (131.06,43.73) -- cycle ;
\draw  [dash pattern={on 4.5pt off 4.5pt}]  (12.54,81.23) -- (131.06,43.73) ;
\draw  [dash pattern={on 4.5pt off 4.5pt}]  (27.07,118.49) -- (131.06,118.74) ;
\draw [fill={rgb, 255:red, 183; green, 134; blue, 225 }  ,fill opacity=1 ]   (179.45,69.86) -- (297.22,31.9) ;
\draw [fill={rgb, 255:red, 183; green, 134; blue, 225 }  ,fill opacity=1 ]   (179.7,69.86) -- (193.41,32.89) ;
\draw [fill={rgb, 255:red, 183; green, 134; blue, 225 }  ,fill opacity=1 ]   (193.41,32.89) -- (223.29,11.23) ;
\draw [fill={rgb, 255:red, 183; green, 134; blue, 225 }  ,fill opacity=1 ]   (223.29,11.23) -- (263.29,11.23) ;
\draw [fill={rgb, 255:red, 183; green, 134; blue, 225 }  ,fill opacity=1 ]   (263.29,11.23) -- (297.48,31.9) ;
\draw    (349.68,71.91) -- (362.34,111.03) ;
\draw    (349.68,71.91) -- (467.93,34.26) ;
\draw    (467.93,34.26) -- (480.59,71.66) ;
\draw    (480.33,71.66) -- (466.41,110.79) ;
\draw    (362.09,111.03) -- (466.41,110.79) ;
\draw    (195.22,119.01) -- (226.87,141.89) ;
\draw  [dash pattern={on 4.5pt off 4.5pt}]  (195.48,119.01) -- (208.65,119.04) -- (295.49,119.26) ;
\draw    (226.37,141.89) -- (265.61,141.89) ;
\draw    (266.12,141.89) -- (295.24,119.26) ;
\draw    (194.97,119.01) -- (181.17,79.64) ;
\draw [fill={rgb, 255:red, 183; green, 134; blue, 225 }  ,fill opacity=1 ]   (181.17,79.64) -- (298.15,43.22) ;
\draw [fill={rgb, 255:red, 183; green, 134; blue, 225 }  ,fill opacity=1 ]   (298.15,43.22) -- (312.33,81.12) ;
\draw    (294.99,119.26) -- (312.33,81.12) ;
\draw    (347.37,62.92) -- (465.14,24.97) ;
\draw    (347.62,62.37) -- (361.33,25.4) ;
\draw    (361.33,25.95) -- (391.21,4.3) ;
\draw    (391.21,4.3) -- (431.21,4.3) ;
\draw    (431.21,4.3) -- (465.4,24.97) ;
\draw    (367.4,119.15) -- (399.05,142.04) ;
\draw    (367.4,119.15) -- (380.57,119.18) -- (467.42,119.4) ;
\draw    (398.04,141.79) -- (437.29,141.79) ;
\draw    (438.05,142.04) -- (467.17,119.4) ;

\draw (5,78.65) node [anchor=north west][inner sep=0.75pt]  [font=\scriptsize] [align=left] {$\displaystyle i$};
\draw (173.04,71.19) node [anchor=north west][inner sep=0.75pt]  [font=\scriptsize] [align=left] {$\displaystyle i$};
\draw (14.76,115.8) node [anchor=north west][inner sep=0.75pt]  [font=\scriptsize] [align=left] {$\displaystyle i'$};
\draw (339.97,64.64) node [anchor=north west][inner sep=0.75pt]  [font=\scriptsize] [align=left] {$\displaystyle i$};
\draw (353.24,112.72) node [anchor=north west][inner sep=0.75pt]  [font=\scriptsize] [align=left] {$\displaystyle i'$};
\draw (182.4,115.77) node [anchor=north west][inner sep=0.75pt]  [font=\scriptsize] [align=left] {$\displaystyle i'$};
\draw (136.86,37.13) node [anchor=north west][inner sep=0.75pt]  [font=\scriptsize] [align=left] {$\displaystyle j$};
\draw (302.17,30.11) node [anchor=north west][inner sep=0.75pt]  [font=\scriptsize] [align=left] {$\displaystyle j$};
\draw (471.44,20.06) node [anchor=north west][inner sep=0.75pt]  [font=\scriptsize] [align=left] {$\displaystyle j$};
\draw (136.86,113.18) node [anchor=north west][inner sep=0.75pt]  [font=\scriptsize] [align=left] {$\displaystyle j'$};
\draw (301,113.59) node [anchor=north west][inner sep=0.75pt]  [font=\scriptsize] [align=left] {$\displaystyle j'$};
\draw (472.22,109.22) node [anchor=north west][inner sep=0.75pt]  [font=\scriptsize] [align=left] {$\displaystyle j'$};
\draw (46.2,83.45) node [anchor=north west][inner sep=0.75pt]  [font=\scriptsize] [align=left] {$\displaystyle X\left( \sigma ^{a+b+c-2}\right)$};
\draw (215.48,24.05) node [anchor=north west][inner sep=0.75pt]  [font=\scriptsize] [align=left] {$\displaystyle X\left( \sigma ^{a}\right)$};
\draw (383.09,18.37) node [anchor=north west][inner sep=0.75pt]  [font=\scriptsize] [align=left] {$\displaystyle X\left( \sigma ^{a}\right)$};
\draw (220.77,87.79) node [anchor=north west][inner sep=0.75pt]  [font=\scriptsize] [align=left] {$\displaystyle X\left( \sigma ^{b+c-1}\right)$};
\draw (400.3,72.8) node [anchor=north west][inner sep=0.75pt]  [font=\scriptsize] [align=left] {$\displaystyle X( \sigma ^{b}$)};
\draw (401.33,122.11) node [anchor=north west][inner sep=0.75pt]  [font=\scriptsize] [align=left] {$\displaystyle X\left( \sigma ^{c}\right)$};
\draw (392.01,142.88) node [anchor=north west][inner sep=0.75pt]  [font=\scriptsize] [align=left] {$\displaystyle 1$};
\draw (222.24,143.31) node [anchor=north west][inner sep=0.75pt]  [font=\scriptsize] [align=left] {$\displaystyle 1$};
\draw (52.86,144.21) node [anchor=north west][inner sep=0.75pt]  [font=\scriptsize] [align=left] {$\displaystyle 1$};
\draw (430.97,145.5) node [anchor=north west][inner sep=0.75pt]  [font=\scriptsize] [align=left] {$\displaystyle k+1$};
\draw (256.88,144.62) node [anchor=north west][inner sep=0.75pt]  [font=\scriptsize] [align=left] {$\displaystyle k+1$};
\draw (89.67,144.21) node [anchor=north west][inner sep=0.75pt]  [font=\scriptsize] [align=left] {$\displaystyle k+1$};

\end{tikzpicture}
    \tikzset{every picture/.style={line width=0.75pt}} 

\begin{tikzpicture}[x=0.75pt,y=0.75pt,yscale=-1,xscale=1]

\draw   (132.32,85.41) -- (120.29,124.71) -- (88.8,149) -- (49.87,149) -- (18.38,124.71) -- (6.35,85.41) -- (18.38,46.1) -- (49.87,21.81) -- (88.8,21.81) -- (120.29,46.1) -- cycle ;
\draw  [dash pattern={on 4.5pt off 4.5pt}]  (18.38,124.71) -- (49.87,21.81) ;
\draw  [dash pattern={on 4.5pt off 4.5pt}]  (88.8,21.81) -- (120.29,124.71) ;
\draw    (332.79,123.42) -- (320.45,84.89) ;
\draw    (320.45,84.89) -- (333.42,46.37) ;
\draw    (333.42,46.37) -- (362.84,23.11) ;
\draw    (362.84,23.11) -- (332.79,123.42) ;
\draw    (453.64,124.46) -- (466.62,85.93) ;
\draw    (466.62,85.93) -- (452.97,47.41) ;
\draw    (452.97,47.41) -- (422.01,24.15) ;
\draw    (422.01,24.15) -- (453.64,124.46) ;
\draw    (171.96,123.27) -- (159.62,84.75) ;
\draw    (159.93,84.75) -- (172.9,46.22) ;
\draw    (172.9,46.22) -- (202.33,22.96) ;
\draw    (202.33,22.96) -- (172.27,123.27) ;
\draw    (284.42,124.66) -- (295.66,85.79) ;
\draw    (295.66,85.79) -- (283.76,47.61) ;
\draw    (283.76,47.61) -- (252.64,24.35) ;
\draw  [dash pattern={on 4.5pt off 4.5pt}]  (252.64,24.35) -- (284.42,124.66) ;
\draw    (284.42,124.66) -- (253.27,148.96) ;
\draw    (253.27,148.61) -- (210.87,148.61) ;
\draw    (210.87,148.61) -- (180.5,124.66) ;
\draw    (180.5,125.01) -- (212.14,24.7) ;
\draw    (212.14,24.7) -- (252.64,24.35) ;
\draw    (412.52,25.19) -- (444.3,125.5) ;
\draw    (444.3,125.5) -- (413.15,149.8) ;
\draw    (413.15,149.45) -- (370.75,149.45) ;
\draw    (370.75,149.45) -- (340.38,125.5) ;
\draw    (340.38,125.85) -- (372.02,25.54) ;
\draw    (372.02,25.54) -- (412.52,25.19) ;
\draw    (162.41,41.39) .. controls (161.28,42.23) and (171.79,52.86) .. (179.1,59.37) ;
\draw [shift={(180.57,60.66)}, rotate = 220.82] [color={rgb, 255:red, 0; green, 0; blue, 0 }  ][line width=0.75]    (10.93,-3.29) .. controls (6.95,-1.4) and (3.31,-0.3) .. (0,0) .. controls (3.31,0.3) and (6.95,1.4) .. (10.93,3.29)   ;
\draw    (322.41,40.99) .. controls (321.28,41.83) and (331.79,52.46) .. (339.1,58.97) ;
\draw [shift={(340.57,60.26)}, rotate = 220.82] [color={rgb, 255:red, 0; green, 0; blue, 0 }  ][line width=0.75]    (10.93,-3.29) .. controls (6.95,-1.4) and (3.31,-0.3) .. (0,0) .. controls (3.31,0.3) and (6.95,1.4) .. (10.93,3.29)   ;
\draw    (460.69,40.54) .. controls (463.06,44.75) and (455.63,54.91) .. (447.6,61.69) ;
\draw [shift={(446.06,62.95)}, rotate = 321.91] [color={rgb, 255:red, 0; green, 0; blue, 0 }  ][line width=0.75]    (10.93,-3.29) .. controls (6.95,-1.4) and (3.31,-0.3) .. (0,0) .. controls (3.31,0.3) and (6.95,1.4) .. (10.93,3.29)   ;

\draw (414.31,11.34) node [anchor=north west][inner sep=0.75pt]  [font=\scriptsize] [align=left] {$\displaystyle i'$};
\draw (254.66,11.29) node [anchor=north west][inner sep=0.75pt]  [font=\scriptsize] [align=left] {$\displaystyle i'$};
\draw (90.62,8.7) node [anchor=north west][inner sep=0.75pt]  [font=\scriptsize] [align=left] {$\displaystyle i'$};
\draw (446.21,129.47) node [anchor=north west][inner sep=0.75pt]  [font=\scriptsize] [align=left] {$\displaystyle j'$};
\draw (286.33,128.63) node [anchor=north west][inner sep=0.75pt]  [font=\scriptsize] [align=left] {$\displaystyle j'$};
\draw (122.2,128.68) node [anchor=north west][inner sep=0.75pt]  [font=\scriptsize] [align=left] {$\displaystyle j'$};
\draw (328.79,128.77) node [anchor=north west][inner sep=0.75pt]  [font=\scriptsize] [align=left] {$\displaystyle i$};
\draw (169.53,127.38) node [anchor=north west][inner sep=0.75pt]  [font=\scriptsize] [align=left] {$\displaystyle i$};
\draw (6.56,127.92) node [anchor=north west][inner sep=0.75pt]  [font=\scriptsize] [align=left] {$\displaystyle i$};
\draw (365.95,11.56) node [anchor=north west][inner sep=0.75pt]  [font=\scriptsize] [align=left] {$\displaystyle j$};
\draw (205.51,10.17) node [anchor=north west][inner sep=0.75pt]  [font=\scriptsize] [align=left] {$\displaystyle j$};
\draw (42.14,8.02) node [anchor=north west][inner sep=0.75pt]  [font=\scriptsize] [align=left] {$\displaystyle j$};
\draw (364.82,151.18) node [anchor=north west][inner sep=0.75pt]  [font=\scriptsize] [align=left] {$\displaystyle 1$};
\draw (205.23,151.13) node [anchor=north west][inner sep=0.75pt]  [font=\scriptsize] [align=left] {$\displaystyle 1$};
\draw (44.63,151.67) node [anchor=north west][inner sep=0.75pt]  [font=\scriptsize] [align=left] {$\displaystyle 1$};
\draw (79.81,151.23) node [anchor=north west][inner sep=0.75pt]  [font=\scriptsize] [align=left] {$\displaystyle k+1$};
\draw (403.63,151.63) node [anchor=north west][inner sep=0.75pt]  [font=\scriptsize] [align=left] {$\displaystyle k+1$};
\draw (243.64,151.13) node [anchor=north west][inner sep=0.75pt]  [font=\scriptsize] [align=left] {$\displaystyle k+1$};
\draw (36.34,97.45) node [anchor=north west][inner sep=0.75pt]  [font=\scriptsize] [align=left] {$\displaystyle X\left( \sigma ^{a+b+c-2}\right)$};
\draw (145.25,25.61) node [anchor=north west][inner sep=0.75pt]  [font=\scriptsize] [align=left] {$\displaystyle X\left( \sigma ^{a}\right)$};
\draw (304.75,25.75) node [anchor=north west][inner sep=0.75pt]  [font=\scriptsize] [align=left] {$\displaystyle X\left( \sigma ^{a}\right)$};
\draw (207.22,97) node [anchor=north west][inner sep=0.75pt]  [font=\scriptsize] [align=left] {$\displaystyle X\left( \sigma ^{b+c-1}\right)$};
\draw (379.39,96.73) node [anchor=north west][inner sep=0.75pt]  [font=\scriptsize] [align=left] {$\displaystyle X( \sigma ^{b}$)};
\draw (448.17,23.67) node [anchor=north west][inner sep=0.75pt]  [font=\scriptsize] [align=left] {$\displaystyle X\left( \sigma ^{c}\right)$};

\end{tikzpicture}
    \caption{Examples of two diagonal cuts. The top is shows a Type A cut and the bottom shows a Type B cut.}
    \label{fig:two cuts intro}
\end{figure}
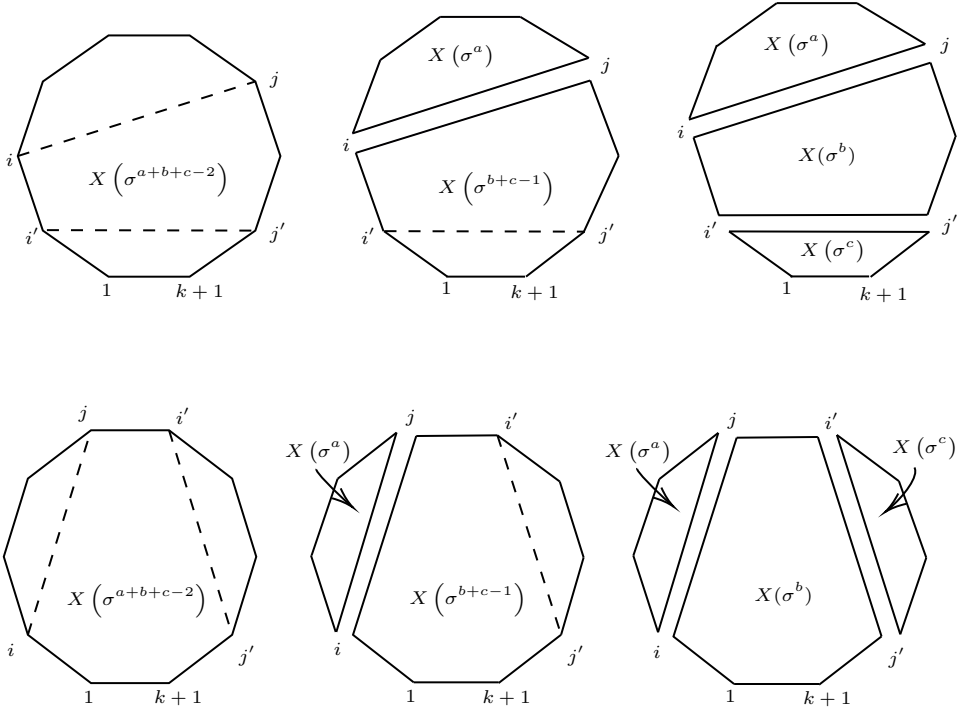

Finally, we study the interactions between the maps $\Phi_{ij}$ associated to different cuts, which can be thought of as associativity of ``gluing" of different braid varieties 
$$X(\sigma^a)\times X(\sigma^b)\times X(\sigma^c)\to X(\sigma^{a+b+c-2}).
$$
Actually, it happens that there are two different cases which we call ``Type A" and ``Type B" cuts (see Figure \ref{fig:two cuts intro} and \ref{fig:two cuts}).

\begin{theorem} Performing two diagonal cuts on $P$ along $\Delta_{ij}$ and $\Delta_{i'j'}$ we have two commutative diagrams
\begin{enumerate}[label=(\roman*)]
    \item For Type A cuts  
$$
\begin{tikzcd}
	{X(\sigma^a)\times X(\sigma^b)\times X(\sigma^c)} && {X(\sigma^{a+b-1})\times X(\sigma^c)} \\
	\\
	{X(\sigma^a)\times X(\sigma^{b+c-1})} && {X(\sigma^{a+b+c-2})}
	\arrow["{\Phi_{ij}\times \Id}", from=1-1, to=1-3]
	\arrow["{\Id\times \Phi_{i'j'}}"', from=1-1, to=3-1]
	\arrow["{\Phi_{ij}}"', from=3-1, to=3-3]
	\arrow["{\Phi_{i'j'}}", from=1-3, to=3-3]
\end{tikzcd}
$$
\item For Type B cuts 
        $$
\begin{tikzcd}
	{X(\sigma^a)\times X(\sigma^b)\times X(\sigma^c)} && {X(\sigma^a)\times X(\sigma^{b+c-1})} && {X(\sigma^{a+b+c-2})} \\
	\\
	{X(\sigma^a)\times X(\sigma^b)\times X(\sigma^c)} &&&& {X(\sigma^{a+b-1})\times X(\sigma^c)}
	\arrow["{\Id\times \Id\times T_{\Delta_{ij}}}"', from=1-1, to=3-1]
	\arrow["{\Id\times\Phi_{i'j'}}", from=1-1, to=1-3]
	\arrow["{\Phi_{ij}}", from=1-3, to=1-5]
	\arrow["{\Phi_{ij}\times \Id}", from=3-1, to=3-5]
	\arrow["{\Phi_{i'j'}}"', from=3-5, to=1-5]
\end{tikzcd}
        $$
        Here $T_{\Delta_{ij}}$ preserves $\Pi_{2,c+1}^{\circ,1}$ and defines a $\C^*$ action on $\Pi_{2,c+1}^{\circ}$ and $\Pi_{2,c+1}^{\circ,1}$ as defined in Lemma \ref{lem: def T} with $\lambda=\Delta_{ij}$.
        Informally, we can say that the gluing $P$ from smaller polygons is associative only up to the additional transformation $T_{\Delta_{ij}}$.
\end{enumerate}

\end{theorem}

\section{Braid varieties}

\subsection{Definition of braid varieties}
      We consider the standard definition of the braid group on $n$ strands, $\Br_n$, given by the presentation 
      \[\Br_n=\langle \sigma_1,\ldots,\sigma_{n-1}\;\textbf{:}\;\sigma_i\sigma_{i+1}\sigma_i=\sigma_{i+1}\sigma_i\sigma_{i+1},\,\sigma_i\sigma_j=\sigma_j\sigma_i\text{ if }|i-j|>1\rangle\]
      where $\sigma_i$ is the positive crossing defined by  
      
       \vspace{.5cm} 
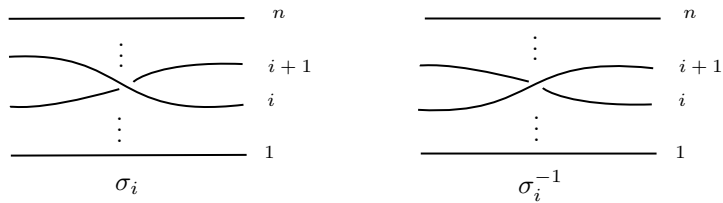
      \begin{center}
          \tikzset{every picture/.style={line width=0.75pt}} 

\begin{tikzpicture}[x=0.75pt,y=0.75pt,yscale=-1,xscale=1]

\draw    (111.62,108.83) .. controls (172.75,105.39) and (161.63,140.82) .. (228.61,132.95) ;
\draw    (112.16,134.09) .. controls (112.74,133.21) and (123.53,137.42) .. (166.64,125.07) ;
\draw    (173.69,121.49) .. controls (174.27,120.61) and (187.13,110.86) .. (228.17,112.48) ;
\draw    (315.73,135.5) .. controls (376.82,138.83) and (366.16,107.29) .. (433.03,114.59) ;
\draw    (316.6,113.05) .. controls (317.17,113.83) and (328.01,110.14) .. (370.96,121.31) ;
\draw    (377.97,124.52) .. controls (378.53,125.3) and (391.27,134.03) .. (432.33,132.77) ;
\draw    (111.45,89.71) -- (229.1,89.35) ;
\draw    (317.07,157.49) -- (434.71,157.51) ;
\draw    (319.09,89.56) -- (436.74,89.58) ;
\draw    (112.54,158.53) -- (230.18,158.17) ;

\draw (164.49,92.57) node [anchor=north west][inner sep=0.75pt]  [font=\small,rotate=-0.12] [align=left] {$\displaystyle \vdots $};
\draw (371.71,129.51) node [anchor=north west][inner sep=0.75pt]  [font=\small,rotate=-0.34] [align=left] {$\displaystyle \vdots $};
\draw (371.07,89.14) node [anchor=north west][inner sep=0.75pt]  [font=\small,rotate=-0.34] [align=left] {$\displaystyle \vdots $};
\draw (163.63,130.08) node [anchor=north west][inner sep=0.75pt]  [font=\small,rotate=-0.12] [align=left] {$\displaystyle \vdots $};
\draw (163.11,168.57) node [anchor=north west][inner sep=0.75pt]   [align=left] {$\displaystyle \sigma _{i}$};
\draw (364.44,164.28) node [anchor=north west][inner sep=0.75pt]  [rotate=-0.88] [align=left] {$\displaystyle \sigma _{i}^{-1}$};
\draw (442.55,152.43) node [anchor=north west][inner sep=0.75pt]  [font=\scriptsize,rotate=-0.5] [align=left] {$\displaystyle 1$};
\draw (446.53,83.7) node [anchor=north west][inner sep=0.75pt]  [font=\scriptsize,rotate=-359.14] [align=left] {$\displaystyle n$};
\draw (443.98,126.02) node [anchor=north west][inner sep=0.75pt]  [font=\scriptsize,rotate=-358.85] [align=left] {$\displaystyle i$};
\draw (444.39,108.75) node [anchor=north west][inner sep=0.75pt]  [font=\scriptsize,rotate=-0.44] [align=left] {$\displaystyle i+1$};
\draw (237.53,152.43) node [anchor=north west][inner sep=0.75pt]  [font=\scriptsize,rotate=-0.5] [align=left] {$\displaystyle 1$};
\draw (241.51,83.7) node [anchor=north west][inner sep=0.75pt]  [font=\scriptsize,rotate=-359.14] [align=left] {$\displaystyle n$};
\draw (238.95,126.02) node [anchor=north west][inner sep=0.75pt]  [font=\scriptsize,rotate=-358.85] [align=left] {$\displaystyle i$};
\draw (239.37,108.75) node [anchor=north west][inner sep=0.75pt]  [font=\scriptsize,rotate=-0.44] [align=left] {$\displaystyle i+1$};

\end{tikzpicture}
      \end{center}
             \vspace{.5cm}

      We consider the positive braid monoid $\Br_n^+\subseteq \Br_n$ which is generated by the nonnegative powers of the generators $\sigma_i$, for $i\in[1,n-1]$. We follow the notations in  \cite{CGGLSS}.\\

      \begin{definition}
          Let $n\in\N$, $i\in[1,n-1]\in\N$ and $z$ a (complex) variable. Then the braid matrix $B_i(z)\in GL(n,\C[z])$ is defined 
      
      \[(B_i(z))_{jk}:=\begin{cases} 1 &j=k\text{ and }j\neq i,i+1\\-1 &(j,k)=(i,i+1)\\1 &(i+1,i)\\z&j=k=i\\0&\text{otherwise}\end{cases},\;\;\text{i.e.}\;\; B_i(z):=\begin{pmatrix}1&\cdots&&&\ldots&0\\ \vdots&\ddots&&&&\vdots\\ 0&\cdots&z&-1&\cdots&0\\ 0&\cdots&1&0&\cdots&0\\ \vdots&&&&\ddots&\vdots\\ 0&\cdots&&&\cdots&1\end{pmatrix}\]
      Given a positive braid word $\beta=\sigma_{i_1}\cdots\sigma_{i_r}\in \Br_n^+$ and $z_1,\ldots,z_r$ complex variables, define the braid matrix \[B_\beta(z_1,\ldots,z_r)=B_{i_1}(z_1)\cdots B_{i_r}(z_r)\in GL(n,\C[z_1,\ldots,z_r]).\]
      
      \end{definition}
      
      Braid matrices satisfy the braid relations up to a change of variables given as
      \begin{align*}
          B_i(z_1)B_{i+1}(z_2)B_i(z_3)&=B_{i+1}(z_3)B_i(z_1z_3-z_2)B_{i+1}(z_1),&\text{for all } i\in[1,n-2]\\
          B_i(z_1)B_j(z_z)&=B_j(z_2)B_i(z_1), &\text{for } |i-j|>1.
      \end{align*}\\

      \vspace{-.25cm}
      \begin{figure}
         \centering
         \tikzset{every picture/.style={line width=0.75pt}} 

\begin{tikzpicture}[x=0.75pt,y=0.75pt,yscale=-1,xscale=1]

\draw    (120.65,90.35) .. controls (184.68,78.46) and (178.05,194.4) .. (228.21,123.1) ;
\draw    (121.33,141.54) .. controls (154.15,149.26) and (173.15,130.58) .. (171.99,125.56) ;
\draw    (177.78,116) .. controls (234.23,39.3) and (243.67,210.64) .. (286.41,121.59) ;
\draw    (237.37,113.74) .. controls (293.82,37.05) and (299.4,206.53) .. (342.62,119.13) ;
\draw    (293.61,110.46) .. controls (352.75,51.6) and (338.67,167.58) .. (398.5,142.23) ;
\draw    (351.98,104.83) .. controls (363.08,80.74) and (398.82,100.17) .. (397.66,95.16) ;
\draw  [dash pattern={on 4.5pt off 4.5pt}]  (120,70.5) -- (120.5,160) ;
\draw  [dash pattern={on 4.5pt off 4.5pt}]  (399.5,70.5) -- (400,160) ;

\draw (167.97,142.34) node [anchor=north west][inner sep=0.75pt]  [font=\footnotesize] [align=left] {$\displaystyle z_{1}$};
\draw (342.2,142.34) node [anchor=north west][inner sep=0.75pt]  [font=\footnotesize] [align=left] {$\displaystyle z_{4}$};
\draw (286.83,143.17) node [anchor=north west][inner sep=0.75pt]  [font=\footnotesize] [align=left] {$\displaystyle z_{3}$};
\draw (226.72,143.99) node [anchor=north west][inner sep=0.75pt]  [font=\footnotesize] [align=left] {$\displaystyle z_{2}$};

\end{tikzpicture}
         \vspace{-1.5cm}
         \caption{The braid $\sigma_1^4$ with each crossing $j$ labeled with a complex variable $z_j$.}
         \label{fig:braid variables}
     \end{figure}
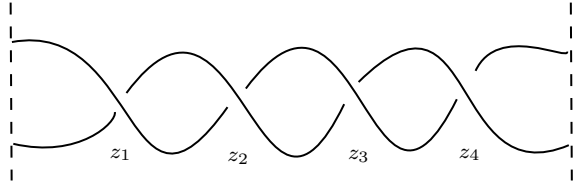

      This paper concerns only braids on two strands, for the remainder of the paper we will refer to a two strand braid with $k$ crossings as $\sigma^k$. The braid matrices on two strands are given by
$$
B(z)=\left(
\begin{matrix}
z & -1\\
1 & 0\\
\end{matrix}
\right)
$$

\begin{definition}
The braid variety $X(\sigma^k)$ on two strands is defined by the equation
$$X(\sigma^k):=\left\{(z_1,\ldots,z_k) :\:
\left(
\begin{matrix}
0 & -1\\
1 & 0\\
\end{matrix}
\right)B(z_1)\cdots B(z_k)\ \text{is upper-triangular}.\right\}
$$
\end{definition}
 If $\beta$ and $\beta'$ are related by braid moves then $X(\beta)\simeq X(\beta')$, this isomorphism arises from the invariance of braid matrices. It is easy to see that the use of $-1$ does not affect that definition of $X(\beta)$.\\

To develop a general understanding of these varieties we first consider cases of a small number of crossings.
\begin{example}
    Let $\beta=\sigma^1\in \Br^+_2$, the braid matrix is given by \[B_\beta(z_1)=\begin{pmatrix} z_1 &-1\\1&0\end{pmatrix}\]
    Therefore, the braid variety is defined as
    \begin{align*}
        X(\sigma^1)&=\left\{z_1:\; 
\left(
\begin{matrix}
0 & -1\\
1 & 0\\
\end{matrix}
\right)\begin{pmatrix} z_1 &-1\\1&0\end{pmatrix}\  \text{is upper-triangular}.\right\}\\
&=\left\{z_1:\begin{pmatrix}
    -1 &0\\z_1&-1
\end{pmatrix}\ \text{is upper triangular}\right\}
=\{z_1=0\}.
\end{align*}
More precisely, $X(\sigma^1)$ is a point.
\end{example}

\begin{example}
    Let $\beta=\sigma^2\in \Br^+_2$ with braid matrix 
            \begin{align*} 
            B_\beta(z_1,z_2)
            =\begin{pmatrix}
                z_1z_2-1 &-z_1\\z_2 &-1
            \end{pmatrix}
            \end{align*}
            then the braid variety associated to $\beta$ is
            \begin{align*}
                X(\sigma^2)&=\left\{(z_1,z_2) :\; \begin{pmatrix}
        0 & -1\\
        1 & 0\\
        \end{pmatrix}
        \begin{pmatrix}
                z_1z_2-1 &-z_1\\z_2 &-1
            \end{pmatrix}\ \text{is upper-triangular}.\right\}\\
            &=\{(z_1,z_2)\in\C^2:z_1z_2-1=0\}
            \cong\{z_1\in\C:z_1\neq0\}
            \end{align*}
        It is important to note that the choice of coordinate $z_1$ on $X(\sigma^2)=\{z_1\ne0\}$ is not unique in this case, we may have also chosen $X(\sigma^2)=\{z_2\ne0\}$. However, the choice of $X(\sigma^2)=\{z_1\ne0\}$ is helpful when developing an inductive way to describe the braid variety in order to compute its cohomology.
\end{example}

\begin{example}
    Let $\beta=\sigma^3\in \Br^+_2$ with braid matrix 
            \begin{align*} 
            B_\beta(z_1,z_2,z_3)
            =\begin{pmatrix}
                z_1z_2z_3-z_3-z_1 &1-z_1z_2\\z_2z_3-1 &-z_2
            \end{pmatrix}
            \end{align*}
            then the braid variety associated to $\beta$ is 
            \begin{align*}
                X(\sigma^3)&=\left\{(z_1,z_2,z_3):\; \begin{pmatrix}
        0 & -1\\
        1 & 0\\
        \end{pmatrix}
        \begin{pmatrix}
                z_1z_2z_3-z_3-z_1 &1-z_1z_2\\z_2z_3-1 &-z_2
            \end{pmatrix}\ \text{is upper-triangular}.\right\}\\
            &=\{(z_1,z_2,z_3)\in\C^3:z_1z_2z_3-z_3-z_1=0\}
            \cong\{(z_1,z_2)\in\C^2:z_1z_2-1\neq0\}
            \end{align*}
\end{example}

There is an inductive relationship between $X(\sigma^k)$ and
$X(\sigma^{k-1})$, we explore this concept further by first establishing 
general formulas for the braid matrices then extending these results to the polynomials defining  the braid varieties. Moreover, with these results we show that the braid variety $X(\sigma^k)$ is smooth.

\begin{lemma}[Hughes\cite{H}, Chantraine-Ng-Sivek\cite{CNS}]
\label{lem: braid mat}
    One can express the braid matrix for $\beta=\sigma^k$ as
    $$B_\beta(z_1,\ldots,z_k)=\begin{pmatrix}
        F_k(z_1,\dots,z_k)&-F_{k-1}(z_1,\dots,z_{k-1})\\F_{k-1}(z_2,\dots,z_k)&-F_{k-2}(z_2,\dots,z_{k-1})
    \end{pmatrix}$$
    where 
    \begin{equation}
    \label{eq: def F.}
    F_k(z_i,\dots,z_{i+k})=z_{i+k}F_{k-1}(z_i,\dots,z_{i+k-1})-F_{k-2}(z_i,\dots,z_{i+k-2})
    \end{equation}
    with initial values $F_1(z_i)=z_i$, $F_0\equiv1$ and $F_{-1}\equiv0$. 
\end{lemma}
\begin{proof}
    We proceed with induction on $k$. Clearly, $$B_{\sigma^1}(z_1)=\begin{pmatrix}
        z_1&-1\\1&0
    \end{pmatrix}=\begin{pmatrix}
        F_1(z_1)&-F_0(\emptyset)\\F_0(\emptyset)&-F_{-1}(\emptyset)
    \end{pmatrix}$$
    Suppose 
    $$
    B_{\sigma^k}(z_1,\ldots,z_k)=\begin{pmatrix}
        F_k(z_1,\dots,z_k)&-F_{k-1}(z_1,\dots,z_{k-1})\\F_{k-1}(z_2,\dots,z_k)&-F_{k-2}(z_2,\dots,z_{k-1})
    \end{pmatrix}
    $$
    Then
    \begin{align*}
        B_{\sigma^{k+1}}(z_1,\ldots,z_k,z_{k+1})&=B_{\sigma^{k}}(z_1,\ldots,z_k)B_\sigma(z_{k+1})\\
        &=\begin{pmatrix}
        F_k(z_1,\dots,z_k)&-F_{k-1}(z_1,\dots,z_{k-1})\\F_{k-1}(z_2,\dots,z_k)&-F_{k-2}(z_2,\dots,z_{k-1})
    \end{pmatrix}\begin{pmatrix}
        z_{k+1} &-1\\1&0
    \end{pmatrix}\\
    &=\begin{pmatrix}
        z_{k+1}F_k(z_1,\dots,z_k)-F_{k-1}(z_1,\dots,z_{k-1}&-F_{k-1}(z_1,\dots,z_{k-1})\\z_{k+1}F_{k-1}(z_2,\dots,z_k)-F_{k-2}(z_2,\dots,z_k)&-F_{k-1}(z_2,\dots,z_k)
    \end{pmatrix}\\
    &=\begin{pmatrix}
        F_{k+1}(z_1,\dots,z_{k+1})&-F_k(z_1,\dots,z_k)\\F_{k}(z_2,\dots,z_k)&-F_{k-1}(z_2,\dots,z_k)
    \end{pmatrix}
    \end{align*}    
\end{proof}

\begin{theorem}[Hughes \cite{H}]
The braid variety $X(\sigma^k)$ is defined in $\C^k$ by the equation $F_k(z_1,\ldots,z_k)=0$ where $F_k$ is given by the recursion \eqref{eq: def F.}. 

Moreover, if $F_k(z_1,\dots,z_k)=0$, then $F_{k-1}(z_1,\dots,z_{k-1})\neq0$ and $z_k=\dfrac{F_{k-2}(z_1,\dots,z_{k-2})}{F_{k-1}(z_1,\dots,z_k)}$.
\end{theorem}

\begin{proof}
By Lemma \ref{lem: braid mat}, we express the braid matrix as $$B_\beta(z_1,\ldots,z_k)=\begin{pmatrix}
        F_k(z_1,\dots,z_k)&-F_{k-1}(z_1,\dots,z_{k-1})\\F_{k-1}(z_2,\dots,z_k)&-F_{k-2}(z_2,\dots,z_{k-1})
    \end{pmatrix}$$
    Using the definition for a braid variety, we find that 
    \begin{align*}
    X(\sigma^k)&=\left\{(z_1,
    \ldots,z_k) :\ \begin{pmatrix}
        0&-1\\1&0
    \end{pmatrix}B_\beta(z_1,\dots,z_k) \text{ is upper-triangular}\right\}\\
    &=\{(z_1,\dots,z_k)\in\C^k:F_k(z_1,\dots,z_k)=0\} 
    \end{align*}

    Given that $F_n(z_1,\dots,z_k)=0$ and $F_k=z_kF_{k-1}-F_{k-2}$. If $F_{k-1}\neq0$, then we can solve the equation $F_k=0$ for $z_k$:
   $$  
        F_k=z_kF_{k-1}-F_{k-2}=0,\ 
        z_k=\frac{F_{k-2}}{F_{k-1}}.
    $$
        Suppose instead that $F_{k-1}(z_1,\dots,z_{k-1})=0$ and given that $F_k(z_1,\dots,z_k)=0$ by the definition of $X(\sigma^k)$, then $F_{k-2}(z_1,\dots,z_{k-2})=0$. By proceeding with downward induction on $k$, we conclude that $F_k(z_1,\dots,z_k)=0$ for all $k$, contradicting $F_0=1$. Therefore, $F_{k-1}(z_1,\dots,z_{k-1})\neq0$.

\end{proof}

\begin{corollary}
\label{cor: f  minus 1}
    We have $X(\sigma^k)\simeq \{(z_1,\dots,z_{k-1}):F_{k-1}(z_1,\dots,z_{k-1})\neq0\}.$
\end{corollary}

\begin{corollary}
    \label{cor: variety smooth}
    The braid variety $X(\sigma^k)$ is smooth of complex dimension $k-1$.
\end{corollary}
\begin{proof}
    By Corollary \ref{cor: f  minus 1}, $X(\sigma^k)=\{(z_1,\dots,z_{k-1}):F_{k-1}(z_1,\dots,z_{k-1})\neq0\}$. Since $\{(z_1,\dots,z_{k-1}):F_{k-1}(z_1,\dots,z_{k-1})\neq0\}$ is an open subset in $\C^{k-1}$, then $X(\sigma^k)$ is a smooth manifold.
\end{proof}

\subsection{Cohomology using Alexander and Poincar\'e duality}

Given the inductive definition of the two strand braid variety $X(\beta)$ we may determine the homology in terms of the vector space with Alexander and Poincar\'e duality. Our varieties are non-compact, so we have to be careful and sometimes use cohomology with compact support.

\begin{theorem}
\label{thm: poincare}
    (Poincar\'e Duality) If $M$ is an orientable $n$-manifold then we have an isomorphism
    $
    \widetilde{H}^k_{c}(M;\C)\simeq  \widetilde{H}_{n-k}(M;\C)
    $    for all $k$.

\end{theorem}

\begin{theorem}
\label{thm: alex}
    (Alexander Duality) If $K$ is a locally contractible, nonempty, proper subspace of $\mathbb{R}^n$, then $\widetilde{H}_i(\mathbb{R}^n-K;\C)\simeq \widetilde{H}^{n-i-1}_c(K;\C)$ for all $i$.
\end{theorem}

The cohomology of two-strand braid varieties was computed in \cite[Section 6.2, Proposition 9.13]{LS} using cluster algebra methods (compare with Theorem \ref{thm: alg forms} below). Here we give a simpler inductive proof using Poincar\'e and Alexander dualities.

\begin{theorem}
\label{thm: homology}
    Let $\beta=\sigma^n$, then the homology of the two-strand braid variety is given by: $$H^i(X(\beta))=\begin{cases}
        \C &\text{for }0\leq i\leq n-1\\ 0 &\text{otherwise}.
    \end{cases}$$
\end{theorem}
\begin{proof}
We proceed by induction on $n$. Given Corollary \ref{cor: f  minus 1}, then $$H^i(X(\sigma^2))=H^i(\{z_1z_2-1=0\})=H^i(\{z_1\neq 0\})=H^i(\C^*)$$ Since $H^i(\C^*)=\C$ for $i=0,1$, then the theorem is true for $n=2$. Supposing the statement holds for $n=k$ we determine that
\begin{align*}
    \widetilde{H}_i(X(\sigma^{k+1}))&=\widetilde{H}_i(\{F_{k+1}=0\}) 
    =\widetilde{H}_i(\{F_k\neq0\})\ (\mathrm{by\ Corollary}\ \ref{cor: f  minus 1}) \\
    &=\widetilde{H}^{2k-1-i}_c(\{F_k=0\})\quad (\mathrm{by\ Theorem}\ \ref{thm: alex})
    =\widetilde{H}^{2k-1-i}_c(X(\sigma^k))\\
    &=\widetilde{H}_{2k-2-(2k-1-i)}(X(\sigma^k))\  (\mathrm{by\  Theorem}\ \ref{thm: poincare})
    =\widetilde{H}_{i-1}(X(\sigma^k)).
\end{align*}
Since $\widetilde{H}_i(X(\sigma^{k+1}))=\begin{cases}
    \C &1\leq i\leq k+1\\ 0 &\text{otherwise}
\end{cases}$, we obtain $H_i(X(\sigma^{k+1}))=\begin{cases}
    \C &0\leq i\leq k+1\\
    0 &\text{otherwise}.
\end{cases}$
\end{proof}

\subsection{The Grassmannian and the open positroid variety}

The \textit{Grassmannian} $\Gr(k,n)$ parametrizes all $k$-dimensional linear subspaces of the $n$-dimensional space, presented as the row span of a $k\times n$ matrix of maximal rank. Let $v_1,\ldots,v_n$ be the columns of the matrix where $v_i$ are $k$-dimensional vectors. Given $I\in\binom{[n]}{k}$ the \textit{Pl\"ucker coordinate} $\Delta_I(A)$ is the minor of $k\times k$ submatrix of $A$ in column set $I$.
    
\begin{definition}[\cite{KLS}]
\label{def: variety}
    The open positroid variety $\Pi_{k,n}$ is defined as the set of elements in the Grassmannian such that there is a representative $k\times n$-matrix such that all cyclically consecutive $k\times k$ minors don't vanish\[\Delta_{i,\ldots,i+k-1}=\det(v_i,\ldots,v_{i+k-1})\neq0\]
    This condition does not depend on the representative, so the positroid is well-defined.
\end{definition}

\begin{definition}
\label{def: variety subset}
    Let $\Pi_{2,n}^{\circ,1}$ be the subset of the the open positroid variety such that each $\Delta_{i,i+1}=1$ for all $1\leq i\le n-1$ and $\Delta_{1,n}\neq0$.
\end{definition}

\begin{lemma}
\label{lem: v to z}
Suppose that $v_1,\ldots,v_{k+1}$ is a collection of vectors in $\C^2$ such that $v_1=(1,0)$ and $\det(v_i,v_{i+1})=1$. Then there exists a unique collection of parameters $z_1,\ldots,z_k$ such that $B(z_1)\cdots B(z_i)=(v_{i+1}\quad -v_i)$ for all $i$.
\end{lemma}

\begin{proof}
Let $v_i=(v_i^{1},v_i^{2})$, we prove the statement by induction in $i$.
For $i=1$ we have $v_1=(1,0)$ and $v_2=(z,1)$ since $\det(v_i,v_{i+1})=1$. For $i>1$ the vectors $v_{i-1},v_i$ form a basis of $\C^2$, so
we can write $v_{i+1}=\alpha v_{i-1}+\beta v_{i}$. Now
$$
\det(v_i,v_{i+1})=\alpha \det(v_i,v_{i-1})+\beta\det(v_i,v_i)=-\alpha \det(v_{i-1},v_i)+0=-\alpha
$$
so $\alpha=-1$ and we can denote $z_i=\beta$ and write 
\begin{equation}
\label{eq: v recursion}
v_{i+1}=-v_{i-1}+z_i v_i.
\end{equation}
Now 
$$
\left(
\begin{matrix}
v_{i+1}^{1} \quad -v_{i}^{1}\\
v_{i+1}^{2} \quad -v_{i}^{2}
\end{matrix}
\right)=
\left(
\begin{matrix}
v_{i}^{1} \quad -v_{i-1}^{1}\\
v_{i}^{2} \quad -v_{i-1}^{2}
\end{matrix}
\right) \left(
\begin{matrix}
z_i & -1\\
1 & 0\\
\end{matrix}
\right)
$$
and by assumption of induction we have
$$
B(z_1)\cdots B(z_{i-1})=\left(
\begin{matrix}
v_{i}^{1} \quad -v_{i-1}^{1}\\
v_{i}^{2} \quad -v_{i-1}^{2}
\end{matrix}
\right).
$$
\end{proof}

\begin{remark}
Note that $B(z_1)\cdots B(z_i)\left(\begin{matrix}1\\ 0\end{matrix}\right)=v_{i+1}$.
\end{remark}

\begin{lemma}
\label{thm: iso}
Let $\Pi_{2,k+1}^o$ and $\Pi_{2,k+1}^{o,1}$ be as described in \ref{def: variety subset} and \ref{def: variety}, then 
\begin{enumerate}[label=\alph*)]
    \item $\Pi_{2,k+1}^{\circ,1}$ is isomorphic to $X(\sigma^k)$.
    \item $\Pi_{2,k+1}^{\circ}$ is isomorphic to $X(\sigma^k)\times (\C^*)^k$.
\end{enumerate}
\end{lemma}

\begin{proof}
a) We package the vectors $v_k$ in a $2\times (k+1)$ matrix $V$.
Since $\Delta_{1,2}=1$, we can use row operations  to make sure that the first column of $V$ is $(1,0)$, so we get
$$
V=\left(
\begin{matrix}
1 & v_2^1 & \cdots &v_{k+1}^1\\
0 & v_2^2 & \cdots &v_{k+1}^2.
\end{matrix}
\right)\in\Pi_{2,k+1}^{\circ,1}.
$$
By Lemma \ref{lem: v to z} we can uniquely find the variables $z_1,\ldots,z_k$ such that 
$$
V=\left(
\begin{matrix}
1 & F_1(z_1) & \cdots &F_k(z_1,\dots,z_k)\\
0 & F_0 & \cdots &F_{k-1}(z_2,\dots,z_k)
\end{matrix}
\right) 
$$
Note that $\det B_i(z)=1$, so $\det B_{\beta}(z_1,\ldots,z_i)=1$ for any braid $\beta$ and 
\begin{equation}
\label{eq: det}
F_i(z_1,\dots,z_i)F_{i}(z_2,\dots,z_{i+1})-F_{i+1}(z_1,\dots,z_{i+1})F_{i-1}(z_2,\dots,z_i)=1,
\end{equation}
so the matrix $V$ indeed satisfies $\Delta_{i,i+1}(V)=1$. The matrix $V$ belongs to $\Pi_{2,k}^{\circ,1}$ if and only if $F_{k-1}(z_2,\dots,z_k)\ne 0$. In this case, we can use row operations to ensure that $F_k(z_1,\dots,z_k)=0$ (we subtract from the first row $F_k(z_1,\dots,z_k)/F_{k-1}(z_2,\dots,z_k)$ times the second row).  

The braid variety $X(\sigma^k)$ is cut out by the equation $\{(z_1,\dots,z_k)\in\C^k:F_{k}(z_1,\dots,z_k)=0\}$, so we get a map from $\Pi_{2,k}^{\circ,1}$ to $X(\sigma^k)$. To construct the inverse, observe that $F_{k}(z_1,\dots,z_k)=0$ and \eqref{eq: det} implies that $$F_{k-1}(z_1,\dots,z_{k-1})F_{k-1}(z_1,\dots,z_k)=1,$$ so $F_{k-1}(z_2,\dots,z_k)\neq 0$.
Therefore $\Pi_{2,k}^{\circ,1}\simeq X(\sigma^k)$.

b) Similarly to the above, we can use row operations to ensure any matrix in $\Pi_{2,k+1}^{\circ}$ has the first column $(1,0)$. Now we define a map $\Pi_{2,k+1}^{\circ,1}\times (\C^*)^k\to \Pi_{2,k+1}^{\circ}$ by rescaling all other columns:
$$
\varphi:[(v_1,v_2,\ldots,v_{k+1}),(\lambda_1,\ldots,\lambda_k)]\mapsto 
(v_1,\lambda_1v_2,\ldots,\lambda_k v_{k+1}).
$$
The inverse map is clear, since we get 
$$
\det (\lambda_{i-1}v_i,\lambda_{i}v_{i+1})=\lambda_{i-1}\lambda_i,
$$
and the scalars $\lambda_i$ can recovered from the minors $\Delta_{i,i+1}$ for the image of $\varphi$.

\end{proof}

\begin{example}
We have 
$$
B(z_1)B(z_2)=\left(
\begin{matrix}
z_1 & -1\\
1 & 0\\
\end{matrix}
\right)
\left(
\begin{matrix}
z_2 & -1\\
1 & 0\\
\end{matrix}
\right)=
\left(
\begin{matrix}
z_1z_2-1 & -z_1\\
z_2 & -1\\
\end{matrix}
\right)
$$
$$
B(z_1)B(z_2)B(z_3)=
\left(
\begin{matrix}
z_1z_2-1 & -z_1\\
z_2 & -1\\
\end{matrix}
\right)\left(
\begin{matrix}
z_3 & -1\\
1 & 0\\
\end{matrix}
\right)=
\left(
\begin{matrix}
z_1z_2z_3-z_1-z_3 & 1-z_1z_2\\
z_2z_3-1 & -z_2\\
\end{matrix}
\right).
$$
This means that we can package $v_i$ in a matrix
$$
\left(v_1 \quad v_2 \quad v_3\quad v_4\right)=
\left(
\begin{matrix}
1 & z_1 & z_1z_2-1 & z_1z_2z_3-z_1-z_3 \\
0 & 1  & z_2 & z_2z_3-1.
\end{matrix}
\right)
$$
\end{example}

\subsection{Cluster algebras}
Cluster algebras are commutative rings that are not defined in the typical sense by generators and relations, instead it is defined by a seed $s$ which consists of a quiver, or exchange matrix, and cluster variables, which is a finite collection of algebraically independent elements of the algebra. This seed along with a concept of mutation generates a subring of a field $\mathcal{F}$. We refer to \cite{W} for more details on cluster algebras.

A cluster variety is an affine algebraic variety $X$ defined by a collection of open charts $U\simeq(\C^*)^d$ where each chart $U$ is parametrized by cluster coordinates $A_1,\ldots,A_d$ which are invertible on $U$ and extend to regular functions on $X$. These coordinates can be either mutable or frozen where the coordinate is frozen if it extends to a non-vanishing regular function on $X$. 

For each chart we assign a skew-symmetric integer matrix $\varepsilon_{ij}$ called the exchange matrix to a quiver $Q$ defined by 
$$
(\varepsilon_{ij})=\begin{cases}
    a &\text{if there are $a$ arrows from vertex $i$ to vertex $j$;}\\
    -a &\text{if there are $a$ arrows from vertex $j$ to vertex $i$;}\\
    0 &\text{otherwise}
\end{cases}$$
For each chart $U$ and each mutable variable $A_k$, there is another chart $U'$ with cluster coordinates $A_1,\ldots,A'_k,\ldots,A_d$ and a skew-symmetric matrix $\varepsilon'_{ij}$ related by mutation $\mu_k$, where the mutation is defined by 
\begin{equation}
\label{eq: mutation}
A_k'A_k=\left(\displaystyle\prod_{\varepsilon_{ki\geq0}}A_i^{\varepsilon_{ki}}+\prod_{\varepsilon_{ki\leq0}}A^{-\varepsilon_{ki}}_i\right)
\end{equation}
If $i\ne k$ then the cluster variables $A_i$ remain unchanged.

When performing a mutation, we modify the quiver using the following rules:
\begin{enumerate}
    \item If there is a path of the vertices $i\rightarrow k\rightarrow j$, then we add an arrow from $i$ to $j$.
    \item Any arrows incident to $k$ change orientation.
    \item Remove a maximal disjoint collection of 2-cycles produced in Steps (1) and (2).
\end{enumerate}

Any two charts in the cluster algebra are related by a sequence of mutations $\underline{\mu}$, and $\mu_k$ is an involution. Given these conditions the ring of functions on $X$ is generated by all cluster variables in all charts.

We will need the notion of exchange ratios defined as the ratio of two terms in \eqref{eq: mutation}:
$$
\widehat{y_i}=\frac{\prod_{\varepsilon_{ki\geq0}}A_i^{\varepsilon_{ki}}}{\prod_{\varepsilon_{ki\leq0}}A^{-\varepsilon_{ki}}_i}.
$$

Let $V$ be a rational affine algebraic variety with algebra of regular functions $\C[V]$ and field of rational functions $\C(V)$. 
\begin{definition}\cite{Fraser, FSB}
\label{def: quasi}
    Let $\Sigma$ and $\Sigma_0$ be seeds of rank $r$ in $\C(V)$. Let $Q, A_i, \hat{y_i}$ denote the  quiver, cluster variables and exchange ratios in $\Sigma$ and use primes to denote these quantities in $\Sigma_0$. We assume that $A_{r+1}, \ldots , A_d$ are frozen.
    Then $\Sigma$ and $\Sigma_0$ are quasi-equivalent, denoted $\Sigma\sim\Sigma_0$, if the following hold:
    \begin{itemize}
        \item The groups $\mathbf{P}, \mathbf{P}_0 \subset \C[V ]$ of Laurent monomials in frozen variables coincide. That is, each frozen variable $A'_i$ is a Laurent monomial in $\{A_{r+1}, \ldots, A_d\}$ and vice versa.
        \item Corresponding mutable variables coincide up to multiplication by an element of $\mathbf{P}$: for $i\in[r]$, there is a Laurent monomial $M_i\in\mathbf{P}$ such that $A_i = M_iA'_i\in \C(V)$.
        \item The exchange ratios (3) coincide: $\hat{y}_i = \hat{y}'_i$ for $i\in[r]$.
    \end{itemize}
    Quasi-equivalence is an equivalence relation on seeds.
Seeds $\Sigma,\Sigma_0$ are related by a quasi-cluster transformation if there exists a finite sequence
$\underline{\mu}$ of mutations such that $\underline{\mu}(\Sigma) \sim \Sigma_0$.
\end{definition}

By the main result of \cite{Fraser}, it is sufficient to check the conditions of quasi-equivalence in one cluster, and they will automatically hold in every other cluster.

\begin{theorem}[Scott\cite{S}, Galashin--Lam\cite{GL}, Serhiyenko--Sherman-Bennett--Williams\cite{SSBW}]
\label{thm: cstructure}
        Any open positroid variety has a cluster structure.
    \end{theorem}

For positroid varieties $\Pi_{2,k+1}^{\circ}$ we obtain a cluster variety of type $A_{k-1}$ with $k+1$ frozen variables. We asssign the vectors $v_i$ from Lemma \ref{thm: iso} to the vertices of a regular polygon $\mathcal{P}$. The cluster charts in $\Pi_{2,k+1}^{\circ}$ are determined by  triangulations of $\mathcal{P}$. Given a triangulation, the edges between the vertices $i$ and $j$ correspond to cluster variables determined by the Pl\"ucker coordinates $\Delta_{i,j}=\det(v_i,v_j)$.

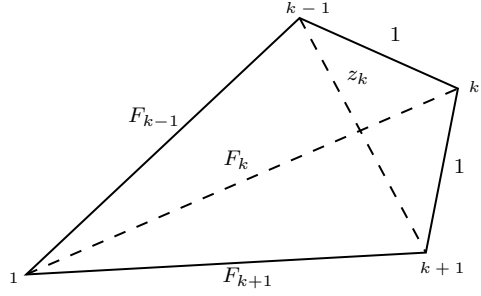
\begin{figure}
    \centering
    \tikzset{every picture/.style={line width=0.75pt}} 

\begin{tikzpicture}[x=0.75pt,y=0.75pt,yscale=-1,xscale=1]

\draw   (255.5,61.5) -- (334.5,97) -- (318.5,179.5) -- (119.09,190.52) -- (119.09,190.52) -- cycle ;
\draw  [dash pattern={on 4.5pt off 4.5pt}]  (119.09,190.52) -- (334.5,97) ;
\draw  [dash pattern={on 4.5pt off 4.5pt}]  (255.5,61.5) -- (318.5,179.5) ;

\draw (109,188) node [anchor=north west][inner sep=0.75pt]  [font=\tiny] [align=left] {$\displaystyle 1$};
\draw (338.5,91.5) node [anchor=north west][inner sep=0.75pt]  [font=\tiny] [align=left] {$\displaystyle k$};
\draw (314,182.5) node [anchor=north west][inner sep=0.75pt]  [font=\tiny] [align=left] {$\displaystyle k+1$};
\draw (247,52) node [anchor=north west][inner sep=0.75pt]  [font=\tiny] [align=left] {$\displaystyle k-1$};
\draw (331,131) node [anchor=north west][inner sep=0.75pt]  [font=\footnotesize] [align=left] {$\displaystyle 1$};
\draw (299,64.5) node [anchor=north west][inner sep=0.75pt]  [font=\footnotesize] [align=left] {$\displaystyle 1$};
\draw (168.5,105) node [anchor=north west][inner sep=0.75pt]  [font=\footnotesize] [align=left] {$\displaystyle F_{k-1}$};
\draw (215.5,185.5) node [anchor=north west][inner sep=0.75pt]  [font=\footnotesize] [align=left] {$\displaystyle F_{k+1}$};
\draw (216,126.5) node [anchor=north west][inner sep=0.75pt]  [font=\footnotesize] [align=left] {$\displaystyle F_{k}$};
\draw (277.5,86.5) node [anchor=north west][inner sep=0.75pt]  [font=\footnotesize] [align=left] {$\displaystyle z_{k}$};

\end{tikzpicture}
    \caption{Section of the triangulation of $\fan$, see Figure \ref{fig:Ufan}, between the vertices $1,k-1,k$ and $k+1$}
    \label{fig:plucker}
\end{figure}

\begin{lemma}[Hughes\cite{H}]
\label{lem: Delta from F}
In $\Pi_{2,k+1}^{\circ,1}$ for all $i<j$ we have
$$
\Delta_{i,j}=F_{j-i-1}(z_{i+1},\ldots,z_{j-1}).
$$
In particular, $\Delta_{i,i+2}=z_{i+1}$.
\end{lemma}

\begin{proof}
Using the results from Lemma \ref{lem: v to z}, we have the following relations \begin{align*}
    B(z_1)\dots B(z_i)&=\begin{pmatrix}
    v_{i+1} &-v_i
    \end{pmatrix}\\
    B(z_1)\dots B(z_j)&=\begin{pmatrix}
    v_{j+1} &-v_j
\end{pmatrix}
\end{align*}
Given that $i<j$, we then rewrite
\begin{align*}
    B(z_1)\dots B(z_i)B(z_{i+1})\dots B(z_j)&=\begin{pmatrix}
        v_{j+1} &-v_j
    \end{pmatrix}\\
    \begin{pmatrix}
        v_{i+1}&-v_i
    \end{pmatrix}B(z_{i+1})\dots B(z_j)&=\begin{pmatrix}
        v_{j+1} &-v_j
    \end{pmatrix}
\end{align*}
From Theorem \ref{lem: braid mat}, the product of the braid matrices from $i+1$ to $j$ can be expressed as $$B(z_{i+1})\dots B(z_j)=\begin{pmatrix}
    F_{j-i}(z_{i+1},\dots,z_j) &-F_{j-i-1}(z_{i+1},\dots,z_{j-1})\\F_{j-i-1}(z_{i+2},\dots,z_j) &-F_{j-i-2}(z_{i+2},\dots,z_{j-1})
\end{pmatrix}$$
Which allows us to rewrite the previous equation as $$
\begin{pmatrix}
        v_{i+1}&-v_i
    \end{pmatrix}\begin{pmatrix}
    F_{j-i}(z_{i+1},\dots,z_j) &-F_{j-i-1}(z_{i+1},\dots,z_{j-1})\\F_{j-i-1}(z_{i+2},\dots,z_j) &-F_{j-i-2}(z_{i+2},\dots,z_{j-1})
\end{pmatrix}=\begin{pmatrix}
        v_{j+1} &-v_j
    \end{pmatrix}
    $$
    Here we obtain the equation 
    $$
    -v_j=-F_{j-i-1}(z_{i+1},\dots,z_{j-1})v_{i+1}+F_{j-i-2}(z_{i+2},\dots,z_{j-1})v_i$$
By finding an expression for $v_j$, we may now determine $\Delta_{ij}$, since determinants are linear, we find that
    \begin{align*}
       \Delta_{ij}&=\det\begin{pmatrix}
        v_i&v_j
    \end{pmatrix}=F_{j-i-1}(z_{i+1},\dots,z_{j-1})\det\begin{pmatrix}
        v_i&v_{i+1}
    \end{pmatrix}-F_{j-i-2}(z_{i+2},\dots,z_{j-1})\det\begin{pmatrix}
        v_i &v_i
    \end{pmatrix}\\
    &=F_{j-i-1}(z_{i+1},\dots,z_{j-1})(1)-F_{j-i-2}(z_{i+2},\dots,z_{j-1})(0)=F_{j-i-1}(z_{i+1},\dots,z_{j-1})
    \end{align*}

    To see that $\Delta_{i,i+2}=z_{i+1}$, we see that $\Delta_{ij}=F_{(i+2)-i-1}(z_{i+1})=F_1(z_{i+1})=z_{i+1}$ as desired.
\end{proof}

For $a<b<c<d$ we have the Pl\"ucker relation 
\begin{equation}
\label{eq: plucker}
\Delta_{ac}\Delta_{bd}=\Delta_{ab}\Delta_{cd}+\Delta_{ad}\Delta_{bc}.
\end{equation}
A special case of \eqref{eq: plucker} is 
$$
\Delta_{i,k}\Delta_{k-1,k+1}=\Delta_{i,k-1}\Delta_{k,k+1}+\Delta_{i,k+1}\Delta_{k-1,k}
$$
which in $\Pi_{2,k+1}^{\circ,1}$ translates to
$$
\Delta_{i,k}z_k=\Delta_{i,k-1}+\Delta_{i,k+1}
$$
For $i=1$ it is indeed equivalent to our recursion \eqref{eq: def F.}, see  Figure \ref{fig:plucker}

Outer edges of $\mathcal{P}$ correspond to frozen variables, while diagonals correspond to  mutable variables. In particular, $\Pi_{2,k+1}^{\circ}$ has $k$ frozen variables, whereas in  $\Pi_{2,k+1}^{\circ,1}$ we specialize the  frozen Pl\"ucker coordinates, $\Delta_{i,i+1}=1$ for $1\le i\le k$ and these can be neglected.   Thus $\Pi_{2,k+1}^{\circ,1}$ has one frozen variable $\Delta_{1,k+1}$ which we denote by $w$. To generate the quiver, in each triangle of the triangulation we connect the cluster variables by arrows in a clockwise order.
Mutations correspond to flips of triangulations due to the Plucker relation.

Consider the special chart $\fan$ in $\Pi_{2,k+1}^{\circ,1}$ corresponding to the ``fan" triangulation where the $k-2$ diagonals are defined by $\Delta_{1,i}$ for $2\le i\le k$, as seen in Figure \ref{fig:Ufan}. Equivalently, the chart $\fan$ is given by inequalities
$$
\fan=\{F_{i-1}(z_2,\dots,z_i)\neq 0,1\le i\le k\}\subset X(\sigma^k).
$$

\begin{figure}
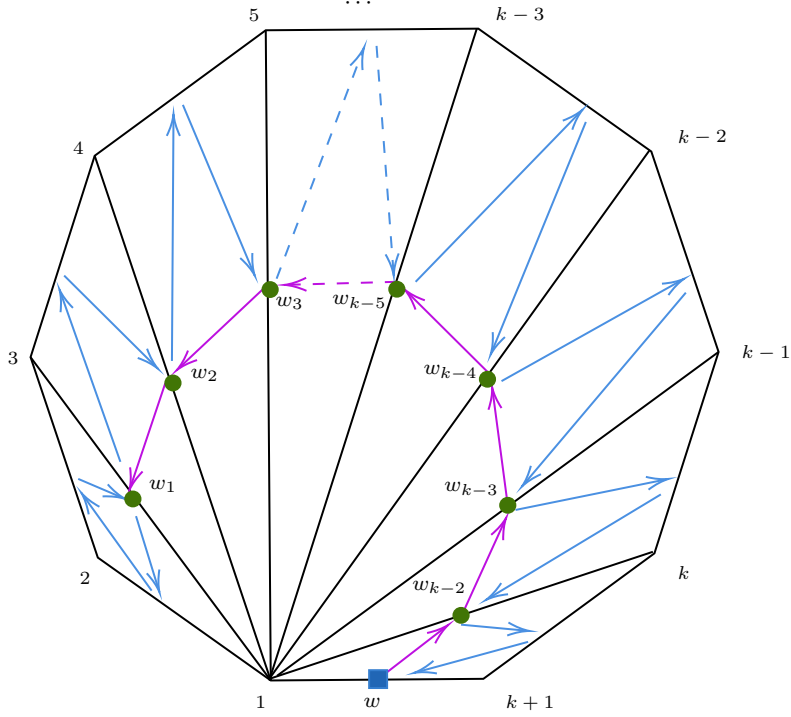

    \centering
    \include{Figures/U_fan}
    \caption{The special chart $\fan\in\Pi_{2,k+1}^{\circ,1}$ where each of the $k-2$ diagonals are have fixed endpoint at $v_1$. The Plu\"cker coordinates, or cluster variables, correspond to the weights of the edges given by either a blue square (frozen vertices) or a green circle (mutable vertices). The quiver of the cluster chart is generated by clockwise orientation of the colored arrows in each triangle of the triangulation. This procedure produces the quiver $A_{k-1}$, seen in purple, with $w$ as the singular frozen variable. In the terminology of \cite{CGGLSS}, this chart is given by the right inductive weave.}
    \label{fig:Ufan}
\end{figure}

 In this chart, the quiver is precisely $A_{k-1}$ with one frozen variable $w$. From lemma {} the mutable cluster variables are precisely $w_i=F_i(z_2,\dots,z_{i+1})$ and the frozen variable is $w=w_{k-2}=F_k(z_2,\ldots,z_{k+1})$.

\section{Ring structure on cohomology using (algebraic) deRham cohomology}

\subsection{Constructing the forms}
\label{sec: forms}

Define the one-form $\alpha=\frac{dw}{w}$ where $w=\Delta_{1,k+1}$ is the frozen cluster variable. Since $w\neq 0$ everywhere, $\alpha$ is regular everywhere.

Define the two-form as 
\begin{equation}
    \label{eq: def Omega}
\omega=\sum \varepsilon_{ij}\frac{dw_i}{w_i}\wedge\frac{dw_j}{w_j}
\end{equation}
on some cluster chart with quiver $(\varepsilon_{ij})$. By \cite[Section 2.3]{GSV} (see also \cite{LS}) the form $\omega$ is well-defined in any other cluster chart  and is given by a similar equation \eqref{eq: def Omega} for the mutated quiver. 
The cluster charts cover $X(\sigma^k)$ up to codimension 2 and $X(\sigma^k)$ is smooth, so $\omega$ extends to a regular form on $X(\sigma^k)$.

For the special chart $\fan$ we get
\begin{equation}
\label{eq: Omega fan}
\omega=\frac{dw}{w}\wedge\frac{dw_{k-2}}{w_{k-2}}+\sum_{i=1}^{k-3}\frac{dw_{i+1}}{w_{i+1}}\wedge \frac{dw_i}{w_i}
\end{equation}
where $w_i=\Delta_{1,i+2}$.

We can also write the forms $\alpha$ and $\omega$ explicitly in the coordinates $z_i$. Thus far, we have expressed $X(\sigma^k)$ as an open subset in the affine space with coordinates $z_1,\dots,z_{k-1}$ with $z_k$ expressed as some function of these. Similarly, we may also have expressed $X(\sigma^k)$ is an open subset in the affine space with coordinates $z_2,\ldots,z_k$ with $z_1$ expressed as some function of these, i.e., $F_k(z_1,\dots,z_k)=z_1F_{k-1}(z_2,\dots,z_k)-F_{k-2}(z_3,\dots,z_k)$ where $F_{-1}\equiv0$, $F_0\equiv1$, and $F_1(z_2)=z_2$. We will use $z_2,\ldots,z_k$ as a coordinate system on $X(\sigma^k)$ below.

\begin{lemma}
\label{lem: delta derivative}
For all $2\le i\le k$ and $2\le n\le k+1$ we have
$$
\frac{\partial \Delta_{1n}}{\partial z_i}=\Delta_{1i}\Delta_{in}.
$$
\end{lemma}

\begin{proof}
We have the matrix identity
$$
\left(
\begin{matrix}
F_n(z_1,\dots,z_n) & -F_{n-1}(z_1,\dots,z_{n-1})\\
F_{n-1}(z_2,\dots,z_n)  & -F_{n-2}(z_2,\dots,z_{n-1})
\end{matrix}
\right)=
C
\left(
\begin{matrix}
z_i & -1\\
1  & 0
\end{matrix}
\right)
\widetilde{C}
$$
where $$C=\left(
\begin{matrix}
F_{i-1}(z_1,\dots,z_{i-1}) & -F_{i-2}(z_1,\dots,z_{i-2})\\
F_{i-2}(z_2,\dots,z_{i-1})  & -F_{i-3}(z_2,\dots,z_{i-2})
\end{matrix}
\right),\, \widetilde{C}=\left(
\begin{matrix}
F_{n-i}(z_{i+1},\dots,z_n) & -F_{n-i-1}(z_{i+2},\dots,z_{n-1})\\
F_{n-i-1}(z_{i+2},\dots,z_n)  & -F_{n-i-2}(z_{i+2},\dots,z_{n-1})
\end{matrix}
\right),$$
which implies

\begin{align*}
    F_{n-2}(z_2,\dots,z_{n-1})=&(F_{i-2}(z_2,\dots,z_{i-1})z_i-F_{i-3}(z_2,\dots,z_{i-2}))F_{n-i-1}(z_{i+2},\dots,z_{n-1})\\&-F_{i-2}(z_2,\dots,z_{i-1})F_{n-i-2}(z_{i+2},\dots,z_{n-1})
\end{align*}
and
$$
\frac{\partial F_{n-2}(z_2,\dots,z_{n-1})}{\partial z_i}=F_{i-2}(z_2,\dots,z_{i-1})F_{n-i-1}(z_{i+2},\dots,z_{n-1}).
$$

Now by Lemma \ref{lem: Delta from F} we have 
$\Delta_{1,n}=F_{n-2}(z_2,\ldots,z_{n-1})$ and
$$
\frac{\partial \Delta_{1,n}}{\partial z_i}=
F_{i-2}(z_2,\dots,z_{i-1})F_{n-i-1}(z_{i+2},\dots,z_{n-1})=\Delta_{1,i}\Delta_{i,n}.
$$
\end{proof}

\begin{corollary}
We have
$$
\alpha=\frac{d\Delta_{1,k+1}}{\Delta_{1,k+1}}=\frac{1}{\Delta_{1,k+1}}\sum_{i=2}^{k}\Delta_{1,i}\Delta_{i,k+1}dz_i.
$$
\end{corollary}

\begin{lemma}
\label{lem: fractions}
For $i\le k$ we have 
$$
\frac{1}{\Delta_{1,i}\Delta_{1,i+1}}+\ldots+\frac{1}{\Delta_{1,k}\Delta_{1,k+1}}=\frac{\Delta_{i,k+1}}{\Delta_{1,i}\Delta_{1,k+1}}.
$$
\end{lemma}

\begin{proof}
We prove it by induction in $k$, for $k=i$ the statement is clear since $\Delta_{i,i+1}=1$. For the step of induction, suppose that it is true for $k-1$, then
\begin{align*}
    \frac{1}{\Delta_{1,i}\Delta_{1,i+1}}+\ldots+\frac{1}{\Delta_{1,k-1}\Delta_{1,k}}+\frac{1}{\Delta_{1,k}\Delta_{1,k+1}}&=\frac{\Delta_{i,k}}{\Delta_{1,i}\Delta_{1,k}}+\frac{1}{\Delta_{1,k}\Delta_{1,k+1}}\\    &=\frac{\Delta_{i,k}\Delta_{1,k+1}+\Delta_{1,i}\Delta_{k,k+1}}{\Delta_{1,i}\Delta_{1,k}\Delta_{1,k+1}},
\end{align*}

which by Pl\"ucker relation simplifies to 
$$
\frac{\Delta_{1,k}\Delta_{i,k+1}}{\Delta_{1,i}\Delta_{1,k}\Delta_{1,k+1}}=\frac{\Delta_{i,k+1}}{\Delta_{1,i}\Delta_{1,k+1}}.
$$
\end{proof}

\begin{lemma}
\label{lem: Omega formula}
We have 
$$
\omega=\frac{1}{\Delta_{1,k+1}}\sum_{2\le i<j\le k}\Delta_{1,i}\Delta_{i,j} \Delta_{j,k+1}dz_i\wedge dz_j.
$$
\end{lemma}

\begin{proof}
By Lemma \ref{lem: delta derivative} we can write 
$$
d\Delta_{1,s}\wedge d\Delta_{1,s+1}=
\sum_{i<j\le s} (\Delta_{1,i}\Delta_{i,s}\Delta_{1,j}\Delta_{j,s+1}-\Delta_{1,i}\Delta_{i,s+1}\Delta_{1,j}\Delta_{j,s})dz_i\wedge dz_j=
$$
$$
\sum_{i<j\le s} \Delta_{1,i}\Delta_{1,j}(\Delta_{i,s}\Delta_{j,s+1}-\Delta_{i,s+1}\Delta_{j,s})dz_i\wedge dz_j.
$$
By Pl\"ucker relation we have
$$
\Delta_{i,s}\Delta_{j,s+1}-\Delta_{i,s+1}\Delta_{j,s}=\Delta_{ij},
$$
hence
$$
d\Delta_{1,s}\wedge d\Delta_{1,s+1}=\sum_{i<j\le s} \Delta_{1,i}\Delta_{1,j}\Delta_{i,j} dz_i\wedge dz_j.
$$

The coefficient at $dz_i\wedge dz_j$ does not depend on $k$, so we get
$$
\omega=\sum_{s=1}^{k}\frac{d\Delta_{1,s}\wedge d\Delta_{1,s+1}}{\Delta_{1,s}\Delta_{1,s+1}}=
\sum_{i<j}\Delta_{1,i}\Delta_{1,j}\Delta_{i,j} dz_i\wedge dz_j\left(\frac{1}{\Delta_{1,j}\Delta_{1,j+1}}+\ldots+\frac{1}{\Delta_{1,k}\Delta_{1,k+1}}\right).
$$
By Lemma \ref{lem: fractions} this simplifies to
$$
\sum_{i<j}\frac{\Delta_{1,i}\Delta_{1,j}\Delta_{i,j} \Delta_{j,k+1}dz_i\wedge dz_j}{\Delta_{1,j}\Delta_{1,k+1}}=\frac{\sum_{i<j}\Delta_{1,i}\Delta_{i,j} \Delta_{j,k+1}dz_i\wedge dz_j}{\Delta_{1,k+1}}.
$$
\end{proof}

In particular, Lemma \ref{lem: Omega formula} gives a direct proof that $\omega$ is regular everywhere on $X(\sigma^k)$. See Section \ref{sec: examples forms} for explicit examples and computations.

\subsection{de Rham cohomology}

By construction, $d\alpha=d\omega=0$, so they represent some de Rham cohomology classes. The following theorem shows that these are in fact nonzero in cohomology and generate $H^*(X(\sigma^k))$ as an algebra.

\begin{theorem}
\label{thm: alg forms}
The forms $\alpha$ and $\omega$ generate $H^*(X(\sigma^k))$ as an algebra, modulo the following relations:

1) If $k$ is even, the only relation is $\omega^{\frac{k}{2}}=0$. The basis in cohomology is given by:
\begin{equation}
\label{eq: basis even}
1,\alpha,\omega,\alpha\omega,\ldots,\omega^{\frac{k}{2}-1},\alpha\omega^{\frac{k}{2}-1}.
\end{equation}

2) If $k$ is odd, the relations are $\alpha\omega^{\frac{k-1}{2}}=\omega^{\frac{k+1}{2}}=0$. The basis in cohomology is given by: 
\begin{equation}
\label{eq: basis odd}
1,\alpha,\omega,\alpha\omega,\ldots,\alpha\omega^{\frac{k-3}{2}},\omega^{\frac{k-1}{2}}.
\end{equation}
\end{theorem}

\begin{proof}
We work in the chart $\fan$, there is a natural inclusion map $i:\fan\to X(\sigma^k)$ and the corresponding restriction map in cohomology:
$i^*:H^*(X(\sigma^k))\to H^*(\fan)$. 

We want to first prove that the restrictions of all the forms \eqref{eq: basis even} and  \eqref{eq: basis odd} to $H^*(\fan)$ do not vanish, this would imply that these forms do not vanish in $H^*(X(\sigma^k))$. Recall that $\fan\simeq (\C^*)^{k-1}$ with coordinates $w_1,\ldots,w_{k-2},w=w_{k-1}$, so $H^*(\fan)$ is isomorphic to an exterior algebra in $\frac{dw_i}{w_i}$. 

Suppose $k$ is odd then
    \begin{align*}
        \omega^{\frac{k-1}{2}}&=\left(\frac{dw_2}{w_2}\wedge\frac{dw_1}{w_1}+\ldots+\frac{dw}{w}\wedge\frac{dw_{k-2}}{w_{k-2}}\right)^{(k-1)/2}\\
        &=\left(k-1)/2\right)!\frac{dw_1}{w_1}\wedge \dots\wedge\frac{dw_{k-2}}{w_{k-2}}\wedge \frac{dw}{w}
    \end{align*}
and 
\begin{align*}
        \alpha\omega^{\frac{k-3}{2}}&=
        \frac{dw}{w}\wedge\left(\frac{dw_2}{w_2}\wedge\frac{dw_1}{w_1}+\ldots+\frac{dw}{w}\wedge\frac{dw_{k-2}}{w_{k-2}}\right)^{(k-3)/2}\\
        &=\frac{dw}{w}\wedge\left(\frac{dw_2}{w_2}\wedge\frac{dw_1}{w_1}+\ldots+\frac{dw_{k-2}}{w_{k-2}}\wedge\frac{dw_{k-3}}{w_{k-3}}\right)^{(k-3)/2}\\
        &=\frac{dw}{w}\wedge\left((k-3)/2 \right)!\sum_{j=0}^{(k-5)/2}\frac{dw_1}{w_1}\wedge \dots
        \widehat{\frac{dw_{2j+1}}{w_{2j+1}}}\dots
        \wedge\frac{dw_{k-3}}{w_{k-3}}
    \end{align*}
In particular, these are nonzero. Suppose $k$ is even, then similarly
    \begin{align*}
        \omega^{\frac{k}{2}-1}=
        \sum_{j=0}^{k/2-2}
        \frac{dw_1}{w_1}\wedge\frac{dw_2}{w_2}\wedge\dots \wedge\widehat{\frac{dw_{2j+1}}{w_{2j+1}}}\wedge \dots \wedge\frac{dw_{k-2}}{w_{k-2}}\wedge\frac{dw}{w}.\\     
    \end{align*}

 and
 \begin{align*}
        \alpha\omega^{\frac{k-3}{2}}&=\left((k-1)/2 \right)!\frac{dw_1}{w_1}\wedge\frac{dw_2}{w_2}\wedge\dots\wedge\frac{dw_{k-3}}{w_{k-3}}\wedge \frac{dw_{k-2}}{w_{k-2}}\wedge \frac{dw}{w}\\
    \end{align*}

This implies that all the forms in \eqref{eq: basis even} and  \eqref{eq: basis odd} are nonzero in $H^*(\fan)$ and hence nonzero in $H^*(X(\sigma^k))$. On the other hand, by Theorem \ref{thm: homology} the corresponding cohomology groups of $X(\sigma^k)$ are one-dimensional in each degree; therefore, we obtain a basis. 
\end{proof}

\subsection{Examples}
\label{sec: examples forms}

\begin{example}
    Braid variety associated to $\beta=\sigma^3$ 
            \begin{align*}
                X(\sigma^3)&=\left\{z_1z_2z_3-z_3-z_1=0\right\}\\
            &=\left\{z_1z_2-1\neq0\right\}
            \end{align*}
            Using row operations and scaling the columns, we can transform any matrix in $\Pi_{2,4}^\circ$ to the form

$$
V=\left(
\begin{matrix}
1 & z_1 & z_1z_2-1 & z_1z_2z_3-z_1-z_3 \\
0 & 1  & z_2 & z_2z_3-1.
\end{matrix}
\right)\in\Pi_{2,4}^{\circ,1}
$$
Using the correspondence of cluster algebras and Grassmannians, we obtain two cluster charts, as seen in Figure \ref{fig:sigma 3}:

    \begin{figure}[ht!]
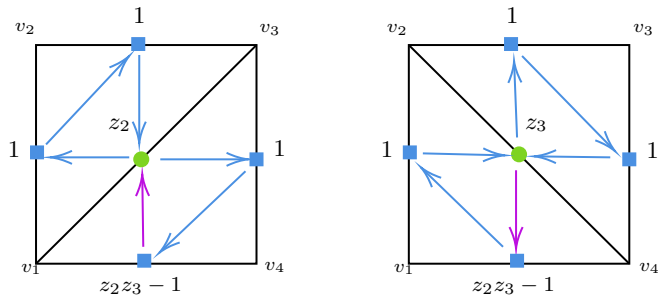

        \centering
        \include{Figures/Charts_sigma_3}
        \vspace{-1cm}
        \caption{The two cluster charts for the braid variety $X(\sigma^3)$. On the left is chart $U_1$ where the vectors $v_i\in\Pi_{2,4}^{\circ,1}$ for $1\le i\le 4$ correspond to the vertices of the polygon. The purple arrow depicts the Dynkin diagram $A_1$ with a frozen. On the right is chart $U_2$ which corresponds to the mutation of chart $U_1$.}
        \label{fig:sigma 3}
    \end{figure}
         \vspace{-.5cm}
        $$U_1=\{z_2\neq0\} \text{ with coordinates }(w_1=z_2,\;w=z_2z_3-1)
        $$
        $$U_2=\{z_3\neq0\} \text{ with coordinates }(w'_1=z_3,\;w=z_2z_3-1)
        $$
        
        We compute the cohomology of $X(\beta)$ using the (algebraic) de Rham cohomology on chart 1.
        Let $U_1=\{w_1=z_2\neq0,\;w=z_2z_3-1\neq0\}$. Then all possible forms   \[ H^*(U_1)=H^*\big((\C^*)^2\big)=\left\langle 1,\frac{dw_1}{w_1},\frac{dw}{w},\frac{dw}{w}\wedge\frac{dw_1}{w_1}\right\rangle\]
        To determine the cohomology, it suffices to determine which of the above forms extend to $X(\sigma^3)$. The forms which extend are
    \begin{itemize}
        \item $1$
        \item $\dfrac{dw}{w}=\dfrac{z_2 dz_3+z_3dz_2}{z_2z_3-1}$
        \item $\dfrac{dw}{w}\wedge\dfrac{dw_1}{w_1}=\dfrac{dz_3\wedge dz_2}{z_2z_3-1}$
    \end{itemize}
  
    The 2-form can be deduced from the quiver shown in Figure \ref{fig:sigma 3} which agrees with \cite{LS}. Therefore, $H^0(X(\sigma^3))=H^1(X(\sigma^3))=H^2(X(\sigma^3))=\C$, which agrees with Theorem \ref{thm: homology}.

    In addition, on the chart $U_2=\{w'_1=z_3\ne0,\;w=z_2z_3-1\ne 0\}$, with possible forms
    $$ 
    H^*(U_2)=H^*\big((\C^*)^2\big)=\left\langle 1,\frac{dw'_1}{w'_1},\frac{dw}{w},\frac{dw}{w}\wedge\frac{dw'_1}{w'_1}\right\rangle
    $$
    the forms which extend are 
    \begin{itemize}
        \item $1$
        \item $\dfrac{dw}{w}=\dfrac{z_2 dz_3+z_3dz_2}{z_2z_3-1}$
        \item $\dfrac{dw}{w}\wedge\dfrac{dw'_1}{w'_1}=\dfrac{dz_2\wedge dz_3}{z_2z_3-1}$
    \end{itemize}
    Indeed, the cohomology of $X(\sigma^3)$ is independent from the choice of a chart.
\end{example}

\begin{example}
    The braid variety associated to $\beta=\sigma^4$ 
            \begin{align*}
                X(\sigma^4)&=\left\{z_1z_2z_3z_4-z_1z_2-z_1z_4-z_3z_4+1=0\right\}\\
            &=\left\{z_1z_2z_3-z_3-z_1\neq0\right\}
            \end{align*}
            with open positroid variety of the form            
\[
V=
\left(
\begin{matrix}
1 & z_1 & z_1z_2-1 & z_1z_2z_3-z_1-z_3 & z_1z_2z_3z_4-z_1z_2-z_1z_4-z_3z_4+1\\
0 & 1  & z_2 & z_2z_3-1 &z_2z_3z_4-z_2-z_4
\end{matrix}
\right)\in\Pi_{2,5}^{\circ,1}
\]
Using the correspondence of cluster algebras and Grassmannians, we obtain one of five cluster charts, see Figure \ref{fig:sigma 4}:
\begin{figure}[ht!]
    \centering
    \tikzset{every picture/.style={line width=0.75pt}} 

\begin{tikzpicture}[x=0.75pt,y=0.75pt,yscale=-1,xscale=1]

\draw   (353.33,120.42) -- (310.88,248.13) -- (171.41,248.77) -- (127.67,121.45) -- (240.1,42.13) -- cycle ;
\draw    (171.39,248.76) -- (240.12,42.13) ;
\draw    (171.39,248.76) -- (353.34,120.43) ;
\draw  [color={rgb, 255:red, 74; green, 144; blue, 226 }  ,draw opacity=1 ][fill={rgb, 255:red, 74; green, 144; blue, 226 }  ,fill opacity=1 ] (238.86,244.35) -- (247.43,244.35) -- (247.43,252.59) -- (238.86,252.59) -- cycle ;
\draw  [color={rgb, 255:red, 74; green, 144; blue, 226 }  ,draw opacity=1 ][fill={rgb, 255:red, 74; green, 144; blue, 226 }  ,fill opacity=1 ] (326.71,181.18) -- (335.29,181.18) -- (335.29,189.42) -- (326.71,189.42) -- cycle ;
\draw  [color={rgb, 255:red, 74; green, 144; blue, 226 }  ,draw opacity=1 ][fill={rgb, 255:red, 74; green, 144; blue, 226 }  ,fill opacity=1 ] (293.14,78.87) -- (301.71,78.87) -- (301.71,87.11) -- (293.14,87.11) -- cycle ;
\draw  [color={rgb, 255:red, 74; green, 144; blue, 226 }  ,draw opacity=1 ][fill={rgb, 255:red, 74; green, 144; blue, 226 }  ,fill opacity=1 ] (172.43,81.62) -- (181,81.62) -- (181,89.86) -- (172.43,89.86) -- cycle ;
\draw  [color={rgb, 255:red, 74; green, 144; blue, 226 }  ,draw opacity=1 ][fill={rgb, 255:red, 74; green, 144; blue, 226 }  ,fill opacity=1 ] (145.29,179.81) -- (153.86,179.81) -- (153.86,188.04) -- (145.29,188.04) -- cycle ;
\draw  [color={rgb, 255:red, 126; green, 211; blue, 33 }  ,draw opacity=1 ][fill={rgb, 255:red, 126; green, 211; blue, 33 }  ,fill opacity=1 ] (203.86,138.61) .. controls (203.86,135.95) and (206.1,133.8) .. (208.86,133.8) .. controls (211.62,133.8) and (213.86,135.95) .. (213.86,138.61) .. controls (213.86,141.26) and (211.62,143.41) .. (208.86,143.41) .. controls (206.1,143.41) and (203.86,141.26) .. (203.86,138.61) -- cycle ;
\draw  [color={rgb, 255:red, 126; green, 211; blue, 33 }  ,draw opacity=1 ][fill={rgb, 255:red, 126; green, 211; blue, 33 }  ,fill opacity=1 ] (257.37,184.6) .. controls (257.37,181.94) and (259.61,179.79) .. (262.37,179.79) .. controls (265.13,179.79) and (267.37,181.94) .. (267.37,184.6) .. controls (267.37,187.25) and (265.13,189.4) .. (262.37,189.4) .. controls (259.61,189.4) and (257.37,187.25) .. (257.37,184.6) -- cycle ;
\draw [color={rgb, 255:red, 74; green, 144; blue, 226 }  ,draw opacity=1 ]   (274.57,187.36) -- (316.86,186.7) ;
\draw [shift={(318.86,186.67)}, rotate = 179.11] [color={rgb, 255:red, 74; green, 144; blue, 226 }  ,draw opacity=1 ][line width=0.75]    (10.93,-3.29) .. controls (6.95,-1.4) and (3.31,-0.3) .. (0,0) .. controls (3.31,0.3) and (6.95,1.4) .. (10.93,3.29)   ;
\draw [color={rgb, 255:red, 74; green, 144; blue, 226 }  ,draw opacity=1 ]   (323.14,192.85) -- (254.76,242.48) ;
\draw [shift={(253.14,243.66)}, rotate = 324.03] [color={rgb, 255:red, 74; green, 144; blue, 226 }  ,draw opacity=1 ][line width=0.75]    (10.93,-3.29) .. controls (6.95,-1.4) and (3.31,-0.3) .. (0,0) .. controls (3.31,0.3) and (6.95,1.4) .. (10.93,3.29)   ;
\draw [color={rgb, 255:red, 74; green, 144; blue, 226 }  ,draw opacity=1 ]   (201,142.73) -- (159.68,177.15) ;
\draw [shift={(158.14,178.43)}, rotate = 320.2] [color={rgb, 255:red, 74; green, 144; blue, 226 }  ,draw opacity=1 ][line width=0.75]    (10.93,-3.29) .. controls (6.95,-1.4) and (3.31,-0.3) .. (0,0) .. controls (3.31,0.3) and (6.95,1.4) .. (10.93,3.29)   ;
\draw [color={rgb, 255:red, 74; green, 144; blue, 226 }  ,draw opacity=1 ]   (182.43,95.35) -- (201.48,130.67) ;
\draw [shift={(202.43,132.43)}, rotate = 241.66] [color={rgb, 255:red, 74; green, 144; blue, 226 }  ,draw opacity=1 ][line width=0.75]    (10.93,-3.29) .. controls (6.95,-1.4) and (3.31,-0.3) .. (0,0) .. controls (3.31,0.3) and (6.95,1.4) .. (10.93,3.29)   ;
\draw [color={rgb, 255:red, 74; green, 144; blue, 226 }  ,draw opacity=1 ]   (298.14,96.04) -- (268.15,173.13) ;
\draw [shift={(267.43,175)}, rotate = 291.26] [color={rgb, 255:red, 74; green, 144; blue, 226 }  ,draw opacity=1 ][line width=0.75]    (10.93,-3.29) .. controls (6.95,-1.4) and (3.31,-0.3) .. (0,0) .. controls (3.31,0.3) and (6.95,1.4) .. (10.93,3.29)   ;
\draw [color={rgb, 255:red, 74; green, 144; blue, 226 }  ,draw opacity=1 ]   (220.29,137.24) -- (288.61,92.33) ;
\draw [shift={(290.29,91.23)}, rotate = 146.69] [color={rgb, 255:red, 74; green, 144; blue, 226 }  ,draw opacity=1 ][line width=0.75]    (10.93,-3.29) .. controls (6.95,-1.4) and (3.31,-0.3) .. (0,0) .. controls (3.31,0.3) and (6.95,1.4) .. (10.93,3.29)   ;
\draw [color={rgb, 255:red, 189; green, 16; blue, 224 }  ,draw opacity=1 ]   (255.29,177.06) -- (218.28,147.41) ;
\draw [shift={(216.71,146.16)}, rotate = 38.7] [color={rgb, 255:red, 189; green, 16; blue, 224 }  ,draw opacity=1 ][line width=0.75]    (10.93,-3.29) .. controls (6.95,-1.4) and (3.31,-0.3) .. (0,0) .. controls (3.31,0.3) and (6.95,1.4) .. (10.93,3.29)   ;
\draw [color={rgb, 255:red, 189; green, 16; blue, 224 }  ,draw opacity=1 ]   (244.57,238.85) -- (259.61,196.79) ;
\draw [shift={(260.29,194.91)}, rotate = 109.68] [color={rgb, 255:red, 189; green, 16; blue, 224 }  ,draw opacity=1 ][line width=0.75]    (10.93,-3.29) .. controls (6.95,-1.4) and (3.31,-0.3) .. (0,0) .. controls (3.31,0.3) and (6.95,1.4) .. (10.93,3.29)   ;
\draw [color={rgb, 255:red, 74; green, 144; blue, 226 }  ,draw opacity=1 ]   (152.43,172.25) -- (174.01,98.64) ;
\draw [shift={(174.57,96.73)}, rotate = 106.34] [color={rgb, 255:red, 74; green, 144; blue, 226 }  ,draw opacity=1 ][line width=0.75]    (10.93,-3.29) .. controls (6.95,-1.4) and (3.31,-0.3) .. (0,0) .. controls (3.31,0.3) and (6.95,1.4) .. (10.93,3.29)   ;

\draw (301,64.3) node [anchor=north west][inner sep=0.75pt]  [font=\scriptsize] [align=left] {$\displaystyle 1$};
\draw (171,64.99) node [anchor=north west][inner sep=0.75pt]  [font=\scriptsize] [align=left] {$\displaystyle 1$};
\draw (342.43,180.34) node [anchor=north west][inner sep=0.75pt]  [font=\scriptsize] [align=left] {$\displaystyle 1$};
\draw (131.71,181.03) node [anchor=north west][inner sep=0.75pt]  [font=\scriptsize] [align=left] {$\displaystyle 1$};
\draw (183.86,128.35) node [anchor=north west][inner sep=0.75pt]  [font=\scriptsize] [align=left] {$\displaystyle z_{2}$};
\draw (209.21,176.3) node [anchor=north west][inner sep=0.75pt]  [font=\scriptsize] [align=left] {$\displaystyle z_{2} z_{3} -1$};
\draw (206.79,257.32) node [anchor=north west][inner sep=0.75pt]  [font=\scriptsize] [align=left] {$\displaystyle z_{2} z_{3} z_{4} -z_{2} -z_{4}$};
\draw (159.5,249.71) node [anchor=north west][inner sep=0.75pt]  [font=\scriptsize] [align=left] {$\displaystyle v_{1}$};
\draw (312,248.08) node [anchor=north west][inner sep=0.75pt]  [font=\scriptsize] [align=left] {$\displaystyle v_{5}$};
\draw (355.5,112.94) node [anchor=north west][inner sep=0.75pt]  [font=\scriptsize] [align=left] {$\displaystyle v_{4}$};
\draw (234,27.18) node [anchor=north west][inner sep=0.75pt]  [font=\scriptsize] [align=left] {$\displaystyle v_{3}$};
\draw (114.5,115.11) node [anchor=north west][inner sep=0.75pt]  [font=\scriptsize] [align=left] {$\displaystyle v_{2}$};

\end{tikzpicture}
    \vspace{-.5cm}
    \caption{The cluster chart $U_1$ of $X(\sigma^4)$. One of the five possible charts given by the triangulation of the pentagon.}
    \label{fig:sigma 4}
\end{figure}
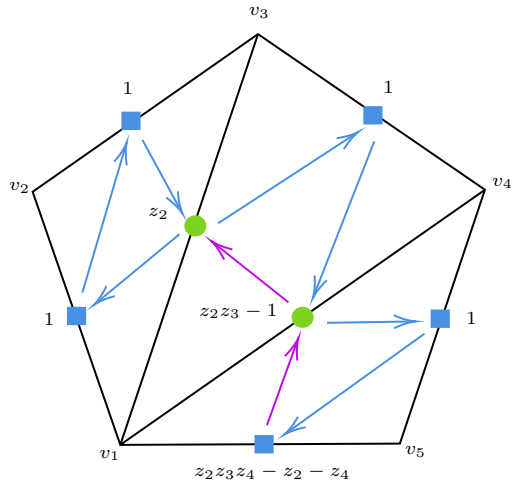
Here
\[U=\left\{w_1:=\Delta_{13}=z_2\neq0,w_2:=
\Delta_{14}=z_2z_3-1\neq0,w:=\Delta_{15}=z_2z_3z_4-z_4-z_2\neq0\right\}\]
Using the de Rham cohomology  \begin{align*}
        H^*(U)&=H^*((\C^*)^3)\\&=\left\langle 1,\frac{dw_1}{w_1},\frac{dw_2}{w_2},\frac{dw}{w},\frac{dw_1}{w_1}\wedge\frac{dw_2}{w_2},\frac{dw_1}{w_1}\wedge\frac{dw}{w},\frac{dw_2}{w_2}\wedge\frac{dw}{w},\frac{dw}{w}\wedge\frac{dw_2}{w_2}\wedge\frac{dw_1}{w_1}\right\rangle\end{align*}
        The forms which extend to $X(\sigma^4)$ are:
        \begin{itemize}
        \item $1$
        \item $\dfrac{dw}{w}=\dfrac{(z_3z_4-1)dz_2+z_2z_4dz_3+(z_2z_3-1)dz_4}{z_2z_3z_4-z_4-z_2}$
        \item $\dfrac{dw}{w}\wedge \dfrac{dw_2}{w_2}+\dfrac{dw_2}{w_2}\wedge \dfrac{dw_1}{w_1}=\dfrac{z_4dz_3\wedge dz_2+z_3dz_4\wedge dz_2+z_2dz_4\wedge dz_3}{z_2z_3z_4-z_4-z_2}$
        \item $\dfrac{dw}{w}\wedge \dfrac{dw_2}{w_2}\wedge \dfrac{dw_1}{w_1}=\dfrac{dz_4\wedge dz_3\wedge dz_2}{z_2z_3z_4-z_4-z_2}$
         \end{itemize}
         \vspace{.1cm}
         
         Therefore, $H^0(X(\sigma^4))=H^1(X(\sigma^4))=H^2(X(\sigma^4))=H^3(X(\sigma^4))=\C$ which agrees with Theorem \ref{thm: homology}.
\end{example}

\section{Performing cuts}

\subsection{Cuts for braid varieties}

In this section, we study various maps between braid varieties and positroid varieties. To work with such maps, it is useful to fix a specific isomorphism between $X(\sigma^k)$ and $\Pi_{2,k+1}^{\circ,1}$ which is given by lemmas below.

\begin{lemma}
\label{lem: matrix}
Let $M=\begin{pmatrix}
    v_1 &v_2 &\ldots &v_n
\end{pmatrix}\in \Pi_{2,n}^{\circ}$. There is a unique matrix $A\in GL(2,\C)$ such that \[AM=\begin{pmatrix} 1 &* &\ldots &0\\0 &1 &\ldots &*\end{pmatrix}=V\] where $\det A=\Delta_{12}^{-1}(M)$ and $\Delta_{ij}(V)=\Delta_{ij}(M)\cdot \det A=\dfrac{\Delta_{ij}(M)}{\Delta_{12}(M)}.$
\end{lemma}
\begin{proof}
    If $M=\begin{pmatrix}
        v_1 &v_2&\ldots&v_n
    \end{pmatrix}$, then acting on the left with the matrix $S=\begin{pmatrix}
        v_1 &v_n
    \end{pmatrix}^{-1} $, we obtain \[S\cdot M=\frac{1}{\Delta_{1n}(M)}\begin{pmatrix}
        v_n^2&-v_n^1\\-v_1^2&v_1^1
    \end{pmatrix}\begin{pmatrix}
        v_1^1 &v_2^1&\dots &v_n^1\\v_1^2 &v_2^2&\dots &v_n^2
    \end{pmatrix}=\begin{pmatrix} 1 & *&\dots &0\\0&\alpha&\dots&1\end{pmatrix}\]
    where $\alpha=\det(S)\Delta_{12}(M)=\dfrac{\Delta_{12}(M)}{\Delta_{1n}(M)}$. Now, if we act on the left by $T=\begin{pmatrix}
        1&0\\0&\alpha^{-1}
    \end{pmatrix}$, we arrive at the matrix \[T\cdot(S\cdot M)=\begin{pmatrix}
        1&0\\0&\alpha^{-1}
    \end{pmatrix}\begin{pmatrix} 1 & *&\dots &0\\0&\alpha&\dots&1\end{pmatrix}=\begin{pmatrix} 1 &* &\ldots &0\\0 &1 &\ldots &\alpha^{-1}\end{pmatrix}\]
    Let $A=T\cdot S$, then $\det A=(\det T) (\det S)=\left(\dfrac{\Delta_{1n}(M)}{\Delta_{12}(M)}\right)\left(\dfrac{1}{\Delta_{1n}(M)}\right)=\Delta_{12}^{-1}(M)$.
\end{proof}

\begin{lemma}
\label{lem: rescale}
    Given the standard form matrix 
    \begin{equation}
    \label{eqn: std form}
        V=\begin{pmatrix}
        1 &* &\dots & 0\\ 0 & 1 &\dots &*
    \end{pmatrix}
    \end{equation}
    where $\Delta_{i,i+1}\neq 0, \, \Delta_{1n}\neq0$, we may rescale the vectors $(v_3,\ldots,v_n)$ to $(v'_3,\ldots,v'_n)=(\lambda_3 v_3,\ldots,\lambda_n v_n)$ such that $\Delta_{i,i+1}'=1$. Furthermore, such $\lambda_i$ are unique.
\end{lemma}

\begin{proof}
    Let 
    \begin{align*}
        v_3'=\dfrac{v_3}{\Delta_{23}},\ 
        v_4'=\dfrac{v_4\cdot\Delta_{23}}{\Delta_{34}},\ \dots,\ 
        v_k'=v_k\displaystyle\prod_{l=2}^{k-1}\Delta_{l,l+1}^{(-1)^{k-l}}        
    \end{align*}
    Note that with the above rescaling $\Delta'_{1n}$ remains nonzero, whereas for $\Delta'_{i,i+1}$ the rescaling gives the desired result: 
    \begin{align*}
        \Delta_{i,i+1}'&=\det \begin{pmatrix}
        v_i' &v_{i+1}'
    \end{pmatrix}
    =\det \begin{pmatrix}
        v_i\displaystyle\prod_{l=2}^{i-1}\Delta_{l,l+1}^{(-1)^{i-l}} &v_{i+1}\displaystyle\prod _{l=2}^{i}\Delta_{l,l+1}^{(-1)^{i+1-l}}
    \end{pmatrix}\\
    &=\Delta_{i,i+1}\Delta_{i,i+1}^{(-1)}=1.
    \end{align*} 
\end{proof}

\begin{corollary}
Given a matrix $M\in \Pi_{2,n}^{\circ}$, we can use Lemmas \ref{lem: matrix} and \ref{lem: rescale} to change $M$ to the matrix
$$V'=\begin{pmatrix}
        1 &* &\dots & 0\\ 0 & 1 &\dots &*
    \end{pmatrix}$$
such that $V'\in \Pi_{2,n}^{\circ,1}$. Furthermore, if $M\in \Pi_{2,n}^{\circ,1}$ then $\Delta_{ij}(V')=\Delta_{ij}(M)$.
\end{corollary}

\begin{proof}
We only need to prove the last equation.
If $M\in \Pi_{2,n}^{\circ,1}$ with each $\Delta_{i,i+1}=1$, using Lemma \ref{lem: matrix} there exists a unique $V'=AM$, and $\Delta_{ij}(V')=\Delta_{ij}(M)/\Delta_{12}(M)=\Delta_{ij}(M)$.
In particular, $\Delta_{i,i+1}(V')=1$ for all $i$ and we do not require the use of Lemma \ref{lem: rescale} to rescale the vectors. 
\end{proof}

Let $\mathcal{P}$ be the $(k+1)$-gon corresponding to the braid variety $X(\sigma^{k})$. We can choose a diagonal $D_{ij}$ which cuts the polygon $\mathcal{P}$ in two pieces, a $(j-i+1)$-gon $\mathcal{P}_1(i,j)$ and a $(k-j+i+2)$-gon $\mathcal{P}_2(i,j)$. These correspond to braid varieties $X(\sigma^{j-i})$ and $X(\sigma^{k-j+i+1})$, respectively. We will refer to this procedure as a diagonal cut. If we denote $a=j-i$ and $b=k-j+i+1$ then $a+b=k+1$.

\begin{theorem}
\label{thm: one cut}
    Performing one diagonal cut on $P$ along $D_{ij}$ defines an injective map \[      \Phi_{ij}:X(\sigma^{a})\times X(\sigma^{b})\longrightarrow X(\sigma^{a+b-1})\] and its image is the open subset $\{\Delta_{ij}\neq 0\}$ in $X(\sigma^{a+b-1})$.      
\end{theorem}

\begin{proof}
  We use the isomorphism $\Pi_{2,k+1}^{\circ,1}\simeq X(\sigma^{k})$ from Theorem \ref{thm: iso}. We first describe the inverse map
  $$
  \Phi_{ij}^{-1}:\{\Delta_{ij}\neq 0\}\to X(\sigma^a)\times X(\sigma^b). 
  $$
        Let $V\in\Pi_{2,k+1}^{\circ,1}$ be a $2\times (k+1)$ matrix,  choose some $i,j$ such that $1\leq i< j\leq k+1$ where $(i,j)\ne(1,k+1)$, to perform the diagonal cut of the $(k+1)$-gon resulting in two polygons $\mathcal{P}_a$ and $\mathcal{P}_b$ where $\mathcal{P}_a$ is a $(j-i+1)$-gon and $\mathcal{P}_b$ is a $(k-j+i+2)$-gon. Assume that $\Delta_{ij}(V)\neq 0$. Then we can decompose the matrix $V$ into two matrices:
        \begin{align*}
            V_1&=\begin{pmatrix}
                v_i &\dots &v_j
            \end{pmatrix}\in \Mat(2,a+1)\\
            V_2&=\begin{pmatrix}
                v_1 &\dots &v_i &v_j &\dots &v_{k+1} 
            \end{pmatrix}\in \Mat(2,b+1)
        \end{align*}
        Let us prove that $V_1\in \Pi_{2,a+1}^{\circ,1}$. As it happens $\Delta_{m,m+1}(V_1)=\Delta_{m+i-1,m+i}(V)=1$ for $1\le m\le a$, and $\Delta_{1,a+1}(V_1)=\Delta_{ij}(V)\neq 0$.
        We use the isomorphism $\Pi_{2,a+1}^{\circ,1}\simeq X(\sigma^{a})$ from Theorem \ref{thm: iso} to obtain a point in $X(\sigma^{a})$ from $V_1$.

    Next, we study the matrix $V_2$. We have 
    $$
    \Delta_{m,m+1}(V_2)=
    \begin{cases}
    \Delta_{m,m+1}(V)=1 & \text{if}\ m<i\\
    \Delta_{ij}(V) & \text{if}\ m=i\\
    \Delta_{m+j-i-1,m+j-i}(V)=1  & \text{if}\ i<m\le k-j+i+1.
    \end{cases}
    $$
    Furthermore, $\Delta_{1,b+1}(V_2)=\Delta_{1,k+1}(V)\neq 0$, so $V_2\in \Pi_{2,b+1}^{\circ}$. We would like to use 
    Lemmas \ref{lem: matrix} and \ref{lem: rescale} to change $V_2$ to a different matrix $V'_2\in \Pi_{2,b+1}^{\circ,1}.$ We have two cases:

    {\bf Case 1:}  
    If $i=1$, then we first apply Lemma \ref{lem: matrix}. Since $S=(v_1\ v_{k+1})^{-1}$ is diagonal, we simply rescale the second row of $V_2$ by $\Delta_{1j}^{-1}$ to get $V_2$ to the form \eqref{eqn: std form}. Next, we apply Lemma \ref{lem: rescale} to rescale the vectors, and get $V_2'=(v_1,v'_2,\ldots,v'_{b+1})$ where
    $$
    v'_m=\begin{cases}
    (v^1_{m+j-2},v^2_{m+j-2}\Delta_{1j}^{-1}) & \text{if}\ m\ \text{is\ even}\\
    (v^1_{m+j-2}\Delta_{1j},v^2_{m+j-2}) & \text{if}\ m\ \text{is\ odd}.
    \end{cases}
    $$
    
    {\bf Case 2:} If $i\geq2$, then we do not need to apply Lemma \ref{lem: matrix}, we rescale the vectors $v_m$ for $m\geq j$. As a result, we get a matrix 
    $V_2'=(v_1,\ldots,v_i,v'_j,v'_{j+1},\ldots,v'_{k+1})$
    where
    $$
    v'_m=\Delta_{ij}^{(-1)^{m-j+1}}v_m.
    $$

    Now we can describe the desired map $\Phi_{ij}:X(\sigma^a)\times X(\sigma^b)\to \{\Delta_{ij}\neq 0\}$ as follows. Given two matrices $V_1\in \Pi_{2,a+1}^{\circ,1},V'_2\in \Pi_{2,b+1}^{\circ,1}$, we can read off $\Delta_{ij}(V)=\Delta_{1,a+1}(V_1)$ which is nonzero by assumption. The matrix $V'_2$ was obtained from $V_2$ above using multiplication by $\Delta_{ij}^{\pm 1}$, and hence is invertible, so given $V'_2$ and $\Delta_{ij}$ we can reconstruct $V_2$. 
    
    Now we can reconstruct $V$ by simply inserting $V_1$ into $V_2$. Note that if the vectors $v_i$ and $v_j$ from $V_1$ do not agree with the ones from $V_2$, we can always use row operations to make them agree since $\det(v_i\quad v_j)=\Delta_{ij}\neq 0$.

\end{proof}   

\begin{theorem}
The map $\Phi_{ij}$ defines a  quasi-equivalence of cluster varietes $\{\Delta_{ij}\neq 0\}\subset X(\sigma^{a+b-1})$ and $X(\sigma^{a})\times X(\sigma^{b})$. The latter has a cluster structure obtained by freezing $\Delta_{ij}$ in the cluster structure from $X(\sigma^{a+b-1})$.
\end{theorem}

\begin{proof}
We use the clusters in $X(\sigma^{a}), X(\sigma^{b})$ and $X(\sigma^{a+b-1})$ defined by the triangulation in Figure \ref{fig: 2-form cuts}. In particular, we get fan triangulations for $X(\sigma^{a}), X(\sigma^{b})$. 

By construction, all cluster variables corresponding to diagonals are multiplied by monomials in $\Delta_{ij}$, but we still need to check that the exchange ratios (as in Definition \ref{def: quasi}) are preserved. All diagonals above $\Delta_{ij}$ are unchanged, so we need to verify that the exchange ratios do not change for diagonals $\Delta_{1,m}$. For $m<i$, this is clear. For $m=i$, the exchange ratio is
$$
\frac{\Delta_{1j}}{\Delta_{ij}\Delta_{1,j-1}}=\frac{\Delta_{1j}\Delta_{ij}^{-1}}{1\cdot \Delta_{1,j-1}}.
$$
For $m=j$, the exchange ratio is
$$
\frac{\Delta_{ij}\Delta_{1,j+1}}{\Delta_{1i}}=\frac{1\cdot (\Delta_{1,j+1}\Delta_{ij})}{\Delta_{1i}}.
$$
Finally, for $m>j$ we get
$$
\frac{\Delta_{1,m+1}}{\Delta_{1,m-1}}=\frac{\Delta_{1,m+1}\Delta_{ij}^{(-1)^{m+1-j+1}}}{\Delta_{1,m-1}\Delta_{ij}^{(-1)^{m-1-j+1}}}
$$
since $m+1-j$ and $m-1-j$ have the same parity.
\end{proof}

Suppose $a+b+c-2=k$. We will study the associativity properties of our cuts along two non-intersecting diagonals $D_{ij}$ and $D_{i'j'}$, see Figure \ref{fig:two cuts}.  There are two general cases to consider when performing two cuts which we label as Type $A$ or Type $B$. The two cuts occur at $D_{ij}$ and $D_{i'j'}$ and will be denoted $\Phi_{ij}$ and $\Phi_{i'j'}$, respectively. Type $A$ cuts are diagonal cuts of the form $1\le i'\le i<j\le j'\le k+1$ given that the cuts do not degenerate to the one cut case, whereas, Type $B$ cuts are diagonal cuts of the form $1\le i<j\le i'<j'\le k+1$, see Figure \ref{fig:two cuts}.

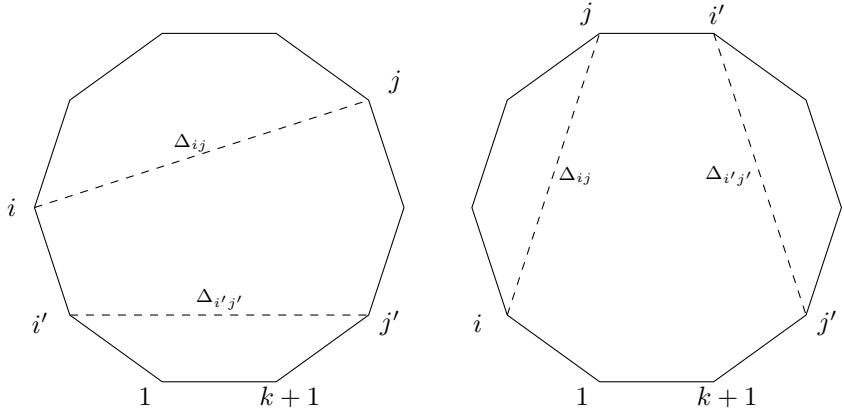
\begin{figure}
\label{fig: cut types}
    \begin{tabular}{c c c}
         \begin{tikzpicture}[scale=.75]
\begin{scope}[decoration={
    markings,
    mark=at position 0.5 with {\arrow{stealth}}}
    ] 
\draw (0,0)--(2,0)--(3.62,1.18)--(4.24,3.08)--(3.62,4.98)--(2,6.16)--(0,6.16)--(-1.62,4.98)--(-2.24,3.08)--(-1.62,1.18)--(0,0);
\draw [dashed] (-1.62,1.18)--(3.62,1.18);
\draw [dashed] (-2.24,3.08)--(3.62,4.98);
\draw (-2.14,1.04) node {$i'$};
\draw (-2.64,3.08) node {$i$};
\draw (4.09,5.3) node {$j$};
\draw (4,1.04) node {$j'$};
\draw (-0.3,-0.26) node {$1$};
\draw (2.24,-0.26) node {$k+1$};
\draw (0,6.5) node {};
\draw (0.5,4.22) node {\scriptsize $\Delta_{ij}$};
\draw (1.02,1.46) node {\scriptsize $\Delta_{i'j'}$};
\end{scope}
\end{tikzpicture}
&& \begin{tikzpicture}[scale=.75]
\begin{scope}[decoration={
    markings,
    mark=at position 0.5 with {\arrow{stealth}}}
    ] 
\draw (0,0)--(2,0)--(3.62,1.18)--(4.24,3.08)--(3.62,4.98)--(2,6.16)--(0,6.16)--(-1.62,4.98)--(-2.24,3.08)--(-1.62,1.18)--(0,0);
\draw [dashed] (-1.62,1.18)--(0,6.16);
\draw [dashed] (2,6.16)--(3.62,1.18);
\draw (-2.14,1.04) node {$i$};
\draw (-0.22,6.5) node {$j$};
\draw (2.06,6.5) node {$i'$};
\draw (4,1.04) node {$j'$};
\draw (-0.3,-0.26) node {$1$};
\draw (2.24,-0.26) node {$k+1$};
\draw (-0.42,3.64) node {\scriptsize $\Delta_{ij}$};
\draw (2.28,3.64) node {\scriptsize $\Delta_{i'j'}$};
\end{scope}
\end{tikzpicture}         
    \end{tabular}
         \caption{The possible cuts when performing two diagonal cuts, the dashed lines indicate these potential cuts. The polygon on the left depicts cuts of Type $A$ and the polygon on the right depicts cuts of Type $B$.}
         \label{fig:two cuts}
     \end{figure}

\begin{theorem}
\label{thm: type A cut}
For Type A cuts  we have a commutative diagram
            $$     
\begin{tikzcd}
	{X(\sigma^a)\times X(\sigma^b)\times X(\sigma^c)} && {X(\sigma^{a+b-1})\times X(\sigma^c)} \\
	\\
	{X(\sigma^a)\times X(\sigma^{b+c-1})} && {X(\sigma^{a+b+c-2})}
	\arrow["{\Phi_{ij}\times \Id}", from=1-1, to=1-3]
	\arrow["{\Id\times \Phi_{i'j'}}"', from=1-1, to=3-1]
	\arrow["{\Phi_{ij}}"', from=3-1, to=3-3]
	\arrow["{\Phi_{i'j'}}", from=1-3, to=3-3]
\end{tikzcd}
        $$
\end{theorem}

\begin{proof}
Let $V\in\Pi_{2,k+1}^{\circ,1}$ by Theorem \ref{thm: iso} $V$ corresponds to a point in $X(\sigma^k)$. 

For Type $A$ cuts, choose some $i,j,i',j'$ such that $1\le i'\le i<j<j'\le k$. Similar to Theorem \ref{thm: one cut} involving a single diagonal cut, we describe the inverse maps then produce the desired map. Here $a=j-i, b=j'-j+i-i'+1$ and $c=k-j'+i'+1$.
Define the matrix $V\in\Pi_{2,k+1}^{\circ,1}$ associated to $X(\sigma^k)$ as
    $$
    V=\begin{pmatrix}
        v_1 &\dots &v_{i'} &\dots &v_i &\dots &v_j &\dots &v_{j'} &\dots &v_{k+1}
    \end{pmatrix}$$
We will be dealing with minors in several different matrices, as such we will include the matrices in the notations.

    (i) First, we consider the case where we cut at along $\Delta_{ij}(V)$ then $\Delta_{i'j'}(V)$ which is described in Figure \ref{fig: Type A Delta ij} by 
    $$
    X(\sigma^{a+b+c-2})\xrightarrow{\Phi^{-1}_{ij}}X(\sigma^a)\times X(\sigma^{b+c-1})\xrightarrow{\Id\times \Phi^{-1}_{i'j'}}X(\sigma^a)\times X(\sigma^b)\times X(\sigma^c)
    $$

    By performing the initial cut $\Delta_{ij}(V)$, given by $\Phi^{-1}_{ij}:X(\sigma^{a+b+c-2})\rightarrow X(\sigma^a)\times X(\sigma^{b+c-1})$, we decompose the matrix $V$ into the two following matrices      
    $$
        V_1=\begin{pmatrix}
            v_i &\dots &v_j
        \end{pmatrix}\in \Mat(2,a+1)
    $$
    $$
        V_2=\begin{pmatrix}
            v_1 &\dots &v_{i'} &\dots &v_i &v_j &\dots &v_{j'} &\dots &v_{k+1}
        \end{pmatrix}\in \Mat(2,b+c)
    $$

    Similar to the argument in Theorem \ref{thm: one cut}, $\Delta_{ij}(V)\ne 0$ and we find that $V_1\in\Pi_{2,a+1}^{\circ,1}\simeq X(\sigma^a)$. Here, the rescaling of vectors $v_m$ for $m\ge j$ in 
    $$V_3=\begin{pmatrix} v_1 &\dots &v_{i'} &\dots &v_i &v'_j &\dots &v'_{j'} &\dots &v'_{k+1}
    \end{pmatrix}
    $$ 
    is given by
    \begin{equation}
    \label{eq: v' j to j'}
        v_m'=v_m\Delta_{ij}(V)^{(-1)^{m-j+1}}
    \end{equation}
        
    Therefore, $V_3\in\Pi_{2,b+c}^{\circ,1}\simeq X(\sigma^{b+c-1})$ and $\Phi_{ij}^{-1}$ is well-defined. Note that during the rescaling of the matrix $V_2$ into $V_3$ the minors of $V_3$ also experience rescaling by a factor of $\Delta_{ij}(V)$, hence given that $v'_{j'}=v_{j'}\Delta_{ij}(V)^{(-1)^{j'-j+1}}$
    \begin{equation}
    \label{eq: scaled delta i'j'}
        \Delta_{i'j'}(V_3)=\Delta_{i'j'}(V)\Delta_{ij}(V)^{(-1)^{j'-j+1}}
    \end{equation}

    Now we perform the second cut $\Delta_{i'j'}(V)$ given by the map $X(\sigma^a)\times X(\sigma^{b+c-1})\xrightarrow{\Id\times \Phi_{i'j'}} X(\sigma^a)\times X(\sigma^b)\times X(\sigma^c)$. The matrix $V_1$ remains unchanged whereas $V_3$ decomposes into
    $$
        V_4=\begin{pmatrix}
            v_{i'} &\dots &v_i &v'_j &\dots &v'_{j'}
        \end{pmatrix}\in\Mat(2,b+1)
    $$
    $$
        V_5=\begin{pmatrix}
            v_1 &\dots &v_{i'} &v'_{j'} &\dots &v'_{k+1}
        \end{pmatrix}\in\Mat(2,c+1)
    $$

    By the rescaling of matrix $V_3$ in the previous cutting and $\Delta_{i'j'}(V_3)\ne 0$, then $V_4\in \Pi_{2,b+1}^{\circ,1}\simeq X(\sigma^b)$. After performing the second cut there is again a rescaling, this time of the matrix $V_5$ which is given by the new matrix 
    $$
    V_6=\begin{pmatrix}
        v_1 &\dots &v_{i'} &v_{j'}'' &\dots &v_{k+1}''
    \end{pmatrix}
    $$ 
    where for $m\ge j'$ the vectors are $$v_m''=v'_m\Delta_{i'j'}(V_3)^{(-1)^{m-j'+1}}=v_m\Delta_{ij}(V)^{(-1)^{m-j+1}}\Delta_{i'j'}(V_3)^{(-1)^{m-j'+1}}.$$ 
    Given that
    \begin{align*}
        \Delta_{i'j'}(V_3)^{(-1)^{m-j'+1}}&=\Delta_{i'j'}(V)^{(-1)^{m-j'+1}}\Delta_{ij}(V)^{(-1)^{j'-j+1}(-1)^{m-j'+1}}\\&=\Delta_{i'j'}(V)^{(-1)^{m-j'+1}}\Delta_{ij}(V)^{(-1)^{m-j}}
    \end{align*} 
    and $(-1)^{m-j+1}+(-1)^{m-j}=0$ we conclude that
    \begin{equation}
    \label{eq: v'' larger than j'}
    v_m''=v_m\Delta_{ij}(V)^{(-1)^{m-j+1}}\Delta_{i'j'}(V)^{(-1)^{m-j'+1}}\Delta_{ij}(V)^{(-1)^{m-j}}=v_m\Delta_{i'j'}(V)^{(-1)^{m-j'+1}}.
    \end{equation}
    As such $V_6\in\Pi_{2,c+1}^{\circ,1}\simeq X(\sigma^c)$. This concludes the construction of the inverse map. 

    To construct the desired map $$\Phi_{ij}\circ(\Id\times\Phi_{i'j'}):X(\sigma^a)\times X(\sigma^b)\times X(\sigma^c)\rightarrow X(\sigma^{a+b+c-2})$$
    We reconstruct $V$ by taking $V_1\in\Pi_{2,a+1}^{\circ,1},\, V_4\in\Pi_{2,b+1}^{\circ,1},\,V_6\in\Pi_{2,c+1}^{\circ,1}$. We can read off $\Delta_{i'j'}(V)=\Delta_{1,b+1}(V_4)$ which is nonzero by assumption. The matrix $V_5$ is obtained from $V_6$ by multiplication of $\Delta_{i'j'}(V)^{\pm1}$ to the vectors $v_l$ for $l\ge i'+1$, which is well-defined since $\Delta_{i'j'}(V)$ is invertible. We reconstruct the matrix $V_3\in \Mat(2,b+c)$ by inserting the matrix $V_4$ into $V_5$ in the appropriate location. Furthermore, $V_3\in\Pi_{2,b+c}^{\circ,1}\simeq X(\sigma^{b+c-1})$ by construction. This concludes the construction of the map 
    $$
    X(\sigma^a)\times X(\sigma^b)\times X(\sigma^c)\xrightarrow{\Id\times\Phi_{i'j'}}X(\sigma^a)\times X(\sigma^{b+c-1})
    $$ 

    Continuing the construction of the desired map, we read off $\Delta_{ij}(V)=\Delta_{1,a+1}(V_1)$ which is again nonzero by assumption. The matrix $V_2$ is obtained from $V_3$ by multiplication of $\Delta_{ij}(V)^{\pm1}$ to the vectors $v_l$ for $l\ge i+1$. We reconstruct $V$ by inserting $V_1$ into $V_2$ at the appropriate location, completing the construction of the map 
    $$
    X(\sigma^a)\times X(\sigma^{b+c-1})\xrightarrow{\Phi_{ij}}X(\sigma^{a+b+c-2})
    $$ and producing the desired map.

    (ii) Now, for the case where we cut along $\Delta_{i'j'}(V)$ then $\Delta_{ij}(V)$, described in Figure \ref{fig: Type_A_Delta_ij2} by 
    $$
    X(\sigma^a)\times X(\sigma^b)\times X(\sigma^c)\xrightarrow{\Phi_{i'j'}^{-1}} X(\sigma^{a+b-1})\times X(\sigma^c)\xrightarrow{\Phi_{ij}^{-1}\times\Id} X(\sigma^{a+b+c-2}).$$

    Perform the initial cut $\Delta_{i'j'}(V)$, to decompose $V$ into the matrices
    $$
        W_1=\begin{pmatrix}
            v_1 &\dots &v_{i'} &v_{j'} &\dots &v_{k+1}
        \end{pmatrix}\in\Mat(2,c+1)
    $$
    $$
        W_2=\begin{pmatrix}
            v_{i'} &\dots &v_i &\dots &v_j &\dots &v_{j'}
        \end{pmatrix}\in\Mat(2,a+b)
    $$

    By the same argument as in Theorem \ref{thm: one cut} $\Delta_{i'j'}\ne0$ and $W_2\in\Pi_{2,a+b}^{\circ,1}\simeq X(\sigma^{a+b-1})$. Now, the matrix $W_1$ requires rescaling of the vectors $v_m$ for $m\ge j'$, producing the matrix 
    $$W_3=\begin{pmatrix}
        v_1 &\dots &v_{i'} &v'_{j'} &\dots &v'_{k+1}
    \end{pmatrix}
    $$ 
    here 
    $$
    v'_m=v_m\Delta_{i'j'}(V)^{(-1)^{m-j'+1}}
    $$    
    which is in agreement with \eqref{eq: v'' larger than j'}. Hence, $W_3\in\Pi_{2,c+1}^{\circ,1}\simeq X(\sigma^c)$.

    We perform the second cut $\Delta_{ij}(V)$, which separates $W_2$ into 
    $$
    W_4=\begin{pmatrix}
            v_{i'} &\dots &v_i &v_j &\dots &v_{j'}
        \end{pmatrix}\in \Mat(2,b+1)
    $$
    $$
        W_5=\begin{pmatrix}
            v_i &\dots &v_j
        \end{pmatrix}\in\Mat(2,a+1)
    $$
    In this case, $W_5\in\Pi_{2,a+1}^{\circ,1}\simeq X(\sigma^a)$, whereas the matrix $W_4\in\Pi_{2,b+1}^{\circ}$ requires a rescaling for the vectors $v_m$ for $j\le m\le j'$. Let $$W_6=\begin{pmatrix}
        v_{i'} &\dots &v_i&v''_j&\dots &v''_{j'}
    \end{pmatrix}
    $$ with the vectors
    $$
    v''_m=v_m\Delta_{ij}(V)^{m-j+1}
    $$
    which agrees with \eqref{eq: v' j to j'}. Therefore, $W_6\in\Pi_{2,b+1}^{\circ,1}\simeq X(\sigma^b)$, completing the construction of the inverse maps.

    Finally, we construct the desired map 
    $$
    \Phi_{i'j'}\circ(\Phi_{ij}\times \Id):X(\sigma^a)\times X(\sigma^b)\times X(\sigma^c)\rightarrow X(\sigma^{a+b+c-2})
    $$
    We reconstruct $V$ by taking $W_5\in\Pi_{2,a+1}^{\circ,1},\,W_6\in\Pi_{2,b+1}^{\circ,1},\,W_3\in\Pi_{2,c+1}^{\circ,1}$. We read off $\Delta_{ij}(V)=\Delta_{1,a+1}(W_5)$ which is nonzero by assumption. The matrix $W_4$ is recovered from $W_6$ by multiplication of $\Delta_{ij}^{\pm1}$ to the vectors $v_l$ for $l\ge i-i'+1$, which is well-defined since $\Delta_{ij}$ is invertible. We reconstruct $W_2\in\Pi_{2,a+b}^{\circ,1}\simeq X(\sigma^{a+b-1})$ by inserting the matrix $W_5$ into $W_4$ in the appropriate position. Concluding the construction of the map
    $$
    X(\sigma^a)\times X(\sigma^b)\times X(\sigma^c)\xrightarrow{\Phi_{ij}\times\Id}X(\sigma^{a+b-1})\times X(\sigma^c)$$
    To complete the construction, we read off $\Delta_{i'j'}(v)=\Delta_{1,c+1}(W_3)$ which is also nonzero by construction. The matrix $W_1\in\Pi_{1,c+1}^{\circ,1}\simeq X(\sigma^c)$ is recovered from the matrix $W_3$ by multiplication of $\Delta_{i'j'}(V)^{\pm1}$ to the vectors $v_l$ for $l\ge i'+1$. We reconstruct $V$ by inserting $W_1$ into $W_2$ at the appropriate location, concluding the construction of the map 
    $$
    X(\sigma^{a+b-1})\times X(\sigma^c)\xrightarrow{\Phi_{i'j'}}X(\sigma^{a+b+c-2})$$
    which produces the desired map showing associativity of Type A cuts.
        
\end{proof}

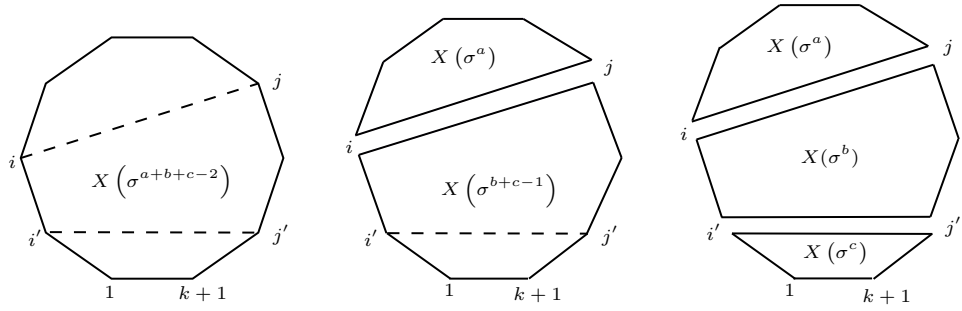
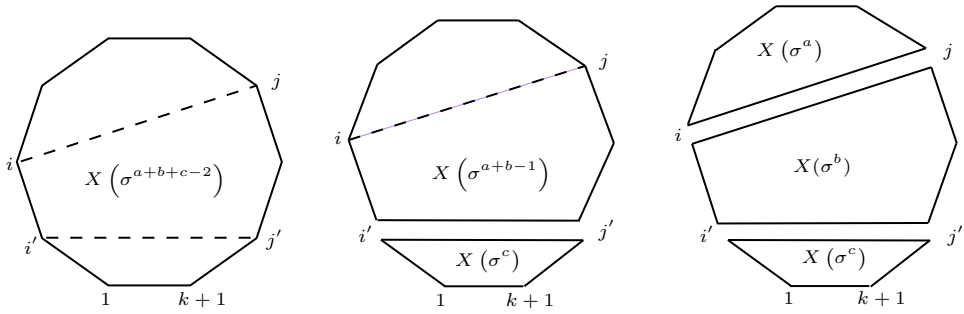
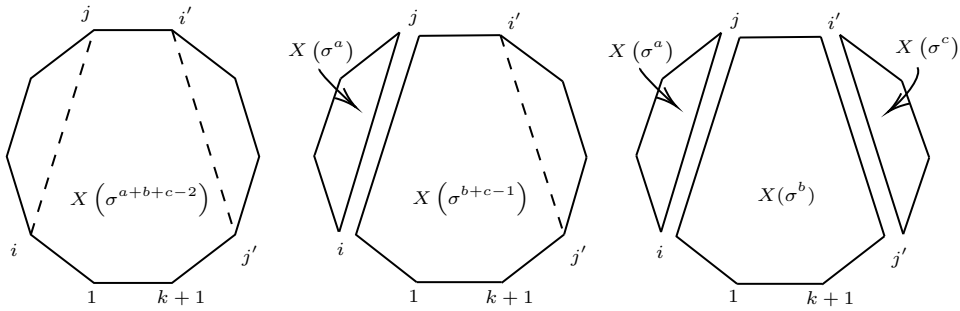
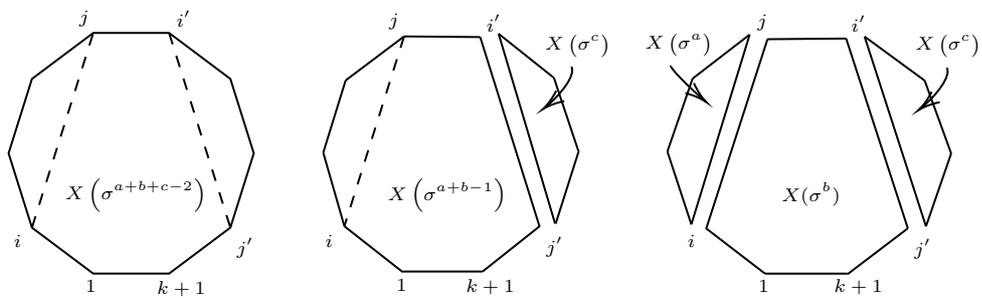
\begin{figure}
        \begin{subfigure}{\textwidth}
            \centering
            \tikzset{every picture/.style={line width=0.75pt}} 

\begin{tikzpicture}[x=0.75pt,y=0.75pt,yscale=-1,xscale=1]

\draw   (143.57,81.23) -- (131.06,118.74) -- (98.3,141.92) -- (57.81,141.92) -- (25.05,118.74) -- (12.54,81.23) -- (25.05,43.73) -- (57.81,20.55) -- (98.3,20.55) -- (131.06,43.73) -- cycle ;
\draw  [dash pattern={on 4.5pt off 4.5pt}]  (12.54,81.23) -- (131.06,43.73) ;
\draw  [dash pattern={on 4.5pt off 4.5pt}]  (27.07,118.49) -- (131.06,118.74) ;
\draw [fill={rgb, 255:red, 183; green, 134; blue, 225 }  ,fill opacity=1 ]   (179.45,69.86) -- (297.22,31.9) ;
\draw [fill={rgb, 255:red, 183; green, 134; blue, 225 }  ,fill opacity=1 ]   (179.7,69.86) -- (193.41,32.89) ;
\draw [fill={rgb, 255:red, 183; green, 134; blue, 225 }  ,fill opacity=1 ]   (193.41,32.89) -- (223.29,11.23) ;
\draw [fill={rgb, 255:red, 183; green, 134; blue, 225 }  ,fill opacity=1 ]   (223.29,11.23) -- (263.29,11.23) ;
\draw [fill={rgb, 255:red, 183; green, 134; blue, 225 }  ,fill opacity=1 ]   (263.29,11.23) -- (297.48,31.9) ;
\draw    (349.68,71.91) -- (362.34,111.03) ;
\draw    (349.68,71.91) -- (467.93,34.26) ;
\draw    (467.93,34.26) -- (480.59,71.66) ;
\draw    (480.33,71.66) -- (466.41,110.79) ;
\draw    (362.09,111.03) -- (466.41,110.79) ;
\draw    (195.22,119.01) -- (226.87,141.89) ;
\draw  [dash pattern={on 4.5pt off 4.5pt}]  (195.48,119.01) -- (208.65,119.04) -- (295.49,119.26) ;
\draw    (226.37,141.89) -- (265.61,141.89) ;
\draw    (266.12,141.89) -- (295.24,119.26) ;
\draw    (194.97,119.01) -- (181.17,79.64) ;
\draw [fill={rgb, 255:red, 183; green, 134; blue, 225 }  ,fill opacity=1 ]   (181.17,79.64) -- (298.15,43.22) ;
\draw [fill={rgb, 255:red, 183; green, 134; blue, 225 }  ,fill opacity=1 ]   (298.15,43.22) -- (312.33,81.12) ;
\draw    (294.99,119.26) -- (312.33,81.12) ;
\draw    (347.37,62.92) -- (465.14,24.97) ;
\draw    (347.62,62.37) -- (361.33,25.4) ;
\draw    (361.33,25.95) -- (391.21,4.3) ;
\draw    (391.21,4.3) -- (431.21,4.3) ;
\draw    (431.21,4.3) -- (465.4,24.97) ;
\draw    (367.4,119.15) -- (399.05,142.04) ;
\draw    (367.4,119.15) -- (380.57,119.18) -- (467.42,119.4) ;
\draw    (398.04,141.79) -- (437.29,141.79) ;
\draw    (438.05,142.04) -- (467.17,119.4) ;

\draw (5,78.65) node [anchor=north west][inner sep=0.75pt]  [font=\scriptsize] [align=left] {$\displaystyle i$};
\draw (173.04,71.19) node [anchor=north west][inner sep=0.75pt]  [font=\scriptsize] [align=left] {$\displaystyle i$};
\draw (14.76,115.8) node [anchor=north west][inner sep=0.75pt]  [font=\scriptsize] [align=left] {$\displaystyle i'$};
\draw (339.97,64.64) node [anchor=north west][inner sep=0.75pt]  [font=\scriptsize] [align=left] {$\displaystyle i$};
\draw (353.24,112.72) node [anchor=north west][inner sep=0.75pt]  [font=\scriptsize] [align=left] {$\displaystyle i'$};
\draw (182.4,115.77) node [anchor=north west][inner sep=0.75pt]  [font=\scriptsize] [align=left] {$\displaystyle i'$};
\draw (136.86,37.13) node [anchor=north west][inner sep=0.75pt]  [font=\scriptsize] [align=left] {$\displaystyle j$};
\draw (302.17,30.11) node [anchor=north west][inner sep=0.75pt]  [font=\scriptsize] [align=left] {$\displaystyle j$};
\draw (471.44,20.06) node [anchor=north west][inner sep=0.75pt]  [font=\scriptsize] [align=left] {$\displaystyle j$};
\draw (136.86,113.18) node [anchor=north west][inner sep=0.75pt]  [font=\scriptsize] [align=left] {$\displaystyle j'$};
\draw (301,113.59) node [anchor=north west][inner sep=0.75pt]  [font=\scriptsize] [align=left] {$\displaystyle j'$};
\draw (472.22,109.22) node [anchor=north west][inner sep=0.75pt]  [font=\scriptsize] [align=left] {$\displaystyle j'$};
\draw (46.2,83.45) node [anchor=north west][inner sep=0.75pt]  [font=\scriptsize] [align=left] {$\displaystyle X\left( \sigma ^{a+b+c-2}\right)$};
\draw (215.48,24.05) node [anchor=north west][inner sep=0.75pt]  [font=\scriptsize] [align=left] {$\displaystyle X\left( \sigma ^{a}\right)$};
\draw (383.09,18.37) node [anchor=north west][inner sep=0.75pt]  [font=\scriptsize] [align=left] {$\displaystyle X\left( \sigma ^{a}\right)$};
\draw (220.77,87.79) node [anchor=north west][inner sep=0.75pt]  [font=\scriptsize] [align=left] {$\displaystyle X\left( \sigma ^{b+c-1}\right)$};
\draw (400.3,72.8) node [anchor=north west][inner sep=0.75pt]  [font=\scriptsize] [align=left] {$\displaystyle X( \sigma ^{b}$)};
\draw (401.33,122.11) node [anchor=north west][inner sep=0.75pt]  [font=\scriptsize] [align=left] {$\displaystyle X\left( \sigma ^{c}\right)$};
\draw (392.01,142.88) node [anchor=north west][inner sep=0.75pt]  [font=\scriptsize] [align=left] {$\displaystyle 1$};
\draw (222.24,143.31) node [anchor=north west][inner sep=0.75pt]  [font=\scriptsize] [align=left] {$\displaystyle 1$};
\draw (52.86,144.21) node [anchor=north west][inner sep=0.75pt]  [font=\scriptsize] [align=left] {$\displaystyle 1$};
\draw (430.97,145.5) node [anchor=north west][inner sep=0.75pt]  [font=\scriptsize] [align=left] {$\displaystyle k+1$};
\draw (256.88,144.62) node [anchor=north west][inner sep=0.75pt]  [font=\scriptsize] [align=left] {$\displaystyle k+1$};
\draw (89.67,144.21) node [anchor=north west][inner sep=0.75pt]  [font=\scriptsize] [align=left] {$\displaystyle k+1$};

\end{tikzpicture}
            \vspace{-.5cm}
            \caption{Type A: Initial cut at $\Delta_{ij}$ followed by $\Delta_{i'j'}$.}
            \vspace{1cm}
            \label{fig: Type A Delta ij}
        \end{subfigure}
        \begin{subfigure}{\textwidth}
            \centering
            \tikzset{every picture/.style={line width=0.75pt}} 

\begin{tikzpicture}[x=0.75pt,y=0.75pt,yscale=-1,xscale=1]

\draw   (148.91,89.73) -- (136.28,128.13) -- (103.22,151.86) -- (62.36,151.86) -- (29.3,128.13) -- (16.68,89.73) -- (29.3,51.34) -- (62.36,27.61) -- (103.22,27.61) -- (136.28,51.34) -- cycle ;
\draw  [dash pattern={on 4.5pt off 4.5pt}]  (18.72,89.48) -- (136.28,51.34) ;
\draw  [dash pattern={on 4.5pt off 4.5pt}]  (31.35,127.88) -- (138.33,127.88) ;
\draw [fill={rgb, 255:red, 183; green, 134; blue, 225 }  ,fill opacity=1 ]   (182.21,78.74) -- (196.53,40.72) ;
\draw [fill={rgb, 255:red, 183; green, 134; blue, 225 }  ,fill opacity=1 ]   (196.53,40.72) -- (226.68,18.56) ;
\draw [fill={rgb, 255:red, 183; green, 134; blue, 225 }  ,fill opacity=1 ]   (226.68,18.56) -- (267.05,18.56) ;
\draw [fill={rgb, 255:red, 183; green, 134; blue, 225 }  ,fill opacity=1 ]   (267.05,18.56) -- (300.26,41.46) ;
\draw    (353.59,80.66) -- (366.36,120.71) ;
\draw    (353.59,80.66) -- (472.92,42.12) ;
\draw    (472.92,42.12) -- (485.69,80.4) ;
\draw    (485.44,80.4) -- (471.38,120.46) ;
\draw    (366.11,120.71) -- (471.38,120.46) ;
\draw    (198.36,128.89) -- (230.3,152.31) ;
\draw    (196.65,119.04) -- (209.94,119.08) -- (297.58,119.29) ;
\draw    (229.79,152.31) -- (269.39,152.31) ;
\draw    (269.9,152.31) -- (299.29,129.14) ;
\draw    (196.13,119.04) -- (182.21,78.74) ;
\draw [fill={rgb, 255:red, 183; green, 134; blue, 225 }  ,fill opacity=1 ] [dash pattern={on 4.5pt off 4.5pt}]  (182.21,78.74) -- (300.26,41.46) ;
\draw [fill={rgb, 255:red, 183; green, 134; blue, 225 }  ,fill opacity=1 ]   (300.26,41.46) -- (314.57,80.25) ;
\draw    (297.07,119.29) -- (314.57,80.25) ;
\draw    (351.26,71.46) -- (470.11,32.61) ;
\draw    (351.51,70.9) -- (365.34,33.05) ;
\draw    (365.34,33.61) -- (395.49,11.45) ;
\draw    (395.49,11.45) -- (435.87,11.45) ;
\draw    (435.87,11.45) -- (470.36,32.61) ;
\draw    (371.47,129.02) -- (403.41,152.45) ;
\draw    (371.47,129.02) -- (384.76,129.05) -- (472.41,129.27) ;
\draw    (402.39,152.19) -- (442,152.19) ;
\draw    (442.76,152.45) -- (472.15,129.27) ;
\draw    (299.29,129.14) -- (198.36,128.89) ;

\draw (343.28,71.38) node [anchor=north west][inner sep=0.75pt]  [font=\scriptsize] [align=left] {$\displaystyle i$};
\draw (173.49,73.17) node [anchor=north west][inner sep=0.75pt]  [font=\scriptsize] [align=left] {$\displaystyle i$};
\draw (9.75,86.14) node [anchor=north west][inner sep=0.75pt]  [font=\scriptsize] [align=left] {$\displaystyle i$};
\draw (355.88,120.15) node [anchor=north west][inner sep=0.75pt]  [font=\scriptsize] [align=left] {$\displaystyle i'$};
\draw (18.8,125.96) node [anchor=north west][inner sep=0.75pt]  [font=\scriptsize] [align=left] {$\displaystyle i'$};
\draw (185.7,120.15) node [anchor=north west][inner sep=0.75pt]  [font=\scriptsize] [align=left] {$\displaystyle i'$};
\draw (477.53,28.87) node [anchor=north west][inner sep=0.75pt]  [font=\scriptsize] [align=left] {$\displaystyle j$};
\draw (305.77,33.79) node [anchor=north west][inner sep=0.75pt]  [font=\scriptsize] [align=left] {$\displaystyle j$};
\draw (142.42,43.63) node [anchor=north west][inner sep=0.75pt]  [font=\scriptsize] [align=left] {$\displaystyle j$};
\draw (479.5,120.15) node [anchor=north west][inner sep=0.75pt]  [font=\scriptsize] [align=left] {$\displaystyle j'$};
\draw (305.38,117.46) node [anchor=north west][inner sep=0.75pt]  [font=\scriptsize] [align=left] {$\displaystyle j'$};
\draw (139.66,122.38) node [anchor=north west][inner sep=0.75pt]  [font=\scriptsize] [align=left] {$\displaystyle j'$};
\draw (48.8,89.26) node [anchor=north west][inner sep=0.75pt]  [font=\scriptsize] [align=left] {$\displaystyle X\left( \sigma ^{a+b+c-2}\right)$};
\draw (233.96,132.42) node [anchor=north west][inner sep=0.75pt]  [font=\scriptsize] [align=left] {$\displaystyle X\left( \sigma ^{c}\right)$};
\draw (222.26,86.18) node [anchor=north west][inner sep=0.75pt]  [font=\scriptsize] [align=left] {$\displaystyle X\left( \sigma ^{a+b-1}\right)$};
\draw (384.2,26.89) node [anchor=north west][inner sep=0.75pt]  [font=\scriptsize] [align=left] {$\displaystyle X\left( \sigma ^{a}\right)$};
\draw (406.77,131.3) node [anchor=north west][inner sep=0.75pt]  [font=\scriptsize] [align=left] {$\displaystyle X\left( \sigma ^{c}\right)$};
\draw (402.46,84.15) node [anchor=north west][inner sep=0.75pt]  [font=\scriptsize] [align=left] {$\displaystyle X( \sigma ^{b}$)};
\draw (56.96,154.15) node [anchor=north west][inner sep=0.75pt]  [font=\scriptsize] [align=left] {$\displaystyle 1$};
\draw (397.76,153.7) node [anchor=north west][inner sep=0.75pt]  [font=\scriptsize] [align=left] {$\displaystyle 1$};
\draw (224.07,154.6) node [anchor=north west][inner sep=0.75pt]  [font=\scriptsize] [align=left] {$\displaystyle 1$};
\draw (434.53,153.26) node [anchor=north west][inner sep=0.75pt]  [font=\scriptsize] [align=left] {$\displaystyle k+1$};
\draw (259.1,154.15) node [anchor=north west][inner sep=0.75pt]  [font=\scriptsize] [align=left] {$\displaystyle k+1$};
\draw (95.48,153.7) node [anchor=north west][inner sep=0.75pt]  [font=\scriptsize] [align=left] {$\displaystyle k+1$};

\end{tikzpicture}
            \vspace{-.5cm}
            \caption{Type A: Initial cut at $\Delta_{i'j'}$ followed by $\Delta_{ij}$.}
            \vspace{1cm}
            \label{fig: Type_A_Delta_ij2}
        \end{subfigure}
        \begin{subfigure}{\textwidth}
            \centering
            \tikzset{every picture/.style={line width=0.75pt}} 

\begin{tikzpicture}[x=0.75pt,y=0.75pt,yscale=-1,xscale=1]

\draw   (132.32,85.41) -- (120.29,124.71) -- (88.8,149) -- (49.87,149) -- (18.38,124.71) -- (6.35,85.41) -- (18.38,46.1) -- (49.87,21.81) -- (88.8,21.81) -- (120.29,46.1) -- cycle ;
\draw  [dash pattern={on 4.5pt off 4.5pt}]  (18.38,124.71) -- (49.87,21.81) ;
\draw  [dash pattern={on 4.5pt off 4.5pt}]  (88.8,21.81) -- (120.29,124.71) ;
\draw    (332.79,123.42) -- (320.45,84.89) ;
\draw    (320.45,84.89) -- (333.42,46.37) ;
\draw    (333.42,46.37) -- (362.84,23.11) ;
\draw    (362.84,23.11) -- (332.79,123.42) ;
\draw    (453.64,124.46) -- (466.62,85.93) ;
\draw    (466.62,85.93) -- (452.97,47.41) ;
\draw    (452.97,47.41) -- (422.01,24.15) ;
\draw    (422.01,24.15) -- (453.64,124.46) ;
\draw    (171.96,123.27) -- (159.62,84.75) ;
\draw    (159.93,84.75) -- (172.9,46.22) ;
\draw    (172.9,46.22) -- (202.33,22.96) ;
\draw    (202.33,22.96) -- (172.27,123.27) ;
\draw    (284.42,124.66) -- (295.66,85.79) ;
\draw    (295.66,85.79) -- (283.76,47.61) ;
\draw    (283.76,47.61) -- (252.64,24.35) ;
\draw  [dash pattern={on 4.5pt off 4.5pt}]  (252.64,24.35) -- (284.42,124.66) ;
\draw    (284.42,124.66) -- (253.27,148.96) ;
\draw    (253.27,148.61) -- (210.87,148.61) ;
\draw    (210.87,148.61) -- (180.5,124.66) ;
\draw    (180.5,125.01) -- (212.14,24.7) ;
\draw    (212.14,24.7) -- (252.64,24.35) ;
\draw    (412.52,25.19) -- (444.3,125.5) ;
\draw    (444.3,125.5) -- (413.15,149.8) ;
\draw    (413.15,149.45) -- (370.75,149.45) ;
\draw    (370.75,149.45) -- (340.38,125.5) ;
\draw    (340.38,125.85) -- (372.02,25.54) ;
\draw    (372.02,25.54) -- (412.52,25.19) ;
\draw    (162.41,41.39) .. controls (161.28,42.23) and (171.79,52.86) .. (179.1,59.37) ;
\draw [shift={(180.57,60.66)}, rotate = 220.82] [color={rgb, 255:red, 0; green, 0; blue, 0 }  ][line width=0.75]    (10.93,-3.29) .. controls (6.95,-1.4) and (3.31,-0.3) .. (0,0) .. controls (3.31,0.3) and (6.95,1.4) .. (10.93,3.29)   ;
\draw    (322.41,40.99) .. controls (321.28,41.83) and (331.79,52.46) .. (339.1,58.97) ;
\draw [shift={(340.57,60.26)}, rotate = 220.82] [color={rgb, 255:red, 0; green, 0; blue, 0 }  ][line width=0.75]    (10.93,-3.29) .. controls (6.95,-1.4) and (3.31,-0.3) .. (0,0) .. controls (3.31,0.3) and (6.95,1.4) .. (10.93,3.29)   ;
\draw    (460.69,40.54) .. controls (463.06,44.75) and (455.63,54.91) .. (447.6,61.69) ;
\draw [shift={(446.06,62.95)}, rotate = 321.91] [color={rgb, 255:red, 0; green, 0; blue, 0 }  ][line width=0.75]    (10.93,-3.29) .. controls (6.95,-1.4) and (3.31,-0.3) .. (0,0) .. controls (3.31,0.3) and (6.95,1.4) .. (10.93,3.29)   ;

\draw (414.31,11.34) node [anchor=north west][inner sep=0.75pt]  [font=\scriptsize] [align=left] {$\displaystyle i'$};
\draw (254.66,11.29) node [anchor=north west][inner sep=0.75pt]  [font=\scriptsize] [align=left] {$\displaystyle i'$};
\draw (90.62,8.7) node [anchor=north west][inner sep=0.75pt]  [font=\scriptsize] [align=left] {$\displaystyle i'$};
\draw (446.21,129.47) node [anchor=north west][inner sep=0.75pt]  [font=\scriptsize] [align=left] {$\displaystyle j'$};
\draw (286.33,128.63) node [anchor=north west][inner sep=0.75pt]  [font=\scriptsize] [align=left] {$\displaystyle j'$};
\draw (122.2,128.68) node [anchor=north west][inner sep=0.75pt]  [font=\scriptsize] [align=left] {$\displaystyle j'$};
\draw (328.79,128.77) node [anchor=north west][inner sep=0.75pt]  [font=\scriptsize] [align=left] {$\displaystyle i$};
\draw (169.53,127.38) node [anchor=north west][inner sep=0.75pt]  [font=\scriptsize] [align=left] {$\displaystyle i$};
\draw (6.56,127.92) node [anchor=north west][inner sep=0.75pt]  [font=\scriptsize] [align=left] {$\displaystyle i$};
\draw (365.95,11.56) node [anchor=north west][inner sep=0.75pt]  [font=\scriptsize] [align=left] {$\displaystyle j$};
\draw (205.51,10.17) node [anchor=north west][inner sep=0.75pt]  [font=\scriptsize] [align=left] {$\displaystyle j$};
\draw (42.14,8.02) node [anchor=north west][inner sep=0.75pt]  [font=\scriptsize] [align=left] {$\displaystyle j$};
\draw (364.82,151.18) node [anchor=north west][inner sep=0.75pt]  [font=\scriptsize] [align=left] {$\displaystyle 1$};
\draw (205.23,151.13) node [anchor=north west][inner sep=0.75pt]  [font=\scriptsize] [align=left] {$\displaystyle 1$};
\draw (44.63,151.67) node [anchor=north west][inner sep=0.75pt]  [font=\scriptsize] [align=left] {$\displaystyle 1$};
\draw (79.81,151.23) node [anchor=north west][inner sep=0.75pt]  [font=\scriptsize] [align=left] {$\displaystyle k+1$};
\draw (403.63,151.63) node [anchor=north west][inner sep=0.75pt]  [font=\scriptsize] [align=left] {$\displaystyle k+1$};
\draw (243.64,151.13) node [anchor=north west][inner sep=0.75pt]  [font=\scriptsize] [align=left] {$\displaystyle k+1$};
\draw (36.34,97.45) node [anchor=north west][inner sep=0.75pt]  [font=\scriptsize] [align=left] {$\displaystyle X\left( \sigma ^{a+b+c-2}\right)$};
\draw (145.25,25.61) node [anchor=north west][inner sep=0.75pt]  [font=\scriptsize] [align=left] {$\displaystyle X\left( \sigma ^{a}\right)$};
\draw (304.75,25.75) node [anchor=north west][inner sep=0.75pt]  [font=\scriptsize] [align=left] {$\displaystyle X\left( \sigma ^{a}\right)$};
\draw (207.22,97) node [anchor=north west][inner sep=0.75pt]  [font=\scriptsize] [align=left] {$\displaystyle X\left( \sigma ^{b+c-1}\right)$};
\draw (379.39,96.73) node [anchor=north west][inner sep=0.75pt]  [font=\scriptsize] [align=left] {$\displaystyle X( \sigma ^{b}$)};
\draw (448.17,23.67) node [anchor=north west][inner sep=0.75pt]  [font=\scriptsize] [align=left] {$\displaystyle X\left( \sigma ^{c}\right)$};

\end{tikzpicture}
            \vspace{-.5cm}
            \caption{Type B: Initial cut at $\Delta_{ij}$ followed by $\Delta_{i'j'}$.}
            \vspace{1cm}
            \label{fig: Type B Delta ij}
        \end{subfigure}
        \begin{subfigure}{\textwidth}
            \centering
            \tikzset{every picture/.style={line width=0.75pt}} 

\begin{tikzpicture}[x=0.75pt,y=0.75pt,yscale=-1,xscale=1]

\draw  [line width=0.75]  (128.14,76.41) -- (116.46,113.82) -- (85.87,136.94) -- (48.05,136.94) -- (17.46,113.82) -- (5.78,76.41) -- (17.46,38.99) -- (48.05,15.87) -- (85.87,15.87) -- (116.46,38.99) -- cycle ;
\draw [line width=0.75]  [dash pattern={on 4.5pt off 4.5pt}]  (17.46,113.82) -- (48.05,15.87) ;
\draw [line width=0.75]  [dash pattern={on 4.5pt off 4.5pt}]  (85.87,15.87) -- (116.46,113.82) ;
\draw [line width=0.75]    (346.44,112.19) -- (334.49,75.52) ;
\draw [line width=0.75]    (334.49,75.52) -- (347.06,38.84) ;
\draw [line width=0.75]    (347.06,38.84) -- (375.56,16.7) ;
\draw [line width=0.75]    (375.56,16.7) -- (346.44,112.19) ;
\draw [line width=0.75]    (463.52,113.18) -- (476.09,76.51) ;
\draw [line width=0.75]    (476.09,76.51) -- (462.87,39.83) ;
\draw [line width=0.75]    (462.87,39.83) -- (432.88,17.7) ;
\draw [line width=0.75]    (432.88,17.7) -- (463.52,113.18) ;
\draw [line width=0.75]    (278.77,111.79) -- (290.34,75.11) ;
\draw [line width=0.75]    (290.04,75.11) -- (277.88,38.44) ;
\draw [line width=0.75]    (277.88,38.44) -- (250.29,16.3) ;
\draw [line width=0.75]    (250.29,16.3) -- (278.47,111.79) ;
\draw [line width=0.75]    (173.33,113.11) -- (162.79,76.1) ;
\draw [line width=0.75]    (162.79,76.1) -- (173.95,39.76) ;
\draw [line width=0.75]    (173.95,39.76) -- (203.13,17.62) ;
\draw [line width=0.75]  [dash pattern={on 4.5pt off 4.5pt}]  (203.13,17.62) -- (173.33,113.11) ;
\draw [line width=0.75]    (173.33,113.11) -- (202.54,136.24) ;
\draw [line width=0.75]    (202.54,135.91) -- (242.29,135.91) ;
\draw [line width=0.75]    (242.29,135.91) -- (270.76,113.11) ;
\draw [line width=0.75]    (270.76,113.44) -- (241.1,17.95) ;
\draw [line width=0.75]    (241.1,17.95) -- (203.13,17.62) ;
\draw [line width=0.75]    (423.68,18.69) -- (454.47,114.18) ;
\draw [line width=0.75]    (454.47,114.18) -- (424.3,137.3) ;
\draw [line width=0.75]    (424.3,136.97) -- (383.22,136.97) ;
\draw [line width=0.75]    (383.22,136.97) -- (353.8,114.18) ;
\draw [line width=0.75]    (353.8,114.51) -- (384.45,19.02) ;
\draw [line width=0.75]    (384.45,19.02) -- (423.68,18.69) ;
\draw [line width=0.75]    (336.32,32.35) .. controls (335.22,33.14) and (345.4,43.27) .. (352.48,49.46) ;
\draw [shift={(353.91,50.69)}, rotate = 220.32] [color={rgb, 255:red, 0; green, 0; blue, 0 }  ][line width=0.75]    (10.93,-3.29) .. controls (6.95,-1.4) and (3.31,-0.3) .. (0,0) .. controls (3.31,0.3) and (6.95,1.4) .. (10.93,3.29)   ;
\draw [line width=0.75]    (473.89,33.41) .. controls (476.19,37.42) and (468.99,47.09) .. (461.21,53.55) ;
\draw [shift={(459.72,54.75)}, rotate = 322.39] [color={rgb, 255:red, 0; green, 0; blue, 0 }  ][line width=0.75]    (10.93,-3.29) .. controls (6.95,-1.4) and (3.31,-0.3) .. (0,0) .. controls (3.31,0.3) and (6.95,1.4) .. (10.93,3.29)   ;
\draw [line width=0.75]    (286.59,32.88) .. controls (288.89,36.89) and (281.67,46.56) .. (273.87,53.02) ;
\draw [shift={(272.38,54.21)}, rotate = 322.47] [color={rgb, 255:red, 0; green, 0; blue, 0 }  ][line width=0.75]    (10.93,-3.29) .. controls (6.95,-1.4) and (3.31,-0.3) .. (0,0) .. controls (3.31,0.3) and (6.95,1.4) .. (10.93,3.29)   ;

\draw (424.79,5.57) node [anchor=north west][inner sep=0.75pt]  [font=\scriptsize] [align=left] {$\displaystyle i'$};
\draw (242.67,5.15) node [anchor=north west][inner sep=0.75pt]  [font=\scriptsize] [align=left] {$\displaystyle i'$};
\draw (88.06,3.03) node [anchor=north west][inner sep=0.75pt]  [font=\scriptsize] [align=left] {$\displaystyle i'$};
\draw (456.15,117.69) node [anchor=north west][inner sep=0.75pt]  [font=\scriptsize] [align=left] {$\displaystyle j'$};
\draw (272.46,116.63) node [anchor=north west][inner sep=0.75pt]  [font=\scriptsize] [align=left] {$\displaystyle j'$};
\draw (118.16,117.34) node [anchor=north west][inner sep=0.75pt]  [font=\scriptsize] [align=left] {$\displaystyle j'$};
\draw (343.09,116.08) node [anchor=north west][inner sep=0.75pt]  [font=\scriptsize] [align=left] {$\displaystyle i$};
\draw (6.9,116.09) node [anchor=north west][inner sep=0.75pt]  [font=\scriptsize] [align=left] {$\displaystyle i$};
\draw (162.28,116.08) node [anchor=north west][inner sep=0.75pt]  [font=\scriptsize] [align=left] {$\displaystyle i$};
\draw (377.94,4.93) node [anchor=north west][inner sep=0.75pt]  [font=\scriptsize] [align=left] {$\displaystyle j$};
\draw (194.15,4.5) node [anchor=north west][inner sep=0.75pt]  [font=\scriptsize] [align=left] {$\displaystyle j$};
\draw (39.54,2.81) node [anchor=north west][inner sep=0.75pt]  [font=\scriptsize] [align=left] {$\displaystyle j$};
\draw (32.65,87.31) node [anchor=north west][inner sep=0.75pt]  [font=\scriptsize] [align=left] {$\displaystyle X\left( \sigma ^{a+b+c-2}\right)$};
\draw (378.41,138.69) node [anchor=north west][inner sep=0.75pt]  [font=\scriptsize] [align=left] {$\displaystyle 1$};
\draw (197.71,138.26) node [anchor=north west][inner sep=0.75pt]  [font=\scriptsize] [align=left] {$\displaystyle 1$};
\draw (42.73,138.71) node [anchor=north west][inner sep=0.75pt]  [font=\scriptsize] [align=left] {$\displaystyle 1$};
\draw (416.16,138.26) node [anchor=north west][inner sep=0.75pt]  [font=\scriptsize] [align=left] {$\displaystyle k+1$};
\draw (232.46,138.26) node [anchor=north west][inner sep=0.75pt]  [font=\scriptsize] [align=left] {$\displaystyle k+1$};
\draw (78.26,139.13) node [anchor=north west][inner sep=0.75pt]  [font=\scriptsize] [align=left] {$\displaystyle k+1$};
\draw (322.04,14.55) node [anchor=north west][inner sep=0.75pt]  [font=\scriptsize] [align=left] {$\displaystyle X\left( \sigma ^{a}\right)$};
\draw (390.06,89.56) node [anchor=north west][inner sep=0.75pt]  [font=\scriptsize] [align=left] {$\displaystyle X( \sigma ^{b}$)};
\draw (456.78,14.98) node [anchor=north west][inner sep=0.75pt]  [font=\scriptsize] [align=left] {$\displaystyle X\left( \sigma ^{c}\right)$};
\draw (271.94,14.98) node [anchor=north west][inner sep=0.75pt]  [font=\scriptsize] [align=left] {$\displaystyle X\left( \sigma ^{c}\right)$};
\draw (193.64,87.74) node [anchor=north west][inner sep=0.75pt]  [font=\scriptsize] [align=left] {$\displaystyle X\left( \sigma ^{a+b-1}\right)$};

\end{tikzpicture}
            \vspace{-.5cm}
            \caption{Type B: Initial cut at $\Delta_{i'j'}$ followed by $\Delta_{ij}$.}
            \vspace{1cm}
            \label{fig: Type_B_Delta_ij2}
        \end{subfigure}
        \caption{All possible variations of Type A and B cuts.}
        \label{Cuts}        
    \end{figure}

\begin{lemma}
\label{lem: def T}
For $A,A'\in\Pi_{2,c+1}^{\circ}$ we define the map $T_\lambda$ as 
        $$
            T_{\lambda}:A\rightarrow A',\ a_m\rightarrow a_m\lambda^{m-j}
        $$
Then $T_\lambda$ preserves   $\Pi_{2,n}^{\circ,1}$ and defines $\C^*$ actions on       $\Pi_{2,n}^{\circ}$ and $\Pi_{2,n}^{\circ,1}$.
\end{lemma}

\begin{proof}
\end{proof}

\begin{theorem}
\label{thm: type B cuts}
For Type B cuts we have a commutative diagram  
        $$     
\begin{tikzcd}
	{X(\sigma^a)\times X(\sigma^b)\times X(\sigma^c)} && {X(\sigma^a)\times X(\sigma^{b+c-1})} && {X(\sigma^{a+b+c-2})} \\
	\\
	{X(\sigma^a)\times X(\sigma^b)\times X(\sigma^c)} &&&& {X(\sigma^{a+b-1})\times X(\sigma^c)}
	\arrow["{\Id\times \Id\times T_{\Delta_{ij}}}"', from=1-1, to=3-1]
	\arrow["{\Id\times\Phi_{i'j'}}", from=1-1, to=1-3]
	\arrow["{\Phi_{ij}}", from=1-3, to=1-5]
	\arrow["{\Phi_{ij}\times \Id}", from=3-1, to=3-5]
	\arrow["{\Phi_{i'j'}}"', from=3-5, to=1-5]
\end{tikzcd}
        $$
        Here $T_{\Delta_{ij}}$ is defined as in Lemma \ref{lem: def T} with $\lambda=\Delta_{ij}$.
        Informally, we can say that the gluing $P$ from smaller polygons is associative only up to the additional transformation $T_{\Delta_{ij}}$.
\end{theorem}

\begin{proof}
    Let $V\in\Pi_{2,k+1}^{\circ,1}$ by Theorem \ref{thm: iso}.

    For Type B cuts, choose some $i,j,i',j'$ such that $1\le i<j\le i'<j'\le k+1$. Similar to Theorem \ref{thm: type A cut}, we describe the inverse maps then produce the desired map. Here $a=j-i,\,b=k-j'+i'-j+i+2,\,c=j'-i'$. Define the matrix $V\in\Pi_{2,k+1}^{\circ,1}$ associated to $X(\sigma^k)$ as 
    $$
    V=\begin{pmatrix}
        v_1 &\dots &v_i &\dots &v_j &\dots &v_{i'} &\dots &v_{j'} &\dots &v_{k+1}
    \end{pmatrix}$$
    Similar to Theorem \ref{thm: type A cut} we will be dealing with minors in several different matrices and will include the matrices in the notations.

    (i) We first consider the case where we cut along $\Delta_{ij}(V)$ then $\Delta_{i'j'}(V)$, see Figure \ref{fig: Type B Delta ij}, given by the map 
    $$
    X(\sigma^{a+b+c-2})\xrightarrow{\Phi_{ij}^{-1}} X(\sigma^{a})\times X(\sigma^{b+c-1})\xrightarrow{\Id\times\Phi_{i'j'}^{-1}}X(\sigma^a)\times X(\sigma^b)\times X(\sigma^c)
    $$
    Performing the initial cut $\Delta_{ij}(V)$, given by $\Phi_{ij}^{-1}:X(\sigma^{a+b+c-2})\rightarrow X(\sigma^a)\times X(\sigma^{b+c-1})$, decomposes $V$ into the two matrices 
    $$
    V_1=\begin{pmatrix}
        v_i &\dots &v_j
    \end{pmatrix}\in\Mat(2,a+1)
    $$
    $$
    V_2=\begin{pmatrix}
    v_1 &\dots &v_i &v_j &\dots &v_{i'} &\dots &v_{j'} &\dots &v_{k+1}
    \end{pmatrix}\in\Mat(2,b+c)
    $$
    By the same argument as in Theorem \ref{thm: one cut}, $V_1\in\Pi_{2,a+1}^{\circ,1}\simeq X(\sigma^a)$ whereas $V_2\in\Pi_{2,b+c}^{\circ}$ and requires rescaling by $\Delta_{ij}(V)$ for the vectors $v_m$ for $m\ge j$, resulting in the matrix 
    $$V_3=\begin{pmatrix}
        v_1 &\dots &v_i &v'_j &\dots &v'_{i'} &\dots &v'_{j'} &\dots &v'_{k+1}
    \end{pmatrix}
    $$
    where 
    \begin{equation}
    \label{eq: v' larger than j type B}
        v'_m=v_m\Delta_{ij}(V)^{(-1)^{m-j+1}}
    \end{equation}
    Note that $\Delta_{i'j'}(V_3)$ experiences a rescaling by factor of $\Delta_{ij}(V)$, given that $v'_{i'}=V_{i'}\Delta_{ij}^{(-1)^{i'-j+1}}$ and $v'_{j'}=v_{j'}\Delta_{ij}^{(-1)^{j'-j+1}}$, the rescaled determinant is given by 
    \begin{align}
    \label{eq: rescale delta Type B}
        \Delta_{i'j'}(V_3)&=\Delta_{i'j'}(V)\Delta_{ij}(V)^{(-1)^{i'-j+1}}\Delta_{ij}(V)^{(-1)^{j'-j+1}}\notag\\
        &=\Delta_{i'j'}(V)\Delta_{ij}(v)^{(-1)^{i'-j+1}+(-1)^{j'-j+1}}
    \end{align}
    
    This completes the construction of the map $\Phi^{-1}_{ij}:X(\sigma_{a+b+c-2})\rightarrow X(\sigma^a)\times X(\sigma^{b+c-1})$. Applying the second cut $\Delta_{i'j'}(V)$ to the matrix $V_3$ produces the two matrices
    $$
    V_4=\begin{pmatrix}
        v'_{i'} &\dots &v'_{j'}
    \end{pmatrix}\in\Mat(2,c+1)
    $$
    $$
    V_5=\begin{pmatrix}
        v_1 &\dots &v_i &v'_j &\dots &v'_{i'} &v'_{j'} &\dots &v'_{k+1}
    \end{pmatrix}\in\Mat(2,b+1)
    $$
    Here, $V_4\in\Pi_{2,c+1}^{\circ,1}\simeq X(\sigma^c)$. Since $V_5\in\Pi_{2,b+1}^{\circ}$ we applying a rescaling of the vectors $v'_m$ for $m\ge j'$ into the matrix 
    $$
    V_6=\begin{pmatrix}
        v_1 &\dots &v_i &v'_j &\dots &v'_{i'} &v''_{j'} &\dots &v''_{k+1}
    \end{pmatrix}$$
    where 
    \begin{equation}
        \label{eq:v'' m larger than j' Type B}
        v''_m=v'_m\Delta_{i'j'}(V_3)^{(-1)^{m-j'+1}}
    \end{equation}
    Using \eqref{eq: rescale delta Type B} we find that 
    \begin{align*}
        \Delta_{i'j'}(V_3)^{(-1)^{m-j'+1}}&= (\Delta_{i'j'}(V)\Delta_{ij}(V)^{(-1)^{i'-j+1}+(-1)^{j'-j+1}})^{(-1)^{m-j'+1}}\\
        &=\Delta_{i'j'}(V)^{(-1)^{m-j'+1}}\Delta_{ij}(V)^{(-1)^{i'-j+1}(-1)^{m-j'+1}+(-1)^{j'-j+1}(-1)^{m-j'+1}}
    \end{align*}
    and $(-1)^{i'-j+1}(-1)^{m-j'+1}+(-1)^{j'-j+1}(-1)^{m-j'+1}=(-1)^{m-j'+i'-j}+(-1)^{m-j}$. Therefore         
    \begin{align}
    \label{eq: v'' m larger j' Type B}
            v''_m&=v_m\Delta_{ij}(V)^{(-1)^{m-j+1}}\Delta_{i'j'}(V)^{(-1)^{m-j'+1}}\Delta_{ij}(V)^{(-1)^{m-j'+i'+j}}\Delta_{ij}^{(-1)^{m-j}}\notag\\
            &=v_m\Delta_{i'j'}(V)^{(-1)^{m-j'+1}}\Delta_{ij}(V)^{(-1)^{m-j'+i'-j}}
    \end{align}

    Now, $V_6\in\Pi_{2,b+1}^{\circ,1}\simeq X(\sigma^b)$. This concludes the construction of the inverse map, now we proceed to the construction of the desired map 
    $$
    X(\sigma^a)\times X(\sigma^b)\times X(\sigma^c)\xrightarrow{\Id\times\Phi_{i'j'}} X(\sigma^a)\times X(\sigma^{b+c-1})\xrightarrow{\Phi_{ij}}X(\sigma^{a+b+c-2})$$

    Given $V_1\in\Pi_{2,a+1}^{\circ,1},\, V_6\in\Pi_{2,b+1}^{\circ,1},\, V_4\in\Pi_{2,c+1}^{\circ,1}$ we reconstruct the matrix $V$. First, we determine that $\Delta_{i'j'}(V)=\Delta_{1,c+1}(V_4)\ne0$.The matrix $V_5$ is found by multiplication of $\Delta_{i'j'}(V)^{\pm1}$ to the vectors $v_l$ for $l\ge i+i'-j+2$ in matrix $V_6$. We then reconstruct $V_3\in\Pi^{\circ,1}_{2,b+c}\simeq X(\sigma^{b+c-1})$ by inserting the matrix $V_4$ into the appropriate position in the matrix $V_5$. This completes the map $\Id\times\Phi_{i'j'}:X(\sigma^a)\times X(\sigma^b)\times X(\sigma^c)\rightarrow X(\sigma^a)\times X(\sigma^{b+c-1})$.
    Now we continue our construction of the matrix $V$ by reading off $\Delta_{ij}(V)=\Delta_{1,a+1}(V_1)$ which is nonzero by assumption. We rescale the vectors $v_l$ for $l\ge i+1$ in the matrix $V_3$ by multiplication of $\Delta_{ij}(V)^{\pm1}$ which is invertible, to obtain the matrix $V_2$. Finally, we insert the matrix $V_1$ into $V_2$ to obtain $V$. Therefore, giving us the desired map above.

    (ii) Now, we consider the case were we first cut along $\Delta_{i'j'}(V)$ followed by the cut $\Delta_{ij}(V)$ and subsequently, a rescaling of $X(\sigma^c)$ by the torus action $T_{\Delta_{ij}}$, illustrated in Figure \ref{fig: Type_B_Delta_ij2}, given by 
    $$
    X(\sigma^{a+b+c-2})\xrightarrow{\Phi_{i'j'}^{-1}}X(\sigma^{a+b-1})\times X(\sigma^c)\xrightarrow{\Phi_{ij}^{-1}\times\Id}X(\sigma^a)\times X(\sigma^b)\times X(\sigma^c)\xrightarrow{\Id\times\Id\times T_{\Delta_{ij}}} X(\sigma^a)\times X(\sigma^b)\times X(\sigma^c)
    $$
    We perform the initial cut $\Delta_{i'j'}(V)$ to $V$ resulting in the matrices 
    $$
    W_1=\begin{pmatrix}
        v_{i'} &\dots &v_{j'}
    \end{pmatrix}\in\Mat(2,c+1)
    $$
    $$
    W_2=\begin{pmatrix}
        v_1 &\dots &v_i &\dots &v_j &\dots &v_{i'} &v_{j'} &\dots &v_{k+1}
    \end{pmatrix}\in\Mat(2,a+b)
    $$
    By the same argument in Theorem \ref{thm: one cut}, $V_1\in\Pi_{2,c+1}\simeq X(\sigma^c)$, whereas the matrix $V_2\in\Pi_{2,a+b}^{\circ}$ requires as rescaling of the vectors $v_m$ for $m\ge j'$ to obtain the matrix 
    $$
    W_3=\begin{pmatrix}
        v_1 &\dots &v_i &\dots &v_j &\dots &v_{i'} &\widetilde{v}_{j'} &\dots &\widetilde{v}_{k+1}
    \end{pmatrix}
    $$
    given by 
    \begin{equation}
        \label{eq: v' m larger j' pre torus Type B}
        \widetilde{v}_m=v_m\Delta_{i'j'}^{(-1)^{m-j'+1}}
    \end{equation}
    Now, $W_3\in\Pi_{2,a+b}^{\circ,1}\simeq X(\sigma^{a+b-1})$, completing the first map.

    We now perform the second cut $\Delta_{ij}(V)=\Delta_{ij}(W_2)$ by decomposing the matrix $W_3$ into the matrices
    $$
    W_4=\begin{pmatrix}
        v_i &\dots & v_j
    \end{pmatrix}\in\Mat(2,a+1)
    $$
    $$
    W_5=\begin{pmatrix}
        v_1 &\dots &v_i &v_j &\dots &v_{i'} &\widetilde{v}_{j'} &\dots &\widetilde{v}_{k+1}
    \end{pmatrix}\in\Mat(2,b+1)$$
    Given that $W_4\in\Pi_{2,a+1}^{\circ,1}\simeq X(\sigma^a)$ and $W_5\in\Pi_{2,b+1}^\circ$, the matrix vectors $v_m$ for $m\ge j$ in $W_5$ are rescaled into the matrix
    $$
    W_6=\begin{pmatrix}
        v_1 &\dots &v_i &v'_j &\dots &v'_{i'} &\widetilde{v}'_{j'} &\dots &\widetilde{v}'_{k+1}
    \end{pmatrix}$$
    where for $j\le m\le i'$
    \begin{equation}
    \label{eq: v'' btwn j and j' Type B}
        v'_m=v_m\Delta_{ij}^{(-1)^{m-j+1}}
    \end{equation}
    and for $m\ge j'$
    \begin{align}
        \label{v'' m larger j' Type B}
        \widetilde{v}'_m&=\widetilde{v}_m\Delta_{ij}(V)^{(-1)^{m-j'+i'-j}}\notag\\
        &=v_m\Delta_{i'j'}(V)^{(-1)^{m-j'+1}}\Delta_{ij}(V)^{(-1)^{m-j'+i'-j}}
    \end{align}
    Now, $W_6\in\Pi_{2,b+1}^{\circ,1}\simeq X(\sigma^b)$.  Note that for the vectors $v'_m$ for $j\le m\le i'$ \eqref{eq: v' larger than j type B} agrees with \eqref{eq: v'' btwn j and j' Type B}, for $j'\le m$ \eqref{eq: v'' m larger j' Type B} agrees with \eqref{v'' m larger j' Type B}. However, the vectors $v_m$ found in $W_1$ for $i'<m<j'$ do not agree with \eqref{eq: v' larger than j type B} and differ by a factor of $\Delta_{ij}(V)^{(-1)^{m-j+1}}$. Since $\Delta_{ij}\neq0$, we can then apply a torus action to the matrix $W_1\in\Pi_{2,c+1}^{\circ,1}\simeq X(\sigma^c)$ using Lemma \ref{lem: rescale}. Let $W_1,W_7\in\Pi_{2,c+1}^{\circ,1}$, define the torus action by the map 
    \begin{align*}
        T_{\Delta_{ij}}^{-1}:W_1&\longrightarrow W_7\\
        v_m&\longmapsto v_m\Delta_{ij}^{(-1)^{m-j-1}}
    \end{align*}
    Thus concluding the construction of the inverse maps.

    Now, we construct the suitable map to establish associativity up to an additional transformation $T_{\Delta_{ij}}$, given by 
    $$
    X(\sigma^a)\times X(\sigma^b)\times X(\sigma^c)\xrightarrow{\Id\times\Id\times T_{\Delta_{ij}}} X(\sigma^a)\times X(\sigma^b)\times X(\sigma^c)\xrightarrow{\Phi_{ij}\times\Id} X(\sigma^{a+b-1})\times X(\sigma^c)\xrightarrow{\Phi_{i'j'}}X(\sigma^{a+b+c-2})
    $$
    We reconstruct the matrix $V$ using $W_4\in\Pi_{2,a+1}^{\circ,1},\, W_6\in\Pi_{2,b+1}^{\circ,1},\,W_7\in\Pi_{2,c+1}^{\circ,1}$. First, we read off $\Delta_{ij}(V)=\Delta_{1,a+1}(W_4)\neq 0$ by assumption. We apply the toric action $T_{\Delta_{ij}}(W_7)=W_1\in\Pi_{2,c+1}^{\circ,1}$. Now, we rescale the matrix $W_6$ by multiplication of $\Delta_{ij}^{\pm1}$ to the vectors $v_l$ for $l\ge j$, producing the matrix $W_5$. We then reinsert the matrix $W_4$ into $W_5$ at the appropriate location, arriving at the matrix $W_3\in\Pi_{2,a+b}^{\circ,1}\simeq X(\sigma^{a+b-1})$. We then read of $\Delta_{i'j'}(V)=\Delta_{1,c+1}(W_1)\ne0$ and multiply $W_3$ by a factor of $\Delta_{i'j'}(V)^{\pm1}$ for vectors $v_l$ for $l\ge j'$ to produce $W_2$. We then reinsert the matrix $W_1$ into $W_2$ arriving at the desired matrix $V$. Thereby completing the construction of the desired map. 
\end{proof}

\subsection{Cuts, forms and cohomology}

Now we can study the effect of the cuts on the forms $\alpha$ and $\omega$. More precisely, we use the map
$\Phi_{ij}:X(\sigma^{j-i})\times X(\sigma^{k-j+i+1})\longrightarrow X(\sigma^{k})$
to compute the pullbacks $\Phi_{ij}^*\alpha$ and $\Phi_{ij}^{*}\omega$. The forms $\alpha$ and $\omega$ are equivalent under cluster mutation by \cite{LS}; hence, we choose an arbitrary cluster chart, see Figure \ref{fig: 2-form cuts}, and determine the how the forms interact with cuts. 

We will denote the forms from $X(\sigma^{j-i})$ by $\alpha_1$ and $\omega_1$, and the forms from $X(\sigma^{k-j+i+1})$ by $\alpha_2$ and $\omega_2$. As an abuse of notation we use the labeling from the larger positroid $\Pi_{2,a+b-1}^{o,1}$ identified with $X(\sigma^{a+b-1})$. Technically, under the isomorphism $\Delta_{1,j-i+1}=\Phi_{ij}^*(\Delta_{ij})$, therefore, $\alpha_1=\Phi_{ij}^*(d\log\Delta_{ij})$, similarly, $\alpha_2=\Phi_{ij}^*(d\log\Delta_{ij}^{(-1)^{k-j+1}}\Delta_{1,k+1})$ with similar considerations made to $\omega_1$ and $\omega_2$.

\begin{figure}        \include{Figures/2_form_cut}
        \caption{Triangulation of $(k+1)$-gon corresponding to the braid variety $X(\sigma^k)$ with its associated quiver. A cut $\Delta_{ij}$ is depicted between vertices $i$ and $j$. The cluster variables from the particular triangulation are the written in black and the rescaling factor of the cluster variables from the cut $\Delta_{ij}$ are written in red.}
        \label{fig: 2-form cuts}
\end{figure}

\begin{lemma}
\label{lem: alpha form}
We have 
$$
\Phi_{ij}^*\alpha=\alpha_2+(-1)^{k-j} \alpha_1.
$$
\end{lemma}

\begin{proof}
Recall that $\alpha=d\log(\Delta_{1,k+1})$. By \cite{LS} let $\alpha_1=d\log(\Delta_{ij})$ be the 1-form associated to $X(\sigma^{j-i})$ and $\alpha_2=d\log(\Delta_{ij}^{(-1)^{k-j+1}}w)=d\log(\Delta_{ij}^{(-1)^{k-j+1}}\Delta_{1,k+1})$ be the 1-form associated to $X(\sigma^{k-j+i+1})$. Given these conditions we find that 
\begin{align*}
    \alpha_2+(-1)^{k-j}\alpha_1&=d\log(\Delta_{ij}^{(-1)^{k-j+1}}\Delta_{1,k+1})+(-1)^{k-j}d\log(\Delta_{ij})\\
    &=d\log(\Delta_{ij}^{(-1)^{k-j+1}})+d\log(\Delta_{1,k+1})+(-1)^{k-j}d\log(\Delta_{ij})\\
    &=(-1)^{k-j+1}d\log(\Delta_{ij})+d\log(\Delta_{1,k+1})+(-1)^{k-j}d\log(\Delta_{ij})\\
    &=d\log(\Delta_{1,k+1})=\alpha
\end{align*}
\end{proof}

\begin{lemma}
\label{lem: omega eqn}
We have 
$$
\Phi_{ij}^*\omega=\omega_1+\omega_2+(-1)^{k-j}\alpha_1\wedge \alpha_2.
$$
\end{lemma}

\begin{proof}
    Consider the quiver associated to the triangulation of $X(\sigma^k)$ in Figure \ref{fig: 2-form cuts} prior to the rescaling given by the cut $\Delta_{ij}$, by \eqref{eq: def Omega}, the two-form $\omega$ is described as
    \begin{align*}
    \omega&= d\log\Delta_{1,k+1}\wedge      d\log\Delta_{1,k}+d\log\Delta_{1,k}\wedge d\log\Delta_{1,k-1}\\
            &\quad +\dots+ d\log\Delta_{1,j+1}\wedge d\log\Delta_{1,j}+d\log\Delta_{1,j}\wedge d\log\Delta_{1,i}\\
            &\quad+d\log \Delta_{1,i}\wedge d\log \Delta_{1,i-1}+d\log \Delta_{1,i-1}\wedge d\log \Delta_{1,i-2}\\
            &\quad+\dots+d\log \Delta_{14}\wedge d\log \Delta_{13}+d\log\Delta_{1,i}\wedge d\log\Delta_{ij}\\
            &\quad +d\log\Delta_{ij}\wedge d\log\Delta_{1,j}+d\log\Delta_{i,j-1}\wedge d\log\Delta_{i,j-2}\\
            &\quad+d\log\Delta_{i,j-2}\wedge d\log\Delta_{i,j-3}+\dots+d\log\Delta_{i,i+3}\wedge d\log\Delta_{i,i+2}
    \end{align*}    
    Let $\alpha_1,\alpha_2$ be the $1$-form and $\omega_1,\,\omega_2$ be the $2$-form associated to $X(\sigma^{j-i})$ and $X(\sigma^{k-j+i+1})$, respectively. By Figure \ref{fig: 2-form cuts}, we define the forms associated to $X(\sigma^{j-i})$ and $X(\sigma^{k-j+i+1})$ directly from quivers as follows:
    \begin{align}
        \alpha_1&=d\log \Delta_{ij}\label{eq: alpha 1 form}\\
        \alpha_2&=d\log( \Delta_{1,k+1}\Delta_{ij}^{(-1)^{k-j+2}})= d\log\Delta_{1,k+1}+(-1)^{k-j+2} d\log\Delta_{ij}\label{eq: alpha 2 form}\\[5pt]
        \omega_1&=d\log\Delta_{i,j-1}\wedge d\log\Delta_{i,j-2}+d\log\Delta_{i,j-2}\wedge d\log\Delta_{i,j-3}\notag\\
        &\quad+\dots+d\log\Delta_{i,i+3}\wedge d\log\Delta_{i,i+2}\notag
    \end{align}  
    While $\alpha_1,\,\alpha_2,\omega_1$ can be easily read from the cluster chart seen in Figure \ref{fig: 2-form cuts}, the $2$-form $\omega_2$ requires a bit more finesse. We notice that there is a triangle formed between the vertices $1,i,j$, to simplify the computation of $\omega_2$, which agrees with \eqref{eq: def Omega}, we decompose the form into parts and call them pre-triangle $\omega_{2,pre}$ for vertices between $1$ and $i$, triangle $\omega_{2,tri}$ for the special vertices $1,i,j$ and post-triangle $\omega_{2,post}$ for vertices between $j$ and $k+1$. By Theorem \ref{thm: one cut} in the rescaled braid variety $X(\sigma^{k-j+i+1})$ the Plu\"cker coordinate $\Delta'_{ij}=\Delta_{ij}\Delta_{ij}^{-1}=1$ resulting in $d\log\Delta'_{ij}=d\log 1=0$, whereas $\Delta_{ij}$ shall remain the nonzero polynomial $w$ describing $X(\sigma^{j-i})$.  Using this decomposition, $\omega_2=\omega_{2,pre}+\omega_{2,tri}+\omega_{2,post}$ is defined by
    \begin{align*}
        \omega_{2,pre}&=d\log \Delta_{1i}\wedge d\log \Delta_{1,i-1}+d\log \Delta_{1,i-1}\wedge d\log \Delta_{1,i-2}+\dots+d\log \Delta_{14}\wedge d\log \Delta_{13}
    \end{align*}
    \begin{align*}
        \omega_{2,tri}&=d\log \left(\Delta_{1j}\Delta_{ij}^{-1}\right)\wedge d\log \Delta_{1i}+d\log \Delta_{1i}\wedge d\log \Delta'_{ij}+d\log \Delta'_{ij}\wedge d\log \left(\Delta_{1j}\Delta_{ij}^{-1}\right)\\
        &=(d\log \Delta_{1j}-d\log \Delta_{ij})\wedge d\log \Delta_{1i}\\
        &=d\log\Delta_{1j}\wedge d\log\Delta_{1i}-d\log\Delta_{ij}\wedge d\log\Delta_{1i}
    \end{align*}
    \begin{align*}
        \omega_{2,post}&=d\log \Delta_{1,k+1}\Delta_{ij}^{(-1)^{k-j+2}}\wedge d\log \Delta_{1,k}\Delta_{ij}^{(-1)^{k-j+1}}\\
         &\quad+d\log \Delta_{1,k}\Delta_{ij}^{(-1)^{k-j+1}}\wedge d\log \Delta_{1,k-1}\Delta_{ij}^{(-1)^{k-j}}\\
         &\quad +\dots+d\log \Delta_{1,j+2}\Delta_{ij}^{-1}\wedge d\log \Delta_{1,j+1}\Delta_{ij} +d\log \Delta_{1,j+1}\Delta_{ij}\wedge d\log \Delta_{1,j}\Delta_{ij}^{-1}   \end{align*}
         \begin{align*}
         &=(d\log\Delta_{1,k+1}+(-1)^{k-j+2} d\log\Delta_{ij})\wedge(d\log\Delta_{1,k}+(-1)^{k-j+1}d\log\Delta_{ij})\\
         &\quad +(d\log\Delta_{1,k}+(-1)^{k-j+1} d\log\Delta_{ij})\wedge(d\log\Delta_{1,k-1}+(-1)^{k-j}d\log\Delta_{ij})\\
         &\quad +\dots+(d\log\Delta_{1,j+2}- d\log\Delta_{ij})\wedge(d\log\Delta_{1,j+1}+d\log\Delta_{ij})\\
         &\quad +(d\log\Delta_{1,j+1}+ d\log\Delta_{ij})\wedge(d\log\Delta_{1,j}-d\log\Delta_{ij})
         \end{align*}
         \begin{align*}
         &=d\log\Delta_{1,k+1}\wedge d\log\Delta_{1,k}+(-1)^{k-j+1} d\log\Delta_{1,k+1}\wedge d\log\Delta_{ij}\\
         &\quad +(-1)^{k-j+2} d\log\Delta_{ij}\wedge d\log\Delta_{1,k}+ d\log\Delta_{1,k}\wedge d\log\Delta_{1,k-1}\\
         &\quad+(-1)^{k-j} d\log\Delta_{1,k}\wedge d\log\Delta_{ij}+(-1)^{k-j+1}d\log\Delta_{ij}\wedge d\log\Delta_{1,k-1}\\
          &\quad+\dots+ d\log\Delta_{1,j+2}\wedge d\log\Delta_{1,j+1}+ d\log\Delta_{1,j+2}\wedge  d\log\Delta_{ij}\\
          &\quad - d\log\Delta_{i,j}\wedge d\log \Delta_{1,j+1}+ d\log\Delta_{1,j+1}\wedge d\log\Delta_{1,j}\\
          &\quad - d\log\Delta_{1,j+1}\wedge d\log\Delta_{ij}+ d\log\Delta_{ij}\wedge d\log\Delta_{1,j}\\[5pt]
          &= d\log\Delta_{1,k+1}\wedge d\log\Delta_{1,k}+ d\log\Delta_{1,k}\wedge d\log\Delta_{1,k-1}\\
          &\quad +\dots+ d\log\Delta_{1,j+2}\wedge d\log\Delta_{1,j+1}+ d\log\Delta_{1,j+1}\wedge d\log\Delta_{1,j}\\
          &\quad +(-1)^{k-j+1} d\log\Delta_{1,k+1}\wedge d\log\Delta_{ij}
    \end{align*}
    Note that from \eqref{eq: alpha 1 form} and \eqref{eq: alpha 2 form}, $\alpha_1\wedge\alpha_2=d\log\Delta_{ij}\wedge d\log\Delta_{1,k+1}$. Therefore, the additional term $(-1)^{k-j+1} d\log\Delta_{1,k+1}\wedge d\log\Delta_{ij}$ from $\omega_{2,post}$ may be negated by $(-1)^{k-j}\alpha_1\wedge\alpha_2$, providing the necessary adjustment to acquire $\Phi^*_{ij}\omega$ as stated.           
\end{proof}

\begin{theorem}
The pullback map
$$
\Phi_{ij}^*:H^*(X(\sigma^{k}))\to H^*(X(\sigma^{j-i}))\otimes H^*(X(\sigma^{k-j+i+1}))
$$
is injective and can be described by Lemmas \ref{lem: alpha form} and \ref{lem: omega eqn}
\end{theorem}
\begin{proof}
    Similar to Theorem \ref{thm: alg forms}, we want to prove that the restrictions of all forms in \eqref{eq: basis even} and \eqref{eq: basis odd} do not vanish in $H^*(X(\sigma^{j-i}))\otimes H^*(X(\sigma^{k-j+i+1}))$, here we use the formulas from Lemmas \ref{lem: alpha form} and \ref{lem: omega eqn}. 

    Suppose $k$ is odd, then we want to show that $\Phi_{ij}^*\left[\alpha\omega^{\frac{k-3}{2}}\right]$ and $\Phi_{ij}^*\left[\omega^{\frac{k-1}{2}}\right]$ are both nonzero. Since $k=a+b-1$ is odd, then either $a,b$ are both even or both odd.

    (i) Suppose $a$ and $b$ are both even. Given that $\omega_1^{\frac{a}{2}-1},\,\alpha_1\omega_1^{\frac{a}{2}-1},\,\omega_2^{\frac{b}{2}-1},\,\alpha_2\omega_2^{\frac{b}{2}-1}$ are nonzero by definition, then 
    \begin{align*}        \Phi_{ij}^*\left[\alpha\omega^{\frac{k-3}{2}}\right]&=(\alpha_2+(-1)^{k-j}\alpha_1)(\omega_1+\omega_2+(-1)^{k-j}\alpha_1\wedge\alpha_2)^{\frac{k-3}{2}}\\
        &=(\alpha_2+(-1)^{k-j}\alpha_1)(\omega_1+\omega_2+(-1)^{k-j}\alpha_1\wedge\alpha_2)^{\frac{a+b-4}{2}}\\
        &=(\alpha_2+(-1)^{k-j}\alpha_1)\displaystyle\sum_{l_1+l_2+l_3=\frac{a+b-4}{2}}\binom{\frac{a+b-4}{2}}{l_1,l_2,l_3}\omega_1^{l_1}\omega_2^{l_2}\left((-1)^{k-j}\alpha_1\wedge\alpha_2\right)^{l_3}\\
        &=\binom{\frac{a+b-4}{2}}{\frac{a}{2}-1,\frac{b}{2}-1,0}\alpha_2\omega_1^{\frac{a}{2}-1}\omega_2^{\frac{b}{2}-1}+\dots
    \end{align*}
    with $\alpha_2\omega_2^{\frac{b}{2}-1},\,\omega_1^{\frac{a}{2}-1}\ne0$, then $\Phi_{ij}^*\left[\alpha\omega^{\frac{k-3}{2}}\right]$ is nonvanishing. Furthermore,
    \begin{align*}        \Phi_{ij}^*\left[\omega^{\frac{k-1}{2}}\right]&=(\omega_1+\omega_2+(-1)^{k-j}\alpha_1\wedge\alpha_2)^{\frac{k-1}{2}}\\
        &=(\omega_1+\omega_2+(-1)^{k-j}\alpha_1\wedge\alpha_2)^{\frac{a+b-2}{2}}\\        &=\displaystyle\sum_{l_1+l_2+l_3=\frac{a+b-2}{2}}\binom{\frac{a+b-2}{2}}{l_1,l_2,l_3}\omega_1^{l_1}\omega_2^{l_2}\left((-1)^{k-j}\alpha_1\wedge\alpha_2\right)^{l_3}\\
        &=\binom{\frac{a+b-2}{2}}{\frac{a}{2}-1,\frac{b}{2}-1,1}\left((-1)^{k-j}\alpha_1\wedge\alpha_2\right)\omega_1^{\frac{a}{2}-1}\omega_2^{\frac{b}{2}-1}+\dots
    \end{align*}
    where $\alpha_1\omega_1^{\frac{a}{2}-1},\,\alpha_2\omega_2^{\frac{b}{2}-1}\neq0$. Then $\Phi_{ij}^*\left[\omega^{\frac{k-1}{2}}\right]$ is nonvanishing.\\
    (ii) Suppose $a$ and $b$ are both odd. Given that $\alpha_1\omega_1^{\frac{a-3}{2}},\,\omega_1^{\frac{a-1}{2}},\,\alpha_2\omega_2^{\frac{b-3}{2}},\,\omega_2^{\frac{b-1}{2}}$ are nonzero, then
    \begin{align*}
        \Phi_{ij}^*\left[\alpha\omega^{\frac{k-3}{2}}\right]
        &=(\alpha_2+(-1)^{k-j}\alpha_1)\displaystyle\sum_{l_1+l_2+l_3=\frac{a+b-4}{2}}\binom{\frac{a+b-4}{2}}{l_1,l_2,l_3}\omega_1^{l_1}\omega_2^{l_2}\left((-1)^{k-j}\alpha_1\wedge\alpha_2\right)^{l_3}\\
        &=(\alpha_2+(-1)^{k-j}\alpha_1)\binom{\frac{a+b-4}{2}}{\frac{a-3}{2},\frac{b-}{2},0}\omega_1^{\frac{a-3}{2}}\omega_2^{\frac{b-1}{2}}+\dots\\
        &=(-1)^{k-j}\binom{\frac{a+b-4}{2}}{\frac{a-3}{2},\frac{b-}{2},0}\alpha_1\omega_1^{\frac{a-3}{2}}\omega_2^{\frac{b-1}{2}}+\dots
    \end{align*}
    Given $\alpha_1\omega_1^{\frac{a-3}{2}},\,\omega_2^{\frac{b-1}{2}}\neq0$, then $\Phi_{ij}^*\left[\alpha\omega^{\frac{k-3}{2}}\right]$ is nonvanishing. Furthermore,
    \begin{align*}
        \Phi_{ij}^*\left[\omega^{\frac{k-1}{2}}\right]        &=\displaystyle\sum_{l_1+l_2+l_3=\frac{a+b-2}{2}}\binom{\frac{a+b-2}{2}}{l_1,l_2,l_3}\omega_1^{l_1}\omega_2^{l_2}\left((-1)^{k-j}\alpha_1\wedge\alpha_2\right)^{l_3}\\
        &=\binom{\frac{a+b-2}{2}}{\frac{a-1}{2},\frac{b-1}{2},0}\omega_1^{\frac{a-1}{2}}\omega_2^{\frac{b-1}{2}}+\dots
    \end{align*}
    Since $\omega_1^{\frac{a-1}{2}},\,\omega_2^{\frac{b-1}{2}}\neq0$, then $\Phi_{ij}^*\left[\omega^{\frac{k-1}{2}}\right]$ is nonvanishing.\\

    Now, suppose $k$ is even, then we want to show that $\Phi_{ij}^*\left[\omega^{\frac{k}{2}-1}\right]$ and $\Phi_{ij}^*\left[\alpha\omega^{\frac{k}{2}-1}\right]$ are both nonzero. Since $k=a+b-1$ is even, without loss of generality $a$ is even and $b$ is odd. Since $a$ is even and $b$ is odd, then $\omega_1^{\frac{a}{2}-1},\,\alpha_1\omega_1^{\frac{a}{2}-1},\,\alpha_2\omega_2^{\frac{b-3}{2}},\,\omega_2^{\frac{b-1}{2}}$ are nonzero, then 
    \begin{align*}
        \Phi_{ij}^*\left[\omega^{\frac{k}{2}-1}\right]&=(\omega_1+\omega_2+(-1)^{k-j}\alpha_1\wedge\alpha_2)^{\frac{k}{2}-1}\\
        &=(\omega_1+\omega_2+(-1)^{k-j}\alpha_1\wedge\alpha_2)^{\frac{a+b-3}{2}}\\        &=\displaystyle\sum_{l_1+l_2+l_3=\frac{a+b-3}{2}}\binom{\frac{a+b-3}{2}}{l_1,l_2,l_3}\omega_1^{l_1}\omega_2^{l_2}\left((-1)^{k-j}\alpha_1\wedge\alpha_2\right)^{l_3}\\
        &=\binom{\frac{a+b-3}{2}}{\frac{a}{2}-1,\frac{b-1}{2},0}\omega_1^{\frac{a}{2}-1}\omega_2^{\frac{b-1}{2}}+\dots
    \end{align*}
    Since $\omega_1^{\frac{a}{2}-1},\,\omega_2^{\frac{b-1}{2}}\neq0$, then $\Phi_{ij}^*\left[\omega^{\frac{k}{2}-1}\right]$ is nonvanishing. Next,
    \begin{align*}              \Phi_{ij}^*\left[\alpha\omega^{\frac{k}{2}-1}\right]&=(\alpha_2+(-1)^{k-j}\alpha_1)(\omega_1+\omega_2+(-1)^{k-j}\alpha_1\wedge\alpha_2)^{\frac{k}{2}-1}\\
        &=(\alpha_2+(-1)^{k-j}\alpha_1)(\omega_1+\omega_2+(-1)^{k-j}\alpha_1\wedge\alpha_2)^{\frac{a+b-3}{2}}\\
        &=(\alpha_2+(-1)^{k-j}\alpha_1)\displaystyle\sum_{l_1+l_2+l_3=\frac{a+b-3}{2}}\binom{\frac{a+b-3}{2}}{l_1,l_2,l_3}\omega_1^{l_1}\omega_2^{l_2}\left((-1)^{k-j}\alpha_1\wedge\alpha_2\right)^{l_3}\\
        &=(\alpha_2+(-1)^{k-j}\alpha_1)\binom{\frac{a+b-3}{2}}{\frac{a}{2}-1,\frac{b-1}{2},0}\omega_1^{\frac{a}{2}-1}\omega_2^{\frac{b-1}{2}}+\dots\\
        &=\binom{\frac{a+b-3}{2}}{\frac{a}{2}-1,\frac{b-1}{2},0}\alpha_2\omega_1^{\frac{a}{2}-1}\omega_2^{\frac{b-1}{2}}+\dots
    \end{align*} 
    Since $\omega_1^{\frac{a}{2}-1},\,\alpha_2\omega_2^{\frac{b-1}{2}}\neq0$, then $\Phi_{ij}^*\left[\alpha\omega^{\frac{k}{2}-1}\right]$ is nonvanishing.\\
    
    This implies that all the forms in \eqref{eq: basis even} and  \eqref{eq: basis odd} are nonzero in $H^*(X(\sigma^{j-i}))\otimes H^*(X(\sigma^{k-j+i+1}))$ and hence nonzero in $H^*(X(\sigma^k))$. 
\end{proof}

\bibliographystyle{amsplain}

\end{document}